\newif \ifdraft
\newif\ifarxiv
\arxivtrue

\documentclass[reqno, 12pt]{article}

\ifdraft
\usepackage[12pt]{extsizes}

\usepackage[inline]{showlabels}
\usepackage[pagewise]{lineno}
\linenumbers
\fi

\usepackage[a4paper, margin=.75in]{geometry}

\usepackage{amsmath, amsthm, amssymb, mathtools}
\usepackage{thmtools}
\usepackage{enumitem}

\usepackage[dvipsnames]{xcolor}
\usepackage{hyperref}
\hypersetup{
	colorlinks = true,
	linkcolor = {BlueViolet},
	citecolor={RoyalBlue},
	urlcolor={BrickRed}
}

\makeatletter
\renewcommand{\itemautorefname}{\@gobble}
\makeatother

\theoremstyle{plain}
\newtheorem{theorem}{Theorem}[section]
\newtheorem*{theorem*}{Theorem}
\newtheorem{theoremX}{Theorem}
\renewcommand{\thetheoremX}{\Alph{theoremX}}

\newtheorem{prop}[theorem]{Proposition}
\newtheorem{corollary}[theorem]{Corollary}
\newtheorem{lemma}[theorem]{Lemma}

\theoremstyle{definition}
\newtheorem{defn}[theorem]{Definition}
\newtheorem{example}[theorem]{Example}

\newtheorem{question}[theorem]{Question}

\newtheorem*{conj*}{Conjecture}
\newtheorem{remark}[theorem]{Remark}

\RenewCommandCopy{\theHtheorem}{\thetheorem}

\RenewCommandCopy{\theHtheoremX}{\thetheoremX}

\RenewCommandCopy{\theHprop}{\theprop}

\RenewCommandCopy{\theHcorollary}{\thecorollary}

\RenewCommandCopy{\theHlemma}{\thelemma}

\RenewCommandCopy{\theHdefn}{\thedefn}

\RenewCommandCopy{\theHexample}{\theexample}

\RenewCommandCopy{\theHobs}{\theobs}

\RenewCommandCopy{\theHconj}{\theconj}

\RenewCommandCopy{\theHremark}{\theremark}

\newcounter{Hequation}

\makeatletter
\g@addto@macro\equation{\stepcounter{Hequation}}
\makeatother

\DeclareMathOperator{\codim}{codim}

\DeclareMathOperator{\im}{Im}
\DeclareMathOperator{\Int}{Int}

\title{On the Focal Locus of Submanifolds of a Finsler Manifold}

\author{
	Aritra Bhowmick\thanks{
		Kerala School of Mathematics, Kozhikode, Kerala, India.
		\href{avowmix@gmail.com}{\texttt{avowmix@gmail.com}}, 
		\href{aritra@ksom.res.in}{\texttt{aritra@ksom.res.in}}
	}
	\and
	Sachchidanand Prasad\thanks{
		School of Mathematics, Jilin University, China;
		Mathematisches Institut, G\"ottingen University, G\"ottingen, Germany
		\texttt{\href{mailto:sachchidanand.prasad1729@gmail.com}{sachchidanand.prasad1729@gmail.com}}
	}
}

\newcommand{\setsubjclass}[1]{\def\thesubjclass{#1}}
\newcommand{\setkeywords}[1]{\def\thekeywords{#1}}
\newcommand{\printclassification}{%
	\renewcommand{\thefootnote}{}%
	\footnotetext{\textbf{2020 Mathematics Subject Classification.} \thesubjclass.}%
	\footnotetext{\textbf{Keywords.} \thekeywords.}%
	\renewcommand{\thefootnote}{\arabic{footnote}}%
}

\setsubjclass{Primary: 53C22, 53B40; Secondary: 53C60}

\setkeywords{cut locus, conjugate locus, focal locus, Finsler geometry, Finsler submanifolds, Warner's regularity, Morse index theorem}

\begin{document}

\date{}
\maketitle
\printclassification

\vspace{-1em}
\begin{abstract}
	In this article, we investigate the focal locus of closed (not necessarily compact) submanifolds in a forward complete Finsler manifold. The main goal is to show that the associated normal exponential map is \emph{regular} in the sense of F.W. Warner (\textit{Am. J. of Math.}, 87, 1965). As a consequence, we show that the normal exponential is non-injective near any tangent focal point. Extending the ideas of Warner, we study the connected components of the regular focal locus. This allows us to identify an open and dense subset, on which the focal time maps are smooth, provided they are finite. We explicitly compute the derivative at a point of differentiability. As an application of the local form of the normal exponential map, following R.L. Bishop's work (\textit{Proc. Amer. Math. Soc.}, 65, 1977), we express the tangent cut locus as the closure of a certain set of points, called the separating tangent cut points. This strengthens the results from the present authors' previous work (\textit{J. Geom. Anal.}, 34, 2024).
\end{abstract}

\tableofcontents

\section{Introduction}
One of the primary aspects of Riemannian geometry is the study of \emph{geodesics}, which are \emph{locally} distance minimizing curves. A geodesic always arises as the solution to a second-order initial value problem, which lets us define the \emph{exponential map}. Given a complete Riemannian manifold, it is a smooth map from the tangent bundle to the base manifold. The singularities of the exponential map are of particular interest, as they correspond to points where a nontrivial variation of geodesics collapses infinitesimally, i.e., where the associated variational vector field (also known as \emph{Jacobi fields}) vanishes at the endpoint. For a given submanifold, one can define the normal bundle of it, and the restriction of the exponential map is then called the \emph{normal} exponential map. The \emph{focal locus} of the submanifold consists of the critical values of the normal exponential map. A closely related concept is that of the \emph{cut locus}, originally introduced by Henri Poincar\'{e} \cite{Poin05}. The cut locus of a submanifold consists of points, beyond which a globally distance-minimizing geodesic from the submanifold fails to be distance-minimizing. Both the cut locus and the focal locus have been studied extensively in the literature, see \cite{Mye35,Mye36,Kob67,Thom72,Buc77,Wol79,Sak96}.

Finsler manifolds are a natural generalization of the Riemannian ones, which were first studied by P. Finsler in his 1918 dissertation, later reprinted in \cite{Fin51}. A \emph{Finsler metric} on a manifold is a parametrized collection of Minkowski norms on each tangent space, which allows us to measure the length of a tangent vector. With this generality, we encounter certain challenges as well, primarily stemming from the fact that we cannot measure the angle between two tangent vectors, unlike the case of a Riemannian metric. Still, most of the results in Riemannian geometry can be translated to Finsler geometry, with suitable modifications. See \cite{AbaPat94,Bao2000,Shen01,Ohta2021,BhoPra2023} for a survey of results. In particular, the notions of cut and focal loci have their counterpart in a Finsler manifold.

Compared to Riemannian geometry, research on submanifolds in Finsler manifolds has been relatively limited \cite{Rund59,Ben98,MoShen01,Li11,Javaloyes2015}. One of the first hurdles to cross is that in the absence of an inner product, the suitable generalization of a normal bundle of a submanifold is no longer a vector bundle, and it is only a topological manifold (see \autoref{defn:normalCone}). In \cite{BhoPra2023}, we have systematically studied the cut locus of a submanifold in a Finsler manifold, and have shown that most of the well-known results in the Riemannian context still hold true. In the same spirit, in this article we study the focal locus of a submanifold in a Finsler manifold.\medskip

The primary goal of this article is to show that the normal exponential map associated with a submanifold in a Finsler manifold is \emph{regular} in the sense of F.W. Warner \cite{Warner1965}, as proved in \autoref{thm:normalExponentialWarnerRegular}. Along the way, we obtain \autoref{thm:regularTangentFocalPointsSubmanifold}, where we show that the set $\mathcal{F}^{\textrm{reg}}$ of \emph{regular} tangent focal points (see \autoref{defn:regularFocalLocus}) are open and dense in the set of all tangent focal points. In fact, we prove that $\mathcal{F}^{\textrm{reg}}$ is an embedded codimension $1$ submanifold of the normal bundle. Then, in \autoref{thm:kernelContainedInTangent} we show that, for any focal point of multiplicity $k \ge 2$, the connected component of $\mathcal{F}^{\textrm{reg}}$ containing it admits a foliation of dimension $k$ (see \autoref{rmk:expMapsLeavesToPoint}). Both \autoref{thm:regularTangentFocalPointsSubmanifold} and \autoref{thm:kernelContainedInTangent}, were stated without proofs in \cite{Hebda81} for submanifolds of a Riemannian manifold. For completeness, we give detailed proofs in the Finsler setup. As a consequence, in \autoref{thm:normalExponentialLocalForm}, we obtain local normal forms for the normal exponential map near certain regular tangent focal points. This culminates in the following interesting result.

\begin{theoremX}[\autoref{thm:normalExponentialNonInjective}]
    The normal exponential map of a submanifold in a (forward) complete Finsler manifold is not injective in any neighborhood of a tangent focal point.
\end{theoremX}

Let us point out why the above result is quite striking. In general, a smooth map can have singularities, and yet be globally injective. As an example, consider the map $x \mapsto x^3$ on $\mathbb{R}$, which has a singularity at the origin, and yet is globally injective; indeed, it is a homeomorphism. On the other hand, by the inverse function theorem, a non-singular map is locally injective. The above theorem can be thought of as a partial converse to this statement for the normal exponential map. This was originally proved for the exponential map $\exp_p : T_p M \rightarrow M$ at a point $p$ of an analytic Finsler manifold in \cite{MorLit32} and later for $C^\infty$ manifold in \cite{Sav43}. Warner proved the same for any smooth Riemannian or Finsler manifold, using a different method. Warner's approach led to a detailed study of the exponential map, which is interesting in itself. In recent years, similar results following Warner's approach have been proved for the \emph{sub-Riemannian} exponential maps \cite{BorKlin23,BorKlin24}. The non-injectivity of the normal exponential map near nondegenerate focal points of a submanifold in a semi-Riemannian manifold was proved in \cite{PiccionePortaluriTausk2004}, albeit via a different approach.\medskip

Building upon the ideas of \cite{Warner1965}, in \autoref{thm:regularFocalLocusComponent} we study the connected components of the regular tangent focal locus, and associate a tuple of positive integers to each component. This allows us to refine the previous density result (\autoref{thm:regularTangentFocalPointsSubmanifold}), and thereby generalize a theorem of R. L. Bishop \cite[Theorem A]{Bishop77} where only the case of first tangent focal points was considered. In particular, we show in \autoref{thm:regularJthTangentFocalLocus} that the set $\mathcal{F}^{\textrm{reg}}_j$ of regular $j^{\textrm{th}}$-tangent focal points is also open and dense in the set of $j^{\textrm{th}}$-tangent focal points (see \autoref{defn:jthTangentFocalLocus}). We also study the higher order focal time maps (\autoref{defn:focalTime}), which are seen to be continuous (\autoref{cor:focalTimeContinuous}). In \cite{Itoh2001}, assuming finiteness, these maps were shown to be locally Lipschitz continuous in the Riemannian setup, whence they are differentiable almost everywhere by Rademacher's theorem. In the same article, the authors noted that the focal time maps are differentiable at a focal point where the focal multiplicity is locally constant. We strengthen this by identifying an open dense subset of unit normal vectors, on which the focal time maps are smooth. We have the following result.
\begin{theoremX}[\autoref{thm:smoothnessOfFocalTime}]
    Given a submanifold in a forward complete Finsler manifold, if the $j^{\text{th}}$-focal time map $\lambda_j$ is finite, then $\lambda_j$ is smooth on an open and dense subset of the unit normal bundle.
\end{theoremX}
In \autoref{prop:focalTimeSmooth}, we have also computed the derivative of $\lambda_j$ explicitly at a point of differentiability as in the above theorem.\medskip

Next, we look at the \emph{first} tangent focal locus in particular, which, as the name suggests, consists of the first tangent focal locus encountered along a geodesic emanating from a submanifold. It is well-known that if a point is in the cut locus, then either it is the endpoint of at least two distinct distance minimizing geodesics, or it is a first focal locus. Furthermore, the points where two or more distance minimizing geodesics meet, called \emph{separating points} in this article, are dense in the cut locus \cite{BhoPra2023}. Both the cut points and separating points have their counterparts in the normal bundle, respectively called the \emph{tangent} cut points, and the \emph{separating} tangent cut points. In the same vein as \cite{Bishop77}, we prove the following.
\begin{theoremX}[\autoref{thm:tangentSepIsDense}]
    Given a compact submanifold of a forward complete Finsler manifold, the set of tangent cut points is the closure of the set of separating tangent cut points.
\end{theoremX}

As a corollary, we reprove part of \cite[Theorem 4.8]{BhoPra2023}, which states that the separating set is dense in the cut locus (see \autoref{rmk:cutIsClosureOfSep}).

\vspace{1em}
\noindent\textbf{Notations and Conventions.} Throughout this article, boldface symbols, e.g. $\mathbf{u},\mathbf{v},\mathbf{x},\mathbf{y}$, will always denote a tangent vector. For any $\mathbf{v} \ne 0$, the unit vector is denoted as $\widehat{\mathbf{v}} \coloneqq \frac{\mathbf{v}}{F(\mathbf{v})}$, where $F$ is the Finsler metric. The unique geodesic with initial velocity $\mathbf{v}$ is denoted by $\gamma_{\mathbf{v}}$ (\autoref{eq:geodesic}). The normal exponential map restricted to the complement of the zero section is denoted as $\mathcal{E}$, which is a smooth map (\autoref{eq:restrictedExponential}). Given any bundle $E$,  by $X \in \Gamma E$ we shall mean that $X$ is a section of $E$ defined locally over an unspecified open set of the base.

\vspace{1em}
\noindent\textbf{Organization of the article.} In \autoref{sec:preliminaries}, we recall some necessary concepts from Finsler geometry. Then, in \autoref{sec:focalLocus}, we state and prove some useful results about the focal locus of a submanifold, that are well-known in the Riemannian context. Next, in \autoref{sec:warnerRegularity} we prove the main results of this article. We have deferred some technical results involving second order tangent vectors needed in this section to the \autoref{sec:appendix}. Finally, in \autoref{sec:openQuestions} we pose a few open questions that have arisen from this work.

\section{Preliminaries on Finsler Geometry} \label{sec:preliminaries}
In this section, we collect a few definitions and results from Finsler geometry. We refer to \cite{Bao2000,Shen01,Ohta2021,Peter06,Javaloyes2014,Javaloyes2014_Correction,Javaloyes2015,Javaloyes2020,AlexAlvesJav24,BhoPra2023} as primary references.

\subsection{Finsler Metric}
\begin{defn}\label{defn:FinslerMetric}
    Let $M$ be a smooth manifold, and $TM$ denotes its tangent bundle. A \textit{Finsler metric} on $M$ is a continuous function $F: TM \to \mathbb{R}$ satisfying the following properties.
    \begin{enumerate}
        \item $F$ is smooth on $\widehat{TM} \coloneqq TM \setminus 0$.
        \item For any $p\in TM$, the restriction $F_p \coloneqq  F\big|_{T_pM}$ is a Minkowski norm, i.e.,
        \begin{itemize}
            \item for any $\lambda>0$ and $\mathbf{v}\in T_pM\setminus\{0\}$, we have $F_p(\lambda \mathbf{v})=\lambda F_p(\mathbf{v})$, and 
            \item for all $\mathbf{v}\in T_pM\setminus\{0\}$, the symmetric tensor $g_{\mathbf{v}}$ on $T_pM$, called the \emph{fundamental tensor}, is positive definite, where 
            \begin{equation}\label{eq:fundamentalTensor}
                g_{\mathbf{v}}(\mathbf{v}_1,\mathbf{v}_2) \coloneqq  \left.\dfrac{1}{2} \dfrac{\partial^2}{\partial s_1 \partial s_2} \right|_{s_1=s_2=0} \left(F_p(\mathbf{v} + s_1\mathbf{v}_1 + s_2\mathbf{v}_2)\right)^2.
            \end{equation}
        \end{itemize}
    \end{enumerate}
    $F$ is \emph{reversible} if $F(-\mathbf{v}) = F(\mathbf{v})$ holds for all $\mathbf{v} \in \widehat{TM}$.
\end{defn}

For $\mathbf{v} \in T_p M \setminus 0$, the associated \emph{Cartan tensor} on $T_p M$ is a symmetric $3$-tensor defined as
\begin{equation}\label{eq:cartanTensor}
    C_{\mathbf{v}}(\mathbf{v}_1,\mathbf{v}_2,\mathbf{v}_3) \coloneqq  \left. \frac{1}{4} \frac{\partial^3}{\partial s_1 \partial s_2 \partial s_3} \right|_{s_1=s_2=s_3=0} \left( F_p(\mathbf{v} + s_1 \mathbf{v}_1 + s_2 \mathbf{v}_2 + s_3 \mathbf{v}_3) \right)^2.
\end{equation}
For each $\mathbf{v}\in T_p M \setminus 0$ and $\mathbf{u}, \mathbf{w}\in T_p M$, we have the following relations
\begin{equation}\label{eq:cartanTensorRelation}
    C_{\mathbf{v}}(\mathbf{v}, \mathbf{u},\mathbf{w}) = C_{\mathbf{v}}(\mathbf{u}, \mathbf{v}, \mathbf{w}) = C_{\mathbf{v}}(\mathbf{u}, \mathbf{w}, \mathbf{v}) = 0.
\end{equation}
We extend the definition of the fundamental tensor and the Cartan tensor to vector fields. For any $V\in \Gamma \widehat{TM}$ and $X,Y,Z \in \Gamma TM$ defined near $p \in M$, we denote
\[g_V(X, Y)(p) \coloneqq  g_{V_p}(X_p, Y_p), \qquad C_V(X, Y, Z)(p) \coloneqq C_{V_p}\left( X_p, Y_p, Z_p \right).\]

\subsubsection{Chern Connection} Unlike the Levi-Civita connection in the Riemannian context, we do not get a canonical connection that is both torsion free and metric compatible in a suitable sense. In this article, we consider the Chern connection, which is a family of torsionless connections on a Finsler manifold.

\begin{defn}\label{defn:chernConnection}\cite{Rademacher2004, Javaloyes2014}
    For each $V \in \Gamma \widehat{TM}$, we have a unique affine connection \[\nabla^V: \Gamma TM \otimes \Gamma TM \rightarrow \Gamma TM,\] called the \emph{Chern connection}, satisfying the following conditions for any $X, Y, Z \in \Gamma TM$.
    \begin{itemize}
        \item (Torsion freeness) $\nabla^V_X Y - \nabla^V_Y X = [X, Y]$.
        \item (Almost metric compatibility) $X (g_V(Y,Z)) = g_V(\nabla^V_X Y, Z) + g_V(Y, \nabla^V_X Z) + 2 C_V(\nabla^V_X V, Y, Z)$.
    \end{itemize}
\end{defn}

The value of $\nabla^V_X Y|_p$ depends on the values of $V$ and $X$ only at $p$, and consequently $\nabla^{\mathbf{v}}_{\mathbf{x}} Y|_p$ is well-defined for any vector $\mathbf{v}, \mathbf{x} \in T_p M$ with $\mathbf{v} \ne 0$, and for any vector field $Y$ defined near $p$. In particular, $\nabla^V$ can be recovered from the Christoffel symbols, as considered in \cite{Bao2000}. Similarly, given a curve $\gamma : [a, b] \rightarrow M$ and $V \in \Gamma \gamma^*\widehat{TM}$, one can define $\nabla^V_X Y \in \Gamma \gamma^*TM$ for any $X, Y \in \Gamma \gamma^*TM$. This leads to defining a covariant derivative along a $\gamma$.
\begin{defn}\label{defn:covariantDerivative}\cite{Javaloyes2014}
    Given a curve $\gamma : [a,b] \rightarrow  M$ and $W \in \Gamma \gamma^*\widehat{TM}$, the \emph{covariant derivative} along $\gamma$ is defined as \[D^W_\gamma : \Gamma \gamma^* TM \rightarrow \Gamma \gamma^* TM,\]
    which satisfies the following.
    \begin{itemize}
        \item For any $X, Y \in \Gamma \gamma^*TM$ and scalars $\alpha, \beta \in \mathbb{R}$ we have 
        \[D^W_\gamma(\alpha X + \beta Y) = \alpha D^W_\gamma X + \beta D^W_\gamma Y.\]

        \item For any $X \in \Gamma \gamma^*TM$ and a smooth function $f : [a,b] \rightarrow \mathbb{R}$, we have 
        \[D^W_\gamma(f X) = \frac{df}{dt} X + f D^W_\gamma X.\]
        
        \item For any $X, Y \in \Gamma \gamma^* TM$, we have \[\frac{d}{dt}g_W(X,Y) = g_W \left( D^W_{\gamma} X, Y \right) + g_W \left( X, D^W_\gamma Y \right) + 2 C_W \left( D^W_\gamma W, X, Y \right).\]
    \end{itemize}
\end{defn}

If for some curve $\gamma$ we have $\dot\gamma(t) \ne 0$ for all time, then we shall use the notations 
\[\dot X \coloneqq D^{\dot \gamma}_\gamma X, \quad \ddot X \coloneqq D^{\dot g}_\gamma \dot X = D^{\dot \gamma}_\gamma D^{\dot \gamma}_\gamma X, \qquad \text{for any $X \in \Gamma \gamma^*TM$},\]
provided the curve $\gamma$ is understood from the context. A vector field $X \in \Gamma \gamma^*TM$ is said to be \emph{parallel} with respect to $\gamma$ if $\dot X = 0$. In particular, $\gamma$ itself is called parallel if $\ddot \gamma = 0$, which are precisely the \emph{geodesics}.

\subsection{Geodesics and Jacobi Fields}
Let us denote the space of piecewise smooth paths $\gamma : [a,b] \rightarrow M$ as $\mathcal{P} = \mathcal{P}([a,b])$. We have two functionals defined on this path space, namely the length and the energy functionals.
\[\begin{aligned}
    L : \mathcal{P} &\longrightarrow \mathbb{R} \\
    \gamma &\longmapsto \int_a^b F(\dot \gamma(t)) dt
\end{aligned} \qquad , \qquad \begin{aligned}
    E : \mathcal{P} &\longrightarrow \mathbb{R} \\
    \gamma &\longmapsto \frac{1}{2}\int_a^b F(\dot \gamma(t))^2 dt.
\end{aligned}\]
A curve $\gamma \in \mathcal{P}$ is a \emph{geodesic} if it is a critical point of the energy functional with respect to proper variations. Recall that given a piecewise smooth vector field $W \in \Gamma \gamma^* TM$, a \emph{$W$-variation} of $\gamma$ is a piecewise smooth map $\Lambda : (-\epsilon, \epsilon) \times [a,b] \rightarrow M$ satisfying \[\Lambda(0,t) = \gamma(t), \quad \left.\frac{\partial}{\partial s}\right|_{s=0} \Lambda(s, t) = W(t),\]
for all $(s,t)$ in the domain. The variation is \emph{proper} if $W(a) = 0 = W(b)$, or equivalently, if $\Lambda(s, a) = \gamma(a)$ and $\Lambda(s, b) = \gamma(b)$ for all $s$. Geodesics are \emph{locally} distance minimizing, where we have the Finsler distance between $p, q \in M$ defined as
    \[d(p, q) = \inf \left\{ L(\gamma) \;\middle|\; \gamma \in \mathcal{P}, \; \gamma(a) = p, \; \gamma(b) = q \right\}.\]
In general, geodesics fail to be globally distance minimizing. We say a geodesic $\gamma$ is a \emph{global distance minimizer} (or simply a \emph{minimizer}) if $\gamma$ is unit-speed, and $d(\gamma(a), \gamma(b)) = b - a = L(\gamma)$.

Geodesics are precisely the solutions to an initial value problem known as the geodesic equation (e.g., \cite[Equation 3.15]{Ohta2021}), which can be written succinctly as $\ddot \gamma = 0$ with the notation introduced earlier. As such, geodesics are always smooth, and given a vector $\mathbf{v} \in T_p M$, we have a unique maximal geodesic $\gamma_{\mathbf{v}} : [0,\ell] \rightarrow  M$ satisfying 
\begin{equation}\label{eq:geodesic}
    \gamma_{\mathbf{v}}(0) = p, \qquad \dot \gamma_{\mathbf{v}}(0) = \mathbf{v}.
\end{equation}
Note that, unlike Riemannian geometry, due to the asymmetry of the Finsler metric, the reversed curve $\bar{\gamma}:[0, \ell] \rightarrow M$ defined by $\bar{\gamma}(t) = \gamma(\ell - t)$ need not be a geodesic. Nevertheless, $\bar{\gamma}$ is a geodesic with initial velocity $-\dot\gamma(\ell)$ for the \emph{reverse Finsler metric} $\bar{F}$ defined by $\bar{F}(\mathbf{v}) = F(-\mathbf{v})$ for all $\mathbf{v} \in TM$ \cite[Section 2.5]{Ohta2021}.

A Finsler manifold $(M, F)$ is said to be \emph{forward complete} if for all $\mathbf{v} \in TM$, the geodesic $\gamma_{\mathbf{v}}$ is defined for all time $[0, \infty)$. We say $(M, F)$ is \emph{backward complete} if $(M, \bar{F})$ is forward complete. By the Hopf-Rinow theorem \cite[Theorem 3.21]{Ohta2021}, if a Finsler manifold $(M, F)$ is either forward or backward complete, then given any two points $p, q \in M$ there exists a minimizer with respect to $F$, and another (possibly distinct) minimizer with respect to $\bar{F}$ joining $p$ to $q$. Throughout this article, we shall assume $(M, F)$ to be a forward complete Finsler manifold, unless otherwise stated.

\begin{defn}\label{defn:exponentialMap}
    The \emph{exponential map} $\exp : TM \rightarrow  M$ at a point $p \in M$ is defined as $\exp_p(\mathbf{v}) = \gamma_{\mathbf{v}}(1)$ for any $\mathbf{v} \in T_p M$, where $\gamma_{\mathbf{v}} : [0,\infty) \rightarrow M$ is the unique geodesic starting at $p$ with initial velocity $\mathbf{v}$.
\end{defn}

It follows from the theory of ordinary differential equations that the exponential map is smooth on $\widehat{TM}$, but only $C^1$ on the whole tangent bundle, see \cite[Section 5.3]{Bao2000} for details.

\subsubsection{Jacobi Fields}
Given a geodesic $\gamma : [a,b] \rightarrow M$, consider a \emph{geodesic variation} $\Lambda : (-\epsilon, \epsilon) \times [a,b] \rightarrow M$ of $\gamma$, that is, for each $-\epsilon < s < \epsilon$ we require the curve $\Lambda_s : [a,b] \rightarrow M$ given by $\Lambda_s(t) \coloneqq \Lambda(s, t)$ to be a geodesic.

\begin{defn}\label{defn:JacobiField}
    Let $\gamma : [a,b] \rightarrow M$ be a geodesic. A vector field $J \in \Gamma \gamma^* TM$ along $\gamma$ is called a \emph{Jacobi field} along $\gamma$ if there exists a geodesic variation $\Lambda : (-\epsilon, \epsilon) \times [a,b] \rightarrow M$ of $\gamma$ satisfying 
    \[\left. \frac{\partial}{\partial s} \right|_{s=0} \Lambda(s,t) = J(t), \quad t\in [a,b].\]
\end{defn}

In order to characterize Jacobi fields as the solution to the Jacobi equation, recall first the \emph{curvature tensor} associated with the Chern connection given as
\begin{equation}\label{eq:curvatureTensor}
    R^V(X,Y) Z \coloneqq \nabla^V_X \nabla^V_Y Z - \nabla^V_Y \nabla^V_X Z - \nabla^V_{[X,Y]} Z, \quad X,Y,Z \in \Gamma TM, \, V \in \Gamma \widehat{TM}.
\end{equation}
In general, $R^V|_p$ depends on the value of $V$ on a neighborhood of $p$. To remedy this, consider the \emph{vertical derivative} of $\nabla$ given as 
\begin{equation}\label{eq:verticalDerivative}
    P_V(X, Y, Z) = \left. \frac{\partial}{\partial t}\right|_{t=0} \nabla^{V + t Z}_X Y, \quad X, Y, Z \in \Gamma TM, \, V \in \Gamma \widehat{TM}.
\end{equation}
Then, the \emph{Chern curvature tensor} is defined as \cite[Equation 5.2]{AlexAlvesJav24}
\begin{equation}\label{eq:chernCurvature}
    R_V(X, Y) Z \coloneqq R^V(X, Y)Z - P_V \left( Y, Z, \nabla^V_X V \right) + P_{V} \left( X, Z, \nabla^V_Y V \right), \quad X, Y, Z \in \Gamma TM, \, V \in \Gamma \widehat{TM}.
\end{equation}
It follows that $R_V|_p$ depends on the value of $V$ at $p$ only.
\begin{prop}\label{prop:jacobiEquation} \cite[Proposition 2.11]{Javaloyes2020}
    Given a geodesic $\gamma : [a, b] \rightarrow M$, a vector field $J \in \Gamma \gamma^*TM$ is a Jacobi field if and only if it satisfies the \emph{Jacobi equation} 
    \begin{equation}\label{eq:JacobiEquation}
        D^{\dot\gamma}_\gamma D^{\dot\gamma}_\gamma J - R_{\dot \gamma}\left( \dot \gamma, J \right) \dot \gamma = 0.
    \end{equation}
\end{prop}

Since the Jacobi equation is a second order ODE, given the initial data, $\mathbf{u}, \mathbf{v} \in T_{\gamma(a)}M$, there exists a unique Jacobi field $J$ along $\gamma$ satisfying, $J(a) = \mathbf{u}$ and $\dot{J}(a) = D^{\dot\gamma}_\gamma J(a) = \mathbf{v}$. In particular, the collection of all Jacobi fields along $\gamma$ forms a vector space of dimension $2 \dim M$. Let us make an important observation regarding the zeros of a Jacobi field.

\begin{lemma}\label{lemma:zerosOfJacobiField}
	Given a non-zero Jacobi field $J$ along a geodesic $\gamma : [0, \ell] \rightarrow M$, if $J(t_0) = 0$ for some $0 \le t_0 \le \ell$, then $\dot{J}(t_0) \ne 0$. Moreover, the zeros of $J$ are isolated.
\end{lemma}
\begin{proof}
	\ifarxiv
    Suppose $J(t_0) = 0$ for some $0 \le t_0 \le \ell$. If possible, suppose $\dot{J}(t_0) = 0$. If $t_0 = 0$, then we must have $J \equiv 0$, by the uniqueness of Jacobi fields. Suppose $t_0 > 0$. 	We have geodesic variation $\Lambda : (-\epsilon, \epsilon) \times [0, \ell] \rightarrow M$ such that $J(t) = \frac{\partial}{\partial s} |_{s=0} \Lambda(s, t)$. Define $\bar{\Lambda} : (-\epsilon, \epsilon) \times [0, t_0] \rightarrow M$ by $\bar{\Lambda}(s, t) = \bar{\Lambda}(s, t_0 - t)$. Then, $\bar{\Lambda}$ is a geodesic variation with respect to the reverse Finsler metric $\bar{F}$. In particular, $\bar{J}(t) \coloneqq \frac{\partial}{\partial s} |_{s=0} \bar{\Lambda}(s, t) = -J(t)$ is a Jacobi field along the geodesic $\bar{\gamma}(t) = \gamma(t_0 - t)$ defined on $[0, t_0]$, with respect to $\bar{F}$. By our assumption, 
	\[0 = \left.\frac{D^\gamma}{dt}\right|_{t=t_0} J(t) = \left.\frac{\partial^2}{\partial t\partial s}\right|_{(t_0,0)} \Lambda(s, t) = - \left.\frac{\partial^2}{\partial t\partial s}\right|_{(0,0)} \bar{\Lambda}(s, t) = - \left.\frac{D^{\bar{\gamma}}}{dt}\right|_{t = 0} \bar{J}(t).\]
	But then $\bar{J} \equiv 0$ along $\bar{\gamma}$, as the initial conditions are $\bar{J}(0) = 0$ and $\frac{D^{\bar{\gamma}}}{dt}|_{t=0} \bar{J}(t) = 0$. Consequently, we have $J \equiv 0$ along $\gamma$, a contradiction. Hence, $J(t_0) = 0$ implies $\dot{J}(t_0) \ne 0$.
    \fi

	Next, let us fix some frame of parallel vector fields $\left\{ e_1(t), \dots, e_n(t) \right\}$ along $\gamma$, near $t_0$, where $n = \dim M$. We can write, $J = \sum_{i = 1}^n J^i(t) e_i (t)$ near $t_0$. Since $\dot{e}_i = 0$, we have $0 \ne \dot{J}(t_0) = \sum \dot J^i(t_0) e_i(t_0)$. Hence, for some $1 \le i_0 \le n$ we get $J^{i_0}(t_0) = 0$ but $\dot J^{i_0}(t_0) \ne 0$. Consequently, $J^{i_0}$ is locally injective near $t_0$, and in particular, $J^{i_0}$ is non-zero in a deleted neighborhood of $t_0$. It follows that zeros of $J$ are isolated.
\end{proof}

We now have the following useful result (compare \cite[Lemma 2.3]{Warner1965}).

\begin{lemma}\label{lemma:jacobiFieldMultipleOfNonZeroField}
    Let $J$ be a non-zero Jacobi field defined along a geodesic $\gamma : [a, b] \rightarrow M$. Suppose $J(t_0) = 0$ for some $a < t_0 < b$. Then, there exists a smooth vector field $Y(t)$ along $\gamma$, such that $J(t) = (t- t_0)Y(t)$, and $Y \ne 0$ in a neighborhood of $t_0$. In particular, $\dot J(t_0) = Y(t_0)$.
\end{lemma}
\begin{proof}
    Since $J$ is nonzero, and $J(t_0) = 0$, by \autoref{lemma:zerosOfJacobiField} we have $\dot J(t_0) \ne 0$. Let us consider some frame of parallel vector fields $E_1, \dots , E_{n}$ along $\gamma$. We can then write $J(t) = \sum f_i(t) E_i(t)$ for some smooth functions $f_i$ along $\gamma$. As $J(t_0) = 0 \Rightarrow f_i(t_0) = 0$ for all $i = 1,\dots ,n$, we can uniquely write $f_i(t) = (t - t_0) \tilde{f}_i(t)$ for some smooth functions $\tilde{f}_i$ along $\gamma$. Indeed, considering $g_i(s) = f_i\left( (t- t_0)s + t_0 \right)$, we have 
    \[f_{i}(t) = g_i(1) - g_i(0) = \int_0^1 g_i^\prime(s)ds = (t - t_0)\underbrace{\int_0^1 f_i^\prime\left( (t-t_0)s + t_0 \right) ds}_{\tilde{f}_i}.\]
    Consequently, we have $J(t) = (t - t_0) Y(t)$, where $Y = \sum \tilde{f}_i E_i$ is a smooth vector field along $\gamma$. As $\dot E_i = 0$, we have $\dot J(t) = \sum f_i^\prime(t) E_i(t)$. Since $\dot J(t_0) \ne 0$, there is some $1 \le i_0 \le k$ so that 
    \[f_{i_0}^\prime(t_0) \ne 0 \Rightarrow \tilde{f}_{i_0}(t_0) \ne 0 \Rightarrow Y(t_0) \ne 0.\]
    Thus, $Y$ is nonvanishing near $t_0$.
\end{proof}

\subsection{Submanifolds in Finsler Manifolds}
Given a submanifold $N$ of a Finsler manifold $(M, F)$, we consider the normal cone bundle as the natural replacement for normal bundles.

\begin{defn}\label{defn:normalCone}
    Given a submanifold $N \subset M$, the set $$\nu_p = \nu_p(N) = \big\{ \mathbf{v} \in T_p M \setminus \left\{ 0 \right\} \;\big|\; g_{\mathbf{v}}(\mathbf{v},\mathbf{w}) = 0 \; \forall \mathbf{w} \in T_p N\big\} \cup \left\{ 0 \right\}$$
    is called the \emph{normal cone} of $N$ at $p \in N$. The set $\nu = \nu(N) = \cup_{p \in N} \nu_p(N)$ is called the \emph{normal cone bundle} of $N$. The \emph{unit normal cone bundle} of $N$ is denoted as $S(\nu) = \cup_{p \in N} S(\nu_p)$, where $S(\nu_p) \coloneqq \left\{ \mathbf{v} \in \nu_p \;\middle|\; F_p(\mathbf{v}) = 1 \right\}$.
\end{defn}

It should be noted that $\nu(N)$ is \emph{not} a vector bundle in general; in fact, it is a cone bundle. That is to say for any $0 \ne \mathbf{v} \in \nu$ we have $\lambda \mathbf{v} \in \nu$ for all $\lambda \ge 0$. We shall denote $\hat{\nu}_p \coloneqq \nu_p \setminus \{ 0 \}$, and $\hat{\nu} \coloneqq \cup_{p\in N} \hat{\nu}_p$ is then called the \emph{slit cone bundle}. It follows that $\hat{\nu}$ (resp. $S(\nu)$) is a smooth submanifold of $TM$ of dimension $\dim N + \codim N = \dim M$ (resp. $\dim M - 1$), whereas $\nu$ is only a topological submanifold. In \autoref{prop:tangentSpaceOfNormalConeBundle}, we shall see a natural way to identify the tangent space of $\hat{\nu}$ at some $\mathbf{v}$.

\begin{defn}\label{defn:normalExponentialMap}
    Given $N \subset M$, the \emph{normal exponential map} $\exp^\nu: \nu(N) \rightarrow  M$ is defined as the restriction of the exponential map to the cone bundle $\nu(N)$. We shall denote the restriction to the smooth submanifold $\hat{\nu} \subset TM$ as 
    \begin{equation}\label{eq:restrictedExponential}
        \mathcal{E} \coloneqq \exp^\nu|_{\hat{\nu}} : \hat{\nu} \rightarrow M,
    \end{equation}
    which is then a smooth map.
\end{defn}

Recall the \emph{Legendre transformation} $\mathcal{L} : TM \rightarrow T^*M$ \cite[Section 3.2]{Ohta2021}, which is a fiber-wise homeomorphism, restricting to a $C^\infty$-diffeomorphism in the complement of the zero sections. For any $\mathbf{v}, \mathbf{w} \in TM$ we have
\begin{equation}\label{eq:legendreTransformation}
    \mathcal{L}(\mathbf{v})(\mathbf{w}) = 
    \begin{cases}
        g_{\mathbf{v}}\left( \mathbf{v}, \mathbf{w} \right), \qquad \mathbf{v} \ne 0, \\
        0, \qquad \mathbf{v} = 0.
    \end{cases}
\end{equation}
It is easily seen that $\mathcal{L}$ maps $\nu$ bijectively onto the annihilator bundle of $TN$. Thus, the cone bundle $\nu$ is fiber-wise (non-linearly) homeomorphic to a vector bundle, and $\hat{\nu}$ is $C^\infty$-diffeomorphic to the complement of the zero section in the annihilator bundle. As a consequence, for any $\mathbf{n} \in \hat{\nu}$, we can define a smooth local extension $\tilde{\mathbf{n}} \in \Gamma \hat{\nu}$ of $\mathbf{n}$ via $\mathcal{L}$.

Given a $0 \ne \mathbf{n} \in \nu_p$ we have a direct sum decomposition of $T_p N$ defined using the fundamental tensor $g_{\mathbf{n}}$ which is nondegenerate. In particular, for any $\mathbf{v} \in T_p M$ with $p \in N$, we can uniquely write 
\begin{equation}\label{eq:directSum}
    \mathbf{v} = \mathbf{v}^{\top_{\mathbf{n}}} + \mathbf{v}^{\perp_{\mathbf{n}}} \in T_p N \oplus \left( T_p N \right)^{\perp_{g_{\mathbf{n}}}}.
\end{equation}
We then have the \emph{second fundamental form} of $N$ in the direction of $\mathbf{n}$ defined as the symmetric tensor
\begin{equation}\label{eq:secondFundamentalForm}
    \begin{aligned} 
        \Pi^{\mathbf{n}} : T_p N \odot T_p N &\longrightarrow \left( T_p N \right)^{\perp_{g_{\mathbf{n}}}} \\
        \left( \mathbf{x}, \mathbf{y} \right) &\longmapsto  - \left( \nabla^{\tilde{\mathbf{n}}}_X Y |_{p} \right)^{\perp_{\mathbf{n}}},
    \end{aligned}
\end{equation}
where $X, Y \in \Gamma TN$ and $\tilde{\mathbf{n}} \in \Gamma \hat{\nu}$ are  arbitrary local extensions of $\mathbf{x}, \mathbf{y}, \mathbf{n}$ respectively. The map is well-defined, as the assignment can be seen to be $C^\infty(N)$-linear. Furthermore, as the Chern connection is torsion-free, we have 
\[\Pi^{\mathbf{n}}\left( \mathbf{x}, \mathbf{y} \right) - \Pi^{\mathbf{n}}\left( \mathbf{x}, \mathbf{y} \right) = \left( \nabla^{\mathbf{n}}_X Y|_p - \nabla^{\mathbf{n}}_Y X|_p \right)^{\perp_{\mathbf{n}}} = \left( [X, Y]_p \right)^{\perp_{\mathbf{n}}} = 0.\]
In other words, $\Pi^{\mathbf{n}}$ is a symmetric linear map. Taking adjoint with respect to $g_{\mathbf{n}}$, we define the \emph{shape operator} $A_{\mathbf{n}} : T_p N \rightarrow T_p N$ of $N$ along $\mathbf{n}$ via the equation
\begin{equation}\label{eq:shapeOperator}
    g_{\mathbf{n}}\left( A_{\mathbf{n}} \mathbf{x}, \mathbf{y} \right) = g_{\mathbf{n}}\left( \mathbf{n}, \Pi^{\mathbf{n}}\left( \mathbf{x}, \mathbf{y} \right) \right), \quad \mathbf{x}, \mathbf{y} \in T_p N.
\end{equation}
\begin{lemma}\label{lemma:shapeOperatorFormula}
    For any $0 \ne \mathbf{n} \in \nu_p$ and $\mathbf{x} \in T_p N$, we have 
    \[A_{\mathbf{n}}(\mathbf{x}) = \left( \nabla^{\tilde{\mathbf{n}}}_X \tilde{\mathbf{n}}|_p \right)^{\top_{\mathbf{n}}},\]
    where $X \in \Gamma TN, \tilde{\mathbf{n}} \in \Gamma \hat{\nu}$ are arbitrary local extensions of $\mathbf{x}, \mathbf{n}$ respectively.
\end{lemma}
\begin{proof}
    For any $\mathbf{y} \in T_p N$ get an extension $Y \in \Gamma TM$ so that $g_{\tilde{\mathbf{n}}}\left( \tilde{\mathbf{n}}, Y \right) = 0$ on points of $N$. As $X \in \Gamma TN$, on points of $N$ we then have
    \[0 = X\left( g_{\tilde{\mathbf{n}}} \left( \tilde{\mathbf{n}}, Y \right) \right) = g_{\tilde{\mathbf{n}}}\left( \nabla^{\tilde{\mathbf{n}}}_X \tilde{\mathbf{n}}, Y \right) + g_{\tilde{\mathbf{n}}}\left( \tilde{\mathbf{n}}, \nabla^{\tilde{\mathbf{n}}}_X Y \right) + 2C_{\tilde{\mathbf{n}}}\left( \nabla^{\tilde{\mathbf{n}}}_X \tilde{\mathbf{n}}, \tilde{\mathbf{n}}, Y \right).\]
    The last term vanishes by \autoref{eq:cartanTensorRelation}. Hence, evaluating at $p$ we have
    \begin{align*}
        g_{\mathbf{n}}\left( A_{\mathbf{n}} \mathbf{x}, \mathbf{y} \right) 
        &= g_{\mathbf{n}} \left( \mathbf{n}, \Pi^{\mathbf{n}}\left( \mathbf{x}, \mathbf{y} \right) \right) \\
        &= g_{\mathbf{n}} \left( \mathbf{n}, - \left( \nabla^{\tilde{\mathbf{n}}}_X Y|_p \right)^{\perp_{\mathbf{n}}} \right) \\
        &= g_{\mathbf{n}} \left( \mathbf{n}, -\nabla^{\tilde{\mathbf{n}}}_X Y|_p \right), \qquad \text{as } g_{\mathbf{n}}\left( \mathbf{n}, -\left( \nabla^{\tilde{\mathbf{n}}}_X Y|_p \right)^{\top_{\mathbf{n}}} \right) = 0 \\
        &= g_{\mathbf{n}}\left( \nabla^{\tilde{\mathbf{n}}}_X \tilde{\mathbf{n}}|_p, \mathbf{y} \right) \\
        &= g_{\mathbf{n}}\left( \left( \nabla^{\tilde{\mathbf{n}}}_X \tilde{\mathbf{n}}|_p \right)^{\top_{\mathbf{n}}}, \mathbf{y} \right), \qquad \text{as } g_{\mathbf{n}}\left( \left( \nabla^{\tilde{\mathbf{n}}}_X \tilde{\mathbf{n}}|_p \right)^{\perp_{\mathbf{n}}}, \mathbf{y} \right) = 0.
    \end{align*}
    Since $\mathbf{y} \in T_p N$ is arbitrary and $g_{\mathbf{n}}|_{T_p N}$ is nondegenerate, we have the claim.
\end{proof}

\subsection{\texorpdfstring{$N$}{N}-Geodesics and \texorpdfstring{$N$}{N}-Jacobi Fields}
Given a submanifold $N \subset M$, we have the subspace of piecewise smooth paths \[\mathcal{P}_N = \mathcal{P}_N([a,b]) = \{ \gamma \in \mathcal{P}([a,b]) \;|\; \gamma(a) \in N\}\] starting at $N$. For any $\gamma\in\mathcal{P}_N$ we say $\gamma$ is a curve joining $N$ to $\gamma(b)$. Given $q \in M$, the distance from $N$ to $q$ is then defined as
    \[d(N,q) \coloneqq \inf \left\{ L(\gamma) \;\middle|\; \gamma\in\mathcal{P}_N, \; \gamma(b) = q \right\}.\]
We identify the tangent space of $\mathcal{P}_N$ at a given $\gamma\in \mathcal{P}_N$ as the infinite dimensional vector space consisting of piecewise smooth vector fields $W \in \Gamma \gamma^* TM$ with $W(a) \in T_{\gamma(a)} N$. Given $W \in T_\gamma\mathcal{P}_N$, a \emph{$W$-variation} is then a piecewise smooth map $\Lambda : (-\epsilon,\epsilon) \times [a,b] \rightarrow  M$ satisfying
\[\Lambda(0, t) = \gamma(t), \quad  \left. \frac{\partial}{\partial s} \right|_{s=0} \Lambda(s, t) = W(t), \quad  \Lambda(s, a) \in N,\]
for all $(s,t)$ in the domain of the definition. The variation is called \emph{proper}, if furthermore, we have $W(b) = 0$, or equivalently if $\Lambda(s, b) = \gamma(b)$ for all $s$.

\begin{defn}\label{defn:NSegment}
    A piecewise $C^1$ curve $\gamma : [a,b] \rightarrow M$ in $\mathcal{P}(N)$ is called an \emph{$N$-geodesic} if $\gamma$ is a critical point of the restricted energy functional $E|_{\mathcal{P}(N)}$ with respect to proper variations. An $N$-geodesic $\gamma$ is called an \emph{$N$-segment} joining $N$ to $\gamma(b)$ if $\gamma$ is unit-speed, and $d(N, \gamma(b)) = b - a = L(\gamma)$.
\end{defn}

We recall the following useful lemma, which can be compared to \cite[Theorem 5.16, pg. 24]{Bus55}. 

\begin{lemma}\label{lemma:convergentSubsequenceNSegment} \cite[Lemma 3.9]{BhoPra2023}
    Suppose, $\gamma_i : [0, \ell_i] \rightarrow M$ are unit-speed minimizers joining $p_i = \gamma_i(0)$ to $q_i = \gamma_i(\ell_i)$, where $\ell_i = d(p_i, q_i)$. Suppose $\ell_i \rightarrow \ell$. If either
    \begin{center}
        (a) $p_i \rightarrow p$ and $F$ is forward complete, \kern.5cm or \kern.5cm (b) $q_i \rightarrow q$ and $F$ is backward complete,
    \end{center}
    then a subsequence of $\gamma_i$ converges uniformly to a minimizer $\gamma$, which satisfies $L(\gamma) = \ell$.
\end{lemma}

As an immediate consequence, we get the following.

\begin{prop}\cite[Proposition 3.10]{BhoPra2023}\label{prop:existenceOfNSegements}
    Suppose $(M, F)$ is a forward complete Finsler manifold, and $N$ is a closed submanifold of $M$. Let $\gamma_i : [0, \ell_i] \rightarrow M$ be $N$-segments joining $N$ to $q_i \coloneqq  \gamma_i(\ell_i)$. Suppose further that either
    \begin{equation}\label{eq:hypothesisH}
        \tag{$\mathsf{H}$}
        \text{(a) $N$ is compact, \kern.5cm or \kern.5cm (b) $F$ is backward complete as well.}
    \end{equation}
    If $\ell_i \rightarrow \ell$ and $q_i \rightarrow q$, then there exists an $N$-segment $\gamma : [0, \ell]\rightarrow M$ joining $N$ to $q$, with a subsequence $\gamma_i \rightarrow \gamma$. In particular, for any $q \in M$ there exists an $N$-segment $\gamma$ joining $N$ to $q$, with $d(N, q) = L(\gamma)$.
\end{prop}

We would like to point out that the hypothesis (\hyperref[eq:hypothesisH]{$\mathsf{H}$}) is automatically satisfied for a closed (not necessarily compact) submanifold of a complete \emph{Riemannian} manifold, since the notion of forward and backward completeness coincides there. On the other hand, in a forward but not backward complete Finsler manifold, the distance from a non-compact closed submanifold to a point may not even be achieved on the submanifold, see \cite[Remark 3.11]{BhoPra2023} for an example. Thus, while dealing with cut locus of a submanifold (\autoref{defn:cutLocusN}), it is natural to assume the hypothesis (\hyperref[eq:hypothesisH]{H}).

\subsubsection{\texorpdfstring{$N$}{N}-Jacobi Fields}
Given a unit-speed $N$-geodesic $\gamma : [a,b]\rightarrow M$, a variation $\Lambda : (-\epsilon, \epsilon) \times [a,b] \rightarrow M$ is said to be an \emph{$N$-geodesic variation} if for each $-\epsilon < s < \epsilon$, the curve $\Lambda_s : [a,b] \rightarrow M$ is an $N$-geodesic.

\begin{defn}\label{defn:NJacobiField}
    Let $\gamma : [a,b] \rightarrow M$ be a unit-speed $N$-geodesic. A vector field $J \in \Gamma\gamma^*TM$ is called an $N$-Jacobi field of $\gamma$ if there exists an $N$-geodesic variation $\Lambda:(-\epsilon,\epsilon)\times [a,b]\rightarrow M$ satisfying 
    \[\left.\frac{\partial}{\partial s}\right|_{s=0} \Lambda(s,t) = J(t), \quad t\in[a,b].\]
\end{defn}

We have the following characterization of an $N$-Jacobi field via the Jacobi equation.

\begin{prop}\label{prop:NJacobiEquation}\cite[Proposition 3.5]{AlexAlvesJav24}
    Given an $N$-geodesic $\gamma:[a,b]\rightarrow M$ with initial velocity $\mathbf{n} = \dot \gamma(a)$, a vector field $J \in \Gamma\gamma^*TM$ is an $N$-Jacobi field if and only if it satisfies the following initial value problem
    \begin{equation}\label{eq:NJacobiEquation}
        D^{\dot\gamma}_\gamma D^{\dot\gamma}_\gamma J - R_{\dot\gamma}(\dot\gamma, J) \dot\gamma = 0, \quad J(a)\in T_{\gamma(a)} N, \quad D^{\dot\gamma}_\gamma J(a) - A_{\mathbf{n}} \left( J(a) \right) \in \left( T_p N \right)^{\perp_{g_{\mathbf{n}}}}.
    \end{equation}
\end{prop}

\begin{example}\label{example:jacobiField}
    Given an $N$-geodesic $\gamma : [a,b]\rightarrow M$, consider the vector field $J(t) = t \dot\gamma(t)$ along $\gamma$. Then, $D^{\dot\gamma}_\gamma J = \dot \gamma + t \ddot \gamma = \dot \gamma \Rightarrow D^{\dot\gamma}_\gamma D^{\dot\gamma}_\gamma J = \ddot \gamma = 0$. On the other hand, $R_{\dot \gamma}(\dot \gamma, J)\dot \gamma = t R_{\dot \gamma}(\dot \gamma, \dot \gamma) \dot \gamma = 0$ by the anti-symmetry of the tensor. Thus, $J$ satisfies \autoref{eq:JacobiEquation}. Also, $J(a) = 0 \in T_{\gamma(a)}N$ and $D^{\dot\gamma}_\gamma J(a) = \dot \gamma(a) \in \nu_{\gamma(a)}$. Since, $A_{\dot \gamma(a)}\left( J(a) \right) = 0$, it follows that $J$ satisfies \autoref{eq:NJacobiEquation}, i.e., $J$ is an $N$-Jacobi field along $\gamma$. Note that $J(b) \ne 0$.
\end{example}

As an $N$-Jacobi field is itself a Jacobi field, by \autoref{lemma:zerosOfJacobiField}, zeros of $N$-Jacobi fields are isolated as well. We shall need the following result.
\begin{lemma}\label{lemma:jacobiFieldAdjoint}
    Given an $N$-geodesic $\gamma : [0, \ell] \rightarrow M$ and two $N$-Jacobi fields $J, K$ along $\gamma$, we have $g_{\dot \gamma}(J, \dot K) = g_{\dot \gamma}(\dot J, K)$.
\end{lemma}
\begin{proof}
    Since $\gamma$ is a geodesic, we have $D^{\dot \gamma}_\gamma \dot \gamma = 0$. Hence, for any two vector fields $X, Y \in \Gamma \gamma^*TM$ we have 
    \[\frac{d}{dt} g_{\dot \gamma}\left( X, Y \right) = g_{\dot \gamma} ( \dot X, Y ) + g_{\dot \gamma}( X, \dot Y ) + 2 \underbrace{C_{\dot \gamma}\left( D^{\dot \gamma}_\gamma \dot \gamma, X, Y  \right)}_0.\]
    Also, using \autoref{eq:cartanTensorRelation}, we get from \cite[Proposition 3.1]{Javaloyes2020} the symmetry 
    \[g_{\dot \gamma} \left( R_{\dot \gamma}\left( \dot \gamma, U \right) \dot \gamma, V \right) = g_{\dot \gamma} \left( R_{\dot \gamma} \left( \dot \gamma , V \right)\dot \gamma, U \right), \quad U, V \in \Gamma \gamma^*TM.\]
    Now, for two Jacobi fields $J, K$ along $\gamma$ we compute
	\begin{align*}
		&\;\frac{d}{dt} \left[ g_{\dot \gamma} ( J, \dot K ) - g_{\dot \gamma} ( \dot J, K )\right] \\
		=&\; g_{\dot \gamma}( J, \ddot K ) - g_{\dot \gamma}( \ddot J, K ) \\
		=&\; g_{\dot \gamma} \left( J, R_{\dot \gamma}(\dot \gamma, K) \dot \gamma \right) - g_{\dot \gamma} \left( R_{\dot \gamma}(\dot \gamma, J) \dot \gamma, K \right), \quad \text{by \autoref{eq:NJacobiEquation}} \\
		=&\; g_{\dot \gamma} \left( J, R_{\dot \gamma}(\dot \gamma, K) \dot \gamma \right) - g_{\dot \gamma} \left( R_{\dot \gamma}(\dot \gamma, K) \dot \gamma, J \right), \quad  \text{by symmetry of the curvature tensor} \\
		=&\; 0.
	\end{align*}
    Evaluating at $t = 0$, from \autoref{eq:NJacobiEquation} and  \autoref{lemma:shapeOperatorFormula} we have
    \[g_{\mathbf{v}}( J(0), \dot K(0) ) = g_{\mathbf{v}} \left( J(0), A_{\mathbf{v}}\left( K(0) \right) \right) = g_{\mathbf{v}}\left( \mathbf{v}, \Pi^{\mathbf{v}}\left( J(0), K(0) \right)\right).\]
    Since the second fundamental form $\Pi^{\mathbf{v}}$ is symmetric, we get $g_{\dot\gamma}( J(0), \dot K(0) ) - g_{\dot \gamma} ( \dot J(0), K(0) ) = 0$. Consequently, 
	$g_{\dot \gamma}(J, \dot K) = g_{\dot \gamma}(\dot J, K)$ holds for all time $t$.
\end{proof}

\subsection{Cut Locus of a Submanifold}
The \emph{cut locus of a point} $p \in M$ is the set consisting of all $q \in M$ such that there exists a minimizer from $p$ to $q$, any extension of which fails to be distance minimizing. We denote the cut locus of $p$ by $\mathrm{Cu}(p)$. Generalizing this notion to an arbitrary submanifold, we have the following definition.

\begin{defn}\label{defn:cutLocusN}
    Given a submanifold $N \subset M$, the \emph{cut locus} of $N$, denoted $\mathrm{Cu}(N)$, consists of points $q \in M$ such that there exists an $N$-segment joining $N$ to $q$, whose extension fails to be an $N$-segment. Given $\mathbf{v} \in S(\nu)$, the \emph{cut time} of $\mathbf{v}$ is defined as 
    \begin{equation}\label{eq:cutTime}
        \rho(\mathbf{v}) \coloneqq \sup \left\{ t \;\middle|\; d(N, \gamma_{\mathbf{v}}(t \mathbf{v})) = t \right\},
    \end{equation}
    where $\gamma_{\mathbf{v}}(t) = \exp^\nu(t \mathbf{v})$ is the unique $N$-geodesic with initial velocity $\mathbf{v}$.
\end{defn}

Note that we allow $\rho$ to take the value $\infty$, although if $M$ is compact then $\rho(\mathbf{v}) < \infty$ for any $\mathbf{v} \in S(\nu)$. The map $\rho: S(\nu ) \rightarrow  [0, \infty]$ is continuous \cite[Theorem 4.7]{BhoPra2023}. It follows from \cite{Alves2019} that $\rho$ is always strictly positive for any closed $N$, not necessarily compact. The \emph{tangent cut locus} of $N$ is defined as 
\begin{equation}\label{eq:tangentCutLocus}
    \widetilde{\mathrm{Cu}}(N) \coloneqq \left\{ \rho(\mathbf{v}) \mathbf{v} \;\middle|\; \mathbf{v}\in S(\nu), \; \rho(\mathbf{v}) \ne \infty \right\} \subset \nu.
\end{equation}
It follows from \autoref{defn:cutLocusN} that $\mathrm{Cu}(N) = \exp^\nu(\widetilde{\mathrm{Cu}}(N))$, as we have assumed $F$ to be forward complete.
    
In order to characterize the points of $\mathrm{Cu}(N)$, we introduce the notion of separating sets, originally called the \emph{several geodesics set} in \cite{Wol79}. In the terminology of \cite{Bishop77}, this is also known as \emph{ordinary} cut points, see \autoref{defn:separatingTangentCutPoint}.
\begin{defn}\label{defn:separtingSet}
    Given a submanifold $N \subset M$, a point $p \in M$ is said to be a \emph{separating point} of $N$ if there exist at least two distinct $N$-segments joining $N$ to $p$. The collection of all separating points of $N$ is called the \emph{separating set} of $N$, denoted $\mathrm{Se}(N)$.
\end{defn}

Under hypotheses (\hyperref[eq:hypothesisH]{$\mathsf{H}$}), it follows that $\mathrm{Cu}(N) = \overline{\mathrm{Se}(N)}$ \cite[Theorem 4.8]{BhoPra2023}. In order to describe the points in $\mathrm{Cu}(N) \setminus \mathrm{Se}(N)$, we need the notion of focal points of $N$, introduced in the next section.

\section{Focal Locus of a Submanifold}\label{sec:focalLocus}
In this section, we recall the focal locus of a submanifold, and the Morse index form of an $N$-geodesic. The main result in this section is \autoref{prop:indexLocallyConstant}, where we show that the index of $N$-geodesics is locally constant. Let us recall the definition.

\begin{defn}\label{defn:tangentFocalLocus}
    Given a submanifold $N \subset M$, a vector $\mathbf{v}\in \hat{\nu}$ is said to be a \emph{tangent focal point} of $N$ if 
    \[d_{\mathbf{v}}(\exp^\nu|_{\hat{\nu}}) : T_{\mathbf{v}} \hat{\nu} \rightarrow T_{\exp^\nu(\mathbf{v})}M\]
    is degenerate, i.e., if $\mathbf{v}$ is a critical point of $\exp^\nu|_{\hat{\nu}}$. The nullity of the map is known as the \emph{multiplicity} of the tangent focal point $\mathbf{v}$. The set of tangent focal points, called the \emph{tangent focal locus} of $N$, will be denoted as $\mathcal{F} = \mathcal{F}(N) \subset \hat{\nu}$.
\end{defn}

The \emph{focal locus} of $N$ is then defined as the image of the tangent focal locus under the $\exp^\nu$ map. In other words, the focal locus of $N$ consists of the critical values of the (restricted) normal exponential map $\mathcal{E} = \exp^\nu|_{\hat{\nu}}$. The \emph{first focal time} for $\mathbf{v} \in S(\nu)$ is defined as 
\begin{equation}\label{eq:firstFocalTime}
    \lambda(\mathbf{v}) \coloneqq \inf \left\{ t \;\middle|\; \text{$d_{t\mathbf{v}}\left( \exp^\nu|_{\hat{\nu}} \right)$ is degenerate} \right\}.
\end{equation}
The $N$-geodesic $\gamma_{\mathbf{v}}$ cannot be an $N$-segment beyond  $\lambda(\mathbf{v})$ \cite[Lemma 4.4]{BhoPra2023}, and in particular we have 
\begin{equation}\label{eq:cutTimeFocalTime}
    \rho(\mathbf{v}) \le \lambda(\mathbf{v}), \quad \forall \; \mathbf{v} \in S(\nu).
\end{equation}
Furthermore, under hypotheses (\hyperref[eq:hypothesisH]{$\mathsf{H}$}), we have a (non-exclusive) dichotomy: a point in $\mathrm{Cu}(N)$ is either a point in $\mathrm{Se}(N)$ or it is a \emph{first} focal locus along some $N$-segment \cite[Theorem 4.6]{BhoPra2023}.

Focal locus can be naturally characterized via $N$-Jacobi fields with vanishing endpoint \cite[Prop. 4.5]{Zhao2017}. For $\mathbf{v} \in \hat{\nu}$, denote the kernel 
\begin{equation}\label{eq:kernelExponential}
   \mathcal{K}_{\mathbf{v}} \coloneqq  \ker d_{\mathbf{v}}\mathcal{E} = \ker\Big( d_{\mathbf{v}}\left( \exp^\nu|_{\hat{\nu}} \right) : T_{\mathbf{v}} \hat{\nu} \rightarrow T_{\exp^\nu(\mathbf{v})} M = T_{\gamma_{\mathbf{v}}(1)} M \Big).
\end{equation}
On the other hand, consider the space of $N$-Jacobi fields
\[\mathfrak{J}_{\mathbf{v}} \coloneqq \left\{ J \;\middle|\; \text{$J$ is an $N$-Jacobi field along $\gamma_{\mathbf{v}}$}\right\},\]
and its subspace 
\[\mathfrak{K}_{\mathbf{v}} \coloneqq \left\{ J \;\middle|\; \text{$J$ is an $N$-Jacobi field along $\gamma_{\mathbf{v}}$, with $J(1) = 0$}\right\}.\]
For any $\mathbf{x} \in T_{\mathbf{v}} \hat{\nu}$, consider a curve $\alpha : (-\epsilon, \epsilon) \rightarrow \hat{\nu}$ such that $\alpha(0) = \mathbf{v}$ and $\dot\alpha(0) = \mathbf{x}$. Since $\nu$ is a cone bundle, we then have a family of curves 
\[\Lambda(s, t) \coloneqq \exp^\nu\left( t \alpha(s) \right), \quad 0 \le t \le 1, \; -\epsilon < s < \epsilon,\]
which is clearly an $N$-geodesic variation. In particular, 
\[J_{\mathbf{x}}(t) = \left.\frac{\partial}{\partial s}\right|_{s=0} \Lambda(s, t), \quad 0 \le t \le 1,\]
is an $N$-Jacobi field along the $N$-geodesic $\gamma_{\mathbf{v}}(t) = \exp^\nu(t\mathbf{v})$. We thus have the map
\begin{equation}\label{eq:canonicalIsoJacobi}
	\begin{aligned}
        \Phi : T_{\mathbf{v}} \hat{\nu} &\rightarrow \mathfrak{J}_{\mathbf{v}}\\
        \mathbf{x} &\mapsto J_{\mathbf{x}}.
    \end{aligned}
\end{equation}

\begin{prop}\label{prop:tangentSpaceOfNormalConeBundle}
	$\Phi$ is a well-defined linear isomorphism $T_{\mathbf{v}} \hat{\nu} \cong \mathfrak{J}_{\mathbf{v}}$, which restricts to an isomorphism $\mathcal{K}_{\mathbf{v}} \cong \mathfrak{K}_{\mathbf{v}}$.
\end{prop}
\begin{proof}
	Let us denote $c(s) = \Lambda(s, 0) = \pi \circ \alpha(s)$, where $\pi : \nu \rightarrow N$ is the projection. Clearly, $c$ is a curve in $N$ passing through $p = \pi(\mathbf{v})$. Observe that $\partial_t|_{t = 0} \Lambda(s, t) = \alpha(s)$, as $\Lambda(s, \_)$ is an $N$-geodesic with initial velocity $\alpha(s)$. We compute the following.
	\begin{align*}
		J_{\mathbf{x}}(0) &= \left. \frac{\partial}{\partial s} \right|_{s=0} \exp^\nu (0 \cdot \alpha(s)) = \left. \frac{\partial}{\partial s} \right|_{s=0} c(s) = \dot c(0) = d\pi(\dot\alpha(0)) = d\pi(\mathbf{x}).\\[1em]
		D^{\dot \gamma}_\gamma J_{\mathbf{x}} (0) &= D^{\dot \gamma}_\gamma |_{t = 0} \partial_s \Lambda(s, t) = D^{\dot \gamma}_{c}|_{s = 0} \partial_t|_{t = 0} \Lambda(s, t) = D^{\dot \gamma}_c \alpha (0). 
	\end{align*}
	Now, $D^{\dot \gamma}_c \alpha (0)$ is determined by the first jet of $\alpha$ at $0$, i.e., by $\mathbf{x} = \dot\alpha(0)$ (see for example, \cite[Eq (3)]{Javaloyes2014}). Hence, $J_{\mathbf{x}}$ is uniquely determined by $\mathbf{x}$, and in particular, $\Phi$ is a well-defined map. Note that for $\mathbf{x} \in \mathcal{K}_{\mathbf{v}} = \ker d_{\mathbf{v}} \mathcal{E}$, we have 
	\[J_{\mathbf{x}}(1) = \left. \frac{\partial}{\partial s} \right|_{s=0} \exp^\nu(\alpha (s)) = d_{\mathbf{v}} \mathcal{E} (\dot \alpha(0)) = d_{\mathbf{v}} \mathcal{E}(\mathbf{x}) = 0.\]
	Consequently, $\Phi$ restricts to a map $\mathcal{K}_{\mathbf{v}} \rightarrow \mathfrak{K}_{\mathbf{v}}$.

	Let us show that $\Phi$ is linear. Clearly, 
	\[J_{a\mathbf{x} + b\mathbf{y}}(0) = d\pi(a\mathbf{x} + b\mathbf{y}) = a \, d\pi(\mathbf{x}) + b \, d\pi(\mathbf{y}) = a J_{\mathbf{x}}(0) + b J_{\mathbf{y}}(0).\]
	Similarly, it follows from \cite[Eq (3)]{Javaloyes2014} that $D^{\dot\gamma}_\gamma J_{a \mathbf{x} + b \mathbf{y}}(0) = a D^{\dot\gamma}_\gamma J_{\mathbf{x}}(0) + b D^{\dot\gamma}_\gamma J_{\mathbf{y}}(0)$. But then, from the uniqueness of Jacobi fields, it follows that $J_{a \mathbf{x} + b \mathbf{y}} = a J_{\mathbf{x}} + b J_{\mathbf{y}}$. Consequently, $\Phi$ is linear.

	If $\Phi(\mathbf{x}) = 0$, i.e., if $J_{\mathbf{x}}$ is the $0$ vector field, we have $J_{\mathbf{x}}(0) = 0$ and $\dot J_{\mathbf{x}}(0) = 0$. This implies $\mathbf{x}=0$, proving the injectivity of $\Phi$. On the other hand, suppose $J \in \mathfrak{J}_{\mathbf{v}}$ is an $N$-Jacobi field. Then, $J$ is given by an $N$-geodesic variation, say, $\Xi : (-\epsilon, \epsilon) \times [0, 1] \rightarrow M$. We have a curve $\beta$ in $\hat{\nu}$ such that $\Xi(s, t) = \exp^\nu(t\beta(s))$. The above computation then shows that $J = J_{\mathbf{x}}$, where $\mathbf{x} = \dot \beta(0)$. If $J(1) = 0$, we have $\mathbf{x} \in \mathcal{K}_{\mathbf{v}}$, since $d_{\mathbf{v}}\mathcal{E}(\mathbf{x}) = J_{\mathbf{x}}(1) = 0$. Thus, $\Phi$ is a linear isomorphism, restricting to an isomorphism $\mathcal{K}_{\mathbf{v}} \cong \mathfrak{K}_{\mathbf{v}}$, concluding the proof.
\end{proof}

\begin{remark}\label{rmk:maximumFocalMultiplicity}
    It is immediate that given a tangent focal point $\mathbf{v} \in \hat{\nu}$, the focal multiplicity $\dim \mathcal{K}_{\mathbf{v}} \le \dim T_{\mathbf{v}} \hat{\nu} = \dim M$. On the other hand, from \autoref{example:jacobiField}, we have a Jacobi field $J \in \mathfrak{J}_{\mathbf{v}} \setminus \mathfrak{K}_{\mathbf{v}}$. Thus, it follows from \autoref{prop:tangentSpaceOfNormalConeBundle} that $\dim \mathcal{K}_{\mathbf{v}} = \dim \mathfrak{K}_{\mathbf{v}} \lneq \dim \mathfrak{J}_{\mathbf{v}} = \dim M$.
\end{remark}

Using \autoref{prop:tangentSpaceOfNormalConeBundle}, we get the following useful result, which is analogous to \cite[Lemma 4.9]{Sak96} in the Riemannian context.
\begin{lemma}\label{lemma:jacobiFrameAtFocalPoint}
    Suppose for $\mathbf{v} \in S(\nu)$, and let $k = \dim \mathcal{K}_{\ell \mathbf{v}}$ for some $\ell > 0$. Then, there exists a frame of $N$-Jacobi fields $\left\{ J_1,\dots, J_n \right\}$ along $\gamma_{\mathbf{v}}$, where $n = \dim M$, such that the following holds.
    \begin{enumerate}[label=(\arabic*)]
        \item \label{lemma:jacobiFrameAtFocalPoint:1} The subspaces $\mathrm{Span}\langle \dot J_1(\ell), \dots ,\dot J_k(\ell) \rangle$ and $\mathrm{Span}\langle J_{k+1}(\ell), \dots , J_n(\ell) \rangle$ are $g_{\ell \mathbf{v}}$ orthogonal.
        \item \label{lemma:jacobiFrameAtFocalPoint:2} $T_{\gamma_{\mathbf{v}}(\ell)} M = \mathrm{Span}\langle \dot J_1(\ell), \dots, \dot J_k(\ell), J_{k+1}(\ell), \dots , J_n(\ell) \rangle$.
        \item \label{lemma:jacobiFrameAtFocalPoint:3} $\left\{ J_i(t) \right\}_{i=1}^n$ is a basis of $T_{\gamma_{\mathbf{v}}(t)} M$ for a deleted neighborhood $0 < |t - \ell| < \epsilon$.
    \end{enumerate}
    In particular, focal points of $N$ along $\gamma_{\mathbf{v}}$ are discrete.
\end{lemma}
\begin{proof}
    Fix a basis $\mathcal{K}_{\ell\mathbf{v}} = \ker d_{\ell \mathbf{v}} \mathcal{E} = \mathrm{Span}\langle \mathbf{x}_1,\dots ,\mathbf{x}_k \rangle$, and extend it to a basis $T_{\ell \mathbf{v}} \hat{\nu} = \mathrm{Span}\langle \mathbf{x}_1,\dots ,\mathbf{x}_n \rangle$. Consider the $N$-Jacobi fields $J_i \coloneqq J_{\mathbf{x}_i} = \Phi(\mathbf{x}_i)$ along $\gamma_{\mathbf{v}}$. By \autoref{prop:tangentSpaceOfNormalConeBundle}, $\left\{ J_i \right\}$ forms a basis of $\mathfrak{J}_{\mathbf{v}}$, and we have $J_1(\ell) = \dots = J_k(\ell) = 0$. Suppose, if possible, $\sum_{i=1}^k a^i \dot J_i(\ell) = 0$ for some scalars $a^i$. Consider the $N$-Jacobi field $J = \sum_{i=1}^k a^i J_i$. Then, $\dot J(\ell) = \sum_{i=1}^k a^i \dot J_i(\ell) = 0$. By \autoref{lemma:zerosOfJacobiField}, we have $J = 0$, which forces $a^1 = \dots = a^k = 0$. Thus, $\left\{ \dot J_1(\ell), \dots , \dot J_k(\ell) \right\}$ are linearly independent. Next, suppose $\sum_{i = k+1}^n b^i J_i (\ell) = 0$ for some scalars $b^i$. Consider the $N$-Jacobi field $\bar{J} = \sum_{i > k} b^i J_i$. As $\Phi$ in \autoref{eq:canonicalIsoJacobi} is a linear isomorphism, we have $\bar{J} = \Phi(\mathbf{y})$, where $\mathbf{y} = \sum_{i > k} b^i \mathbf{x}_i$. Now, $d_{\ell \mathbf{v}}\mathcal{E}(\mathbf{y}) = \bar{J}(\ell) = \sum_{i > k} b^i J_i(\ell) = 0$ implies $\mathbf{y} \in \mathcal{K}_{\ell \mathbf{v}}$. By our choice of $\mathbf{x}_i$, we must have $b^{k+1} = \dots = b^n = 0$. Thus, $\left\{ J^{k+1}(\ell), \dots , J^n(\ell) \right\}$ are linearly independent as well. Lastly, for some $J = \sum_{i=1}^k a^i J_i$ and some $\bar{J} = \sum_{i > k} b^i J_i$, we have from \autoref{lemma:jacobiFieldAdjoint}
    \[g_{\ell \mathbf{v}}\left( \dot J(\ell), \bar{J}(\ell) \right) = g_{\ell \mathbf{v}}\left( J(\ell), \dot{\bar{J}}(\ell) \right) = 0.\]
    Consequently, $\mathrm{Span}\langle \dot J_1(\ell),\dots, \dot J_k(\ell) \rangle$ and $\mathrm{Span}\langle J_{k+1}(\ell), \dots, J_n(\ell) \rangle$ are $g_{\ell \mathbf{v}}$ orthogonal, which proves \autoref{lemma:jacobiFrameAtFocalPoint:1}. A simple dimension counting then shows that $T_{\gamma_{\mathbf{v}}(\ell)} M = \mathrm{Span}\langle \dot J_1(\ell), \dots, \dot J_k(\ell), J_{k+1}(\ell), \dots , \dot J_n (\ell) \rangle$, proving \autoref{lemma:jacobiFrameAtFocalPoint:2}.
    

    Now, for $1 \le i \le k$, we have $J_i(\ell) = 0$ and $\dot J_i(\ell) \ne 0$, i.e., $J_i(t)$ has a $0$ of order $1$ at $t = \ell$. It follows from \autoref{lemma:jacobiFieldMultipleOfNonZeroField} that we can uniquely write $J_i(t) = (t - \ell) Y_i(t)$, where $Y_i(t)$ is a smooth vector field along $\gamma_{\mathbf{v}}$, nonvanishing near $\ell$. Observe that, $\dot J_i(\ell) = Y_i(\ell) = \lim_{t \rightarrow \ell} \frac{J_i(t)}{t - \ell}$. Thus, near $\ell$ we have the smooth nonvanishing vector fields
    \begin{equation}\label{eq:nonZeroMultipleOfJacobi}
        Y_i(t) = 
    \begin{cases}
        \frac{J_i(t)}{t - \ell}, \quad t \ne \ell \\
        \dot J_i(\ell), \quad t = \ell.
    \end{cases}
    \end{equation}
    Set $Y_i = J_i$ for $k < i \le n$. Clearly, $\mathrm{Span}\langle Y_i(t) \rangle = \mathrm{Span}\langle J_i(t) \rangle$ for $t$ in a deleted neighborhood of $\ell$. Now, $\left\{ Y_i(\ell) \right\}$ are linearly independent by \autoref{lemma:jacobiFrameAtFocalPoint:2}, and hence, in some neighborhood $|t - \ell| < \epsilon$, we have $\left\{ Y_i(t) \right\}$ are linearly independent as well. But then for the deleted neighborhood $0 < |t - \ell| < \epsilon$, we get $\left\{ J_i(t) \right\}$ are linearly independent, which proves \autoref{lemma:jacobiFrameAtFocalPoint:3}.
    
    Lastly, suppose $t_0 \mathbf{v}$ is a focal point of $N$ along $\gamma_{\mathbf{v}}$, for some $0 < |t_0 - \ell| < \epsilon$, where $\epsilon > 0$ as before. Then there exists some non-vanishing $N$-Jacobi field $J$ along $\gamma_{\mathbf{v}}$ with $J(t_0) = 0$. Write $J = \sum c^i J_i$ for some scalars $c^i$. We have $\sum c^i J_i(t_0) = J(t_0) = 0$, which implies $c^i = 0$, as $t_0 \ne \ell$. This is a contradiction, and hence, there are no focal points of $N$ along $\gamma_{\mathbf{v}}$ for $0 < |t - \ell| < \epsilon$. Thus, focal points of $N$ along $\gamma_{\mathbf{v}}$ are discrete, which concludes the proof.
\end{proof}

\subsection{Index Form and the Morse Index Theorem}
From the second variation formula of the restricted energy functional $E|_{\mathcal{P}_N}$, we get the index form.

\begin{defn}\cite{Javaloyes2014, Zhao2017} \label{defn:NIndexForm}
    Let $\gamma : [a,b] \rightarrow M$ be a unit-speed $N$-geodesic. Then, the \emph{index form} $\mathcal{I}_\gamma$ is defined for $X, Y \in T_\gamma \mathcal{P}_N$ as 
    \begin{equation}\label{eq:indexForm}
        \mathcal{I}_\gamma(X,Y) \coloneqq \int_a^b \left[ g_{\dot\gamma}(D^{\dot\gamma}_\gamma X, D^{\dot\gamma}_\gamma Y) - g_{\dot\gamma}\left( R_{\dot\gamma}(\dot\gamma, X) Y, \dot \gamma \right) \right] - g_{\dot\gamma(a)} \left( \pi^\perp \nabla^{\dot\gamma(a)}_X Y, \dot\gamma(a) \right),
    \end{equation}
    where $\pi^\perp$ denotes projection in the canonical splitting $T_{\gamma(a)}M = T_{\gamma(a)}N \oplus \left( T_{\gamma(a)}N \right)^{\perp_{g_{\dot\gamma(a)}}}$ on to the second component.
\end{defn}

It follows that $\mathcal{I}_\gamma$ is a symmetric $2$-form, and the kernel of the index form $\mathcal{I}_\gamma$ consists of precisely the $N$-Jacobi fields $J$ along $\gamma$, with $J(b) = 0$. In particular, $\gamma$ is a \emph{nondegenerate} critical point (i.e., an $N$-geodesic) of the energy functional, precisely when there are no $N$-Jacobi fields along $\gamma$ vanishing at the endpoint. In other words, $\gamma$ is nondegenerate if and only if $\gamma(b)$ is not a focal point of $N$ along $\gamma$. 

\begin{defn}\label{defn:index}
    The \emph{index} of an $N$-geodesic $\gamma$ is defined as 
    \[\mathrm{Ind}(\gamma) \coloneqq \max \left\{ \dim K \;\middle|\; \text{$\mathcal{I} _\gamma$ is negative definite restricted to some $K \subset T_\gamma\mathcal{P}_N$} \right\}.\]    
\end{defn}

By the Morse index theorem for Finsler submanifolds \cite{Peter06,Lu24}, it follows that for the unit-speed $N$-geodesic $\gamma : [0, T] \rightarrow M$ given as $\gamma(t) = \exp^\nu(t \mathbf{v})$, such that $\gamma(T)$ is not a focal point of $N$ along $\gamma$, one has 
\begin{equation}\label{eq:morseIndexFormula}
    \mathrm{Ind}(\gamma) = \sum_{0 < t < T} \dim \mathcal{K}_{t\mathbf{v}} = \sum_{0 < t < T} \dim \ker \left( d_{t\mathbf{v}} \left( \exp^\nu|_{\hat{\nu}} \right) \right) < \infty.
\end{equation}
In other words, $\mathrm{Ind}(\gamma)$ equals the number of focal points of $N$ along $\gamma$, counted with multiplicities, excluding the endpoint. Note that by \autoref{lemma:jacobiFrameAtFocalPoint}, only finitely many non-zero terms can appear in the sum. For a fixed unit vector $\mathbf{v} \in S(\nu)$, based on the index of $\gamma_{\mathbf{v}}$, we can now define the following.

\begin{defn}\label{defn:focalTime}
    Given $N \subset M$ and $\mathbf{v} \in S(\nu)$, the \emph{$k^{\text{th}}$ focal time} in the direction of $\mathbf{v}$ is defined as 
    \[\lambda_k(\mathbf{v}) \coloneqq \sup \left\{ t \;\middle|\; \mathrm{Ind}(\gamma_{\mathbf{v}}|_{[0, t]}) \le k - 1 \right\},\]
    where $\gamma_{\mathbf{v}} : [0, \infty) \rightarrow M$ is the unit speed $N$-geodesic given by $\gamma_{\mathbf{v}}(t) = \exp^\nu(t \mathbf{v})$. The \emph{$k^{\text{th}}$ focal locus} of $N$ is defined as the set $\left\{ \exp^\nu\left( \lambda_k(\mathbf{v}) \mathbf{v} \right) \mid \mathbf{v} \in S(\nu), \lambda_k(\mathbf{v}) \ne \infty \right\}$.
\end{defn}

Clearly, $\lambda_1(\mathbf{v}) = \lambda(\mathbf{v})$ as defined in \autoref{eq:firstFocalTime}. In general, we have $0 < \lambda_1(\mathbf{v}) \le \lambda_2(\mathbf{v}) \le \dots $. If $\gamma_{\mathbf{v}}(T)$ is not a focal point of $N$ along $\gamma_{\mathbf{v}}$, we can immediately see 
\begin{equation}\label{eq:indexFormulaNoFocal}
    \mathrm{Ind}(\gamma_{\mathbf{v}}|_{[0, T]}) = \sum \dim \mathcal{K}_{\lambda_k(\mathbf{v})\mathbf{v}},
\end{equation}
where the sum runs over the finite set $\left\{ \lambda_k(\mathbf{v}) < T \right\}$. We now prove a useful result about the index being locally constant near a non-focal point, analogous to \cite[Prop. 1.2]{Itoh2001}.

\begin{prop}\label{prop:indexLocallyConstant}
    Suppose, for $\mathbf{v}_0 \in S(\nu)$, we have $\gamma_{\mathbf{v}_0}(T)$ is not a focal point of $N$ along $\gamma_{\mathbf{v}_0}$. Then, there exists a neighborhood $\mathbf{v}_0 \in U \subset S(\nu)$ such that, 
    \[\mathrm{Ind}(\gamma_{\mathbf{v}}|_{[0, T]}) = \mathrm{Ind}(\gamma_{\mathbf{v}_0}|_{[0, T]}), \qquad \text{for all } \mathbf{v} \in U.\]
\end{prop}
\begin{proof}
    Let us first show that $\mathrm{Ind}(\gamma_{\mathbf{v}}|_{[0, T]}) \ge \mathrm{Ind}(\gamma_{\mathbf{v}_0}|_{[0, T]})$ for $\mathbf{v}$ sufficiently near $\mathbf{v}_0$. As $S(\nu)$ is a smooth manifold of dimension $n - 1$, consider a chart $\mathbf{v}_0 \in U \subset S(\nu)$, where $U \cong \mathbb{D}^{n-1}$. Then, we have an $(n-1)$-dimensional $N$-geodesic variation given as 
    \begin{align*}
        \Lambda : U \times [0, T] &\longrightarrow M \\
        (\mathbf{v}, t) &\longmapsto \exp^\nu(t \mathbf{v}).
    \end{align*}
    Clearly, $\gamma_{\mathbf{v}} = \Lambda(\mathbf{v}, \_)$. Suppose for some $X \in \Gamma \gamma_{\mathbf{v}_0}^*TM$, we have $I_{\gamma_{\mathbf{v}_0}}(X, X) < 0$. Since $U$ is contractible, we can get an extension, say, $\tilde{X} \in \Gamma \Lambda^* TM$. Denoting $\tilde{X}_{\mathbf{v}}(t) \coloneqq \tilde{X}(\mathbf{v}, t)$, we see that $\tilde{X}_{\mathbf{v}} \in \Gamma \gamma_{\mathbf{v}}^* TM$. Since the index form (\autoref{eq:indexForm}) is continuous, shrinking $U$ if necessary, we may assume that $I_{\gamma_{\mathbf{v}}}(\tilde{X}_{\mathbf{v}}, \tilde{X}_{\mathbf{v}}) < 0$ for all $\mathbf{v} \in U$. Now, suppose $k_0 \coloneqq \mathrm{Ind}(\gamma_{\mathbf{v}_0}|_{[0, T]})$. Choose a frame $X^1, \dots , X^{k_0} \in \Gamma \gamma_{\mathbf{v}_0}^*TM$ spanning a maximal subspace on which $I_{\gamma_{\mathbf{v}_0}}$ is negative definite, and thus satisfying
    \[I_{\gamma_{\mathbf{v}_0}}(X^i, X^i) < 0, \qquad 1 \le i \le k_0.\]
    As described above, shrinking $U$ as necessary for finitely many times, we have $\tilde{X}^i_{\mathbf{v}} \in \Gamma \gamma_{\mathbf{v}}^* TM$ such that 
    \[I_{\gamma_{\mathbf{v}}}\left( \tilde{X}^i_{\mathbf{v}}, \tilde{X}^i_{\mathbf{v}} \right) < 0, \qquad 1 \le i \le k_0, \; \mathbf{v} \in U.\]
    Since $\left\{ X^i \right\}$ are linearly independent, possibly shrinking $U$ further, we have $\left\{ \tilde{X}^i_{\mathbf{v}} \right\}$ are linearly independent for all $\mathbf{v} \in U$. But then, $\mathrm{Ind}(\gamma_{\mathbf{v}}|_{[0, T]}) \ge k_0 = \mathrm{Ind}(\gamma_{\mathbf{v}_0}|_{[0, T]})$ for all $\mathbf{v} \in U$. Note that this inequality is true irrespective of whether $\gamma_{\mathbf{v}_0}(T)$ is a focal point of $N$ along $\gamma_{\mathbf{v}_0}$ or not.\medskip

    Since by hypothesis $d_{T\mathbf{v}_0} \mathcal{E}$ is nonsingular, we can assume that $d_{T \mathbf{v}} \mathcal{E}$ is nonsingular for all $\mathbf{v} \in U$ as well. Thus, $\gamma_{\mathbf{v}}(T)$ is not a focal point of $N$ along $\gamma_{\mathbf{v}}$ for all $\mathbf{v}\in U$. We show that $\mathrm{Ind}(\gamma_{\mathbf{v}})$ is constant on some neighborhood of $\mathbf{v}_0$ in $U$. If not, choose some sequence of distinct $\mathbf{v}_j \in U$ such that $\mathrm{Ind}(\gamma_{\mathbf{v}_j}|_{[0, T]}) \ne \mathrm{Ind}(\gamma_{\mathbf{v}_0}|_{[0, T]})$. Passing to a subsequence, we may assume that $\mathbf{v}_j \rightarrow \mathbf{v}_0$, and
    \begin{equation}\label{eq:indexLocallyConstant:contradiction} \tag{$*$}
        \mathrm{Ind}(\gamma_{\mathbf{v}_j}|_{[0, T]}) > \mathrm{Ind}(\gamma_{\mathbf{v}_0}|_{[0, T]}) \quad \forall j.
    \end{equation}
    Suppose, for each $j$, there are precisely $r_j$ many distinct focal times of $N$ along $\gamma_{\mathbf{v}_j}$ appearing in $(0, T)$, and let us rename them as $0 < \mu_1(\mathbf{v}_j) < \dots < \mu_{r_j}(\mathbf{v}_j) < T$. Note that for each $j$ we must have at least the first focal time $\lambda_1(\mathbf{v}_j)$ appears in $(0, T)$, as (\hyperref[eq:indexLocallyConstant:contradiction]{$*$}) implies $\mathrm{Ind}(\gamma_{\mathbf{v}_j}|_{[0, T]}) \ge 1$. Thus, $r_j \ge 1$. Denote, 
    \[K(j, k) \coloneqq \mathcal{K}_{\mu_k(\mathbf{v}_j)\mathbf{v}} = \ker \left( d_{\mu_k(\mathbf{v}_j)\mathbf{v}_j} \mathcal{E} \right).\]
    By the Morse index theorem, we have 
    \begin{equation}\label{eq:indexLocallyConstant:indexAtVj} \tag{$*_j$}
        \mathrm{Ind}\left( \gamma_{\mathbf{v}_j}|_{[0, T]} \right) = \sum_{k = 1}^{r_j} \dim K(j, k), \qquad  \text{for all } j \ge 1.
    \end{equation}
    As $\dim K(j, 1) \in \left\{ 1,\dots , n -1 \right\}$, we must have $\dim K(j,1)$ is a constant for infinitely values of $j$. Thus, passing to a subsequence, we may assume that $\eta_1 \coloneqq \dim K(j, 1)$ for all $j$. Let us now consider $K(j, 1)$ as points in the Grassmann bundle $\textrm{Gr}_{\eta_1}(T\hat{\nu})$. Since the base-points $\mathbf{v}_j \rightarrow \mathbf{v}_0$, passing to a subsequence, we may assume that $K(j, 1)$ converges to an $\eta_1$-dimensional subspace, say, $Z_1 \subset T_{t_1 \mathbf{v}_0} \hat{\nu}$ for some time $t_1 > 0$. Clearly, $\mu_1(\mathbf{v}_j) \mathbf{v}_j \rightarrow t_1 \mathbf{v}_0 \Rightarrow \mu_1(\mathbf{v}_j) \rightarrow t_1$, as $F(\mathbf{v}_j) = 1 = F(\mathbf{v}_0)$. Since $\mu_1(\mathbf{v}_j) < T$, it follows that $t_1 \le T$ in the limit. Now, $d_{\mu_1(\mathbf{v}_j)} \mathcal{E}$ vanishes on $K(j, 1)$, and hence by continuity, $d_{t_1 \mathbf{v}_0} \mathcal{E}$ vanishes on the limit subspace $Z_1$ as well. In particular, $t_1 \mathbf{v}_0$ is a focal point of $N$ along $\gamma_{\mathbf{v}_0}$. Since, by assumption, $\gamma_{\mathbf{v}_0}(T)$ is not a focal point, we must have $0 < t_1 < T$. Also, note that $\mathrm{Ind}(\gamma_{\mathbf{v}_0}|_{[0, T]}) \ge \dim Z_1 = \eta_1$. Inductively, passing on to further subsequences as necessary, for some $k \ge 1$, we assume that the following holds for all $j \ge 1$.
    \begin{itemize}
        \item The $k^{\text{th}}$-distinct focal time $\mu_k(\mathbf{v}_j)$ is defined in $(0, T)$ for infinitely many $j$. 
        \item For infinitely many $j$, there is a constant value, say, $\eta_k \coloneqq \dim K(j, k)$ such that 
        \[\eta_1 + \dots + \eta_k \le \mathrm{Ind}(\gamma_{\mathbf{v}_0}|_{[0, T]}).\]
        \item $K(j, k) \rightarrow Z_k \subset T_{t_k \mathbf{v}_0} \hat{\nu}$, where $\dim Z_k = \eta_k$, and $0 < t_k < T$.
    \end{itemize}
    Since $\mu_k(\mathbf{v}_j) < \mu_{k+1}(\mathbf{v}_j)$ holds for all $j$, in the limit we get $t_k \le t_{k+1}$, and moreover, $0 < t_k < T$. For distinct values of $t_k$, each $Z_k$ as above are then subspaces of the tangent spaces at distinct points of $\hat{\nu}$ along the ray $R_{\mathbf{v}_0} = \left\{ t \mathbf{v}_0 \;\middle|\; t > 0 \right\}$. On the other hand, if a particular value of $t_k$ is repeated more than once, we shall see that the corresponding subspaces $Z_k$ form a direct sum. Thus, the total number of $t_k$, counting multiplicity, cannot exceed $\mathrm{Ind}\left( \gamma_{\mathbf{v}_0} |_{[0, T]} \right)$, and hence, the above induction terminates after finitely many steps.

    Let us now consider the situation $t_{k_1} = \dots = t_{k_s} = t_0$ for some $s \ge 2$. We show that $Z_{k_1} + \dots + Z_{k_s}$ is a direct sum in the tangent space $T_{t_0 \mathbf{v}_0} \hat{\nu}$. Clearly, $\sum Z_{k_i} \subset \mathcal{K}_{t_0 \mathbf{v}_0}$. Fix some $\alpha \ne \beta \in \left\{ k_1, \dots , k_s \right\}$. For some $\mathbf{a} \in Z_{\alpha}$ and $\mathbf{b} \in Z_{\beta}$, get sequences $\mathbf{a}_j \in K(j, \alpha)$ and $\mathbf{b}_j \in K(j, \beta)$ such that $\mathbf{a}_j \rightarrow \mathbf{a}, \mathbf{b}_j \rightarrow \mathbf{b}$. Consider the $N$-Jacobi fields $J_{\mathbf{a}_j} = \Phi(\mathbf{a}_j), J_{\mathbf{b}_j} = \Phi(\mathbf{b}_j)$ along $\gamma_{\mathbf{v}_j}$. By \autoref{lemma:jacobiFieldAdjoint}, we have $g_{\dot\gamma_{\mathbf{v}_j}}\left( \dot J_{\mathbf{a}_j}, J_{\mathbf{b}_j} \right) = g_{\dot\gamma_{\mathbf{v}_j}}\left( J_{\mathbf{a}_j}, \dot J_{\mathbf{b}_j} \right)$ for all time $t$. As $J_{\mathbf{a}_j}\left( t \right) = 0$ for $t = \mu_\alpha(\mathbf{v}_j)$, we have 
    \begin{equation}\label{eq:indexLocallyConstant::dotJajJbjIsZero}\tag{$\dagger$}
        g_{\dot \gamma_{\mathbf{v}_j}(\mu_\alpha(\mathbf{v}_j))} \left( \dot J_{\mathbf{a}_j} \left( \mu_\alpha(\mathbf{v}_j) \right), J_{\mathbf{b}_j}\left( \mu_\alpha(\mathbf{v}_j) \right) \right) = 0, \qquad j \ge 1.
    \end{equation}
    On the other hand, it follows from \autoref{lemma:jacobiFieldMultipleOfNonZeroField}, that there are smooth vector fields $\bar{J}_{\mathbf{b}_j}$ along $\gamma_{\mathbf{v}_j}$, and $\bar{J}_{\mathbf{b}}$ along $\gamma_{\mathbf{v}_0}$, satisfying 
    \[\bar{J}_{\mathbf{b}_j}(t) = 
    \begin{cases}
        \frac{J_{\mathbf{b}_j}(t)}{t - \mu_\beta(\mathbf{v}_j)} , \quad t \ne \mu_\beta(\mathbf{v}_j) \\[1em]
        \dot J_{\mathbf{b}_j} \left( \mu_\beta(\mathbf{v}_j) \right), \quad t = \mu_\beta(\mathbf{v}_j),
    \end{cases}, \qquad \bar{J}_{\mathbf{b}}(t) = 
    \begin{cases}
        \frac{J_{\mathbf{b}}(t)}{t - t_0} , \quad t \ne t_0 \\[1em]
        \dot J_{\mathbf{b}} \left( t_0 \right), \quad t = t_0,
    \end{cases}\]
    As $\mu_\alpha(\mathbf{v}_j) - \mu_\beta(\mathbf{v}_j) \ne 0$, from (\hyperref[eq:indexLocallyConstant::dotJajJbjIsZero]{$\dagger$}) we have, 
    \[g_{\dot\gamma_{\mathbf{v}_j}(\mu_\alpha(\mathbf{v}_j))}\left( \dot J_{\mathbf{a}_j} \left( \mu_\alpha(\mathbf{v}_j) \right), \bar{J}_{\mathbf{b}_j} \left( \mu_\alpha(\mathbf{v}_j) \right) \right) = 0, \qquad j \ge 1.\]
    Since $\mu_\alpha(\mathbf{v}_j) \rightarrow t_\alpha = t_0$ and $\mu_\beta(\mathbf{v}_j) \rightarrow t_\beta = t_0$ as $j \rightarrow \infty$, taking limit we get 
    \[0 = g_{\dot \gamma_{\mathbf{v}_0}(t_0)}\left( \dot J_{\mathbf{a}} (t_0), \bar{J}_{\mathbf{b}} (t_0) \right) = g_{\dot \gamma_{\mathbf{v}_0}(t_0)}\left( \dot J_{\mathbf{a}}(t_0), \dot J_{\mathbf{b}}(t_0) \right).\]
    Now, it follows from \autoref{prop:tangentSpaceOfNormalConeBundle} and \autoref{lemma:jacobiFrameAtFocalPoint} that the linear map
    \begin{align*}
        \Psi : \mathcal{K}_{t_0 \mathbf{v}_0} &\longrightarrow T_{\gamma_{\mathbf{v}_0}(t_0)} M \\
        \mathbf{x} &\longmapsto \dot J_{\mathbf{x}}(t_0)
    \end{align*}
    is \emph{injective}. The above discussion then shows that for any $\alpha \ne \beta \in \left\{ k_1,\dots, k_s \right\}$, the subspaces $\Psi(Z_\alpha), \Psi(Z_\beta)$ are $g_{\dot\gamma_{\mathbf{v}_0}(t_0)}$-orthogonal. Consequently, $\sum_{i=1}^s \Psi\left( Z_{k_i} \right)$ is a direct sum. As $\Psi$ is injective, it follows that $\sum_{i=1}^s Z_{k_i}$ is a direct sum as well.
    
    We can now assume that the inductive process stops at, say, $r\textsuperscript{th}$-step. Consequently, counting multiplicity, there are precisely $r$-many focal times of $N$ along $\gamma_{\mathbf{v}_j}$ appearing in $(0, T)$ for each $j$. Now, we have obtained that 
    \begin{equation}\label{eq:indexLocallyConstant:indexAtV0}\tag{$*_0$}
        \mathrm{Ind}(\gamma_{\mathbf{v}_0}|_{[0, T]}) \ge \sum_{k=1}^r \dim Z_k = \sum_{k=1}^r \eta_k.
    \end{equation}
    But then for each $j \ge 1$ we get from (\hyperref[eq:indexLocallyConstant:indexAtVj]{$*_j$}) and (\hyperref[eq:indexLocallyConstant:indexAtV0]{$*_0$}) that
    \[\mathrm{Ind}(\gamma_{\mathbf{v}_j}|_{[0, T]}) = \sum_{k=1}^r \dim K(j, k) = \sum_{k=1}^r \dim Z_k \le \mathrm{Ind}((\gamma_{\mathbf{v}_0}|_{[0, T]})),\]
    which is a contradiction to (\ref{eq:indexLocallyConstant:contradiction}). Hence, we have some neighborhood of $\mathbf{v}_0$ in $S(\nu)$ on which $\mathrm{Ind}(\gamma_{\mathbf{v}}|_{[0, T]})$ is constant. This concludes the proof.
\end{proof}

As an immediate corollary, we get the following.

\begin{corollary}\label{cor:focalTimeContinuous}
    For each $k\ge 1$, the functions $\lambda_k : S(\nu) \rightarrow (0, \infty]$ is continuous.
\end{corollary}
\begin{proof}
    Fix some $\mathbf{v}_0 \in S(\nu)$. Suppose $\lambda_k(\mathbf{v}_0) = T < \infty$. Then, for $\epsilon > 0$ small, we have from \autoref{lemma:jacobiFrameAtFocalPoint} that $\gamma_{\mathbf{v}_0}(T+\epsilon)$ and $\gamma_{\mathbf{v}_0}(T-\epsilon)$ are not focal points of $N$ along $\gamma_{\mathbf{v}_0}$. It follows from \autoref{prop:indexLocallyConstant} that there exists some neighborhood $\mathbf{v}_0 \in U \subset S(\nu)$ such that for all $\mathbf{v} \in U$ we have
    \[\mathrm{Ind}(\gamma_{\mathbf{v}}|_{[0, T- \epsilon]}) = \mathrm{Ind}(\gamma_{\mathbf{v}_0}|_{[0, T-\epsilon]}) \le k - 1, \quad \text{and}\quad \mathrm{Ind}(\gamma_{\mathbf{v}}|_{[0, T+\epsilon]}) = \mathrm{Ind}(\gamma_{\mathbf{v}_0}|_{[0, T+\epsilon]})\ge k.\]
    Consequently, for all $\mathbf{v} \in U$ we have $T - \epsilon \le \lambda_k(\mathbf{v}) \le T + \epsilon$. Since $\epsilon > 0$ is arbitrary, we see that $\lambda_k$ is continuous at $\mathbf{v}_0$ if $\lambda_k(\mathbf{v}_0) < \infty$.

    Now suppose $\lambda_k(\mathbf{v}_0) = \infty$. Then, for all $T$ large we have $\mathrm{Ind}(\gamma_{\mathbf{v}_0}|_{[0, T]}) \le k - 1$, and we may also assume that $\gamma_{\mathbf{v}_0}(T)$ is not a focal point of $N$ along $\gamma_{\mathbf{v}_0}$. Then, for some neighborhood $U \subset S(\nu)$ of $\mathbf{v}_0$ we have, $\mathrm{Ind}(\gamma_{\mathbf{v}}|_{[0, T]}) = \mathrm{Ind}(\gamma_{\mathbf{v}_0}|_{[0, T]}) \le k - 1$, and hence $\lambda_k(\mathbf{v}) \ge T$, for all $\mathbf{v} \in U$. Since $T$ is arbitrary, we have $\lambda_k$ is continuous at $\mathbf{v}_0$ with $\lambda_{k}(\mathbf{v}) = \infty$ as well. This concludes the proof.
\end{proof}

\section{Warner's Regularity of the Normal Exponential Map} \label{sec:warnerRegularity}
Suppose $N$ is a closed submanifold of a forward complete Finsler manifold $(M, F)$. In this section, we prove that $\mathcal{E} = \exp^\nu|_{\hat{\nu}}$ is regular in the sense of Warner \cite{Warner1965} while modifying the author's definition suitably. Then, we introduce the notion of regular tangent focal locus, and deduce a normal form for $\mathcal{E}$ near such points. Firstly, for any $\mathbf{v} \in \hat{\nu}$, let us denote the ray 
\begin{equation}\label{eq:ray}
	R_{\mathbf{v}} = \left\{ t \mathbf{v} \;\middle|\; t > 0 \right\} \subset \hat{\nu}.
\end{equation}
We also need the following definition.
\begin{defn}\label{defn:radialConvexity}
	A subset $C \subset \hat{\nu}$ is said to be \emph{radially convex} if $C$ is connected, and if for each $\mathbf{v} \in C$ the subset $\left\{ t \;\middle|\; \exp^\nu(t\mathbf{v}) \in C \right\} \subset (0, \infty)$ is connected.
\end{defn}

We can now prove the following, generalizing \cite[Theorem 4.5 and 5.1]{Warner1965} simultaneously.
\begin{theorem}\label{thm:normalExponentialWarnerRegular}
	Given a closed submanifold $N$ of a forward complete Finsler manifold $(M, F)$, the restricted normal exponential map $\mathcal{E} : \hat{\nu} \rightarrow M$ satisfies the following three regularity properties at each $\mathbf{v} \in \hat{\nu}$.
	\begin{enumerate}[label=(R\arabic*)]
		\item \label{R1} The derivative $d_{\mathbf{v}} \mathcal{E}$ is nonvanishing on $T_{\mathbf{v}}R_{\mathbf{v}}$.
		\item \label{R2} The map 
		\begin{align*}
			\mathcal{K}_{\mathbf{v}} &\longrightarrow T_{\mathcal{E}(\mathbf{v})}M/\mathop{\mathrm{Im}}d_{\mathbf{v}}\mathcal{E} \\
			\mathbf{x} &\longmapsto \dot J_{\mathbf{x}}(1)
		\end{align*}
		is a linear isomorphism.
		\item \label{R3} There exists a radially convex open neighborhood $\mathbf{v} \in U \subset \hat{\nu}$, such that for each $\mathbf{u} \in U$, the number of critical points (counted with multiplicity) of $\mathcal{E}$ on $R_{\mathbf{u}} \cap U$ is a constant independent of $\mathbf{u}$, and equals to the dimension of $\mathcal{K}_{\mathbf{v}} = \ker d_{\mathbf{v}}\mathcal{E}$.
	\end{enumerate}
\end{theorem}
\begin{proof}
	Consider the curve $\eta(t) = t\mathbf{v}$, so that $\eta(1) = \mathbf{v}$ and $T_{\mathbf{v}}R_{\mathbf{v}} = \mathrm{Span}\langle \dot \eta(1) \rangle$. Then, 
	\[d_{\mathbf{v}}\mathcal{E}(\dot\eta(1)) = \left. \frac{d}{dt} \right|_{t=1} (\mathcal{E} \circ \eta) = \left. \frac{d}{dt} \right|_{t=1} \exp^\nu(t \mathbf{v}) = \mathbf{v} \ne 0,\]
	which proves \autoref{R1}.

	For \autoref{R2}, pick a basis $\mathcal{K}_{\mathbf{v}} = \mathrm{Span}\langle \mathbf{x}_1, \dots , \mathbf{x}_k \rangle$ and extend it to a basis $T_{\mathbf{v}} \hat{\nu} = \mathrm{Span}\langle \mathbf{x}_1,\dots, \mathbf{x}_n \rangle$. It follows from \autoref{prop:tangentSpaceOfNormalConeBundle} that $d_{\mathbf{v}} \mathcal{E}(\mathbf{x}_i) = J_{\mathbf{x}_i}(1)$. But then \autoref{R2} follows immediately from \autoref{lemma:jacobiFrameAtFocalPoint} and \autoref{prop:tangentSpaceOfNormalConeBundle}.
 
	Lastly, we show \autoref{R3}. If $\mathbf{v}$ is not a focal locus of $N$ along $\gamma_{\widehat{\mathbf{v}}}$, then one can easily choose a radially convex open neighborhood $\mathbf{v} \in U \subset \hat{\nu}$ so that $d_{\mathbf{u}}\mathcal{E}$ is nonsingular for each $\mathbf{u} \in U$. The claim is then immediate. Otherwise, suppose $\mathbf{v}=T\widehat{\mathbf{v}}$ is a focal point, $T = F(\mathbf{v})$.  Since by \autoref{lemma:jacobiFrameAtFocalPoint} focal points along $\gamma_{\widehat{\mathbf{v}}}$ are discrete, we have some $\epsilon > 0$ so that $T \widehat{\mathbf{v}}$ is the only focal point of $N$ along $\gamma_{\widehat{\mathbf{v}}}$ for $|T - t| \le \epsilon$. Applying \autoref{prop:indexLocallyConstant} twice for both $\gamma_{\widehat{\mathbf{v}}}(T - \epsilon)$ and $\gamma_{\widehat{\mathbf{v}}}(T + \epsilon)$, we have a neighborhood $\widehat{\mathbf{v}} \in \hat{U} \subset S(\nu)$ so that for all $\widehat{\mathbf{u}}\in \hat{U}$ we have
	\begin{itemize}
		\item $\gamma_{\widehat{\mathbf{u}}}(T \pm \epsilon)$ are not focal points of $N$ along $\gamma_{\widehat{\mathbf{u}}}$,
		\item $\mathrm{Ind}(\gamma_{\widehat{\mathbf{u}}}|_{[0, T - \epsilon]}) = \mathrm{Ind}(\gamma_{\widehat{\mathbf{v}}}|_{[0, T - \epsilon]})$, and
		\item $\mathrm{Ind}(\gamma_{\widehat{\mathbf{u}}}|_{[0, T + \epsilon]}) = \mathrm{Ind}(\gamma_{\widehat{\mathbf{v}}}|_{[0, T + \epsilon]})$.
	\end{itemize}
	Consider the radially convex set 
	\[U = \left\{ t\widehat{\mathbf{u}} \;\middle|\; T - \epsilon < t < T + \epsilon, \; \widehat{\mathbf{u}} \in \hat{U}\right\} \subset \hat{\nu},\]
	which is open by an application of the invariance of domain. For any $\mathbf{u} \in U$, we have $R_{\mathbf{u}} \cap U = \left\{ t \widehat{\mathbf{u}} \;\middle|\; T - \epsilon < t < T + \epsilon \right\}$, where $\widehat{\mathbf{u}} = \frac{\mathbf{u}}{F(\mathbf{u})} \in \hat{U}$. By the Morse index theorem, the number of focal points counted with multiplicity on $R_{\mathbf{u}} \cap U$ then equals 
	\begin{align*}
		\mathrm{Ind}(\gamma_{\widehat{\mathbf{u}}}|_{[0, T + \epsilon]}) - \mathrm{Ind}(\gamma_{\widehat{\mathbf{u}}}|_{[0, T - \epsilon]}) &= \mathrm{Ind}(\gamma_{\widehat{\mathbf{v}}}|_{[T + \epsilon]}) - \mathrm{Ind}(\gamma_{\widehat{\mathbf{v}}}|_{[0, T - \epsilon]}) \\
		&= \dim \mathcal{K}_{T\widehat{\mathbf{v}}} \\
		&= \dim \mathcal{K}_{\mathbf{v}}.
	\end{align*}
	This shows \autoref{R3}, concluding the proof.
\end{proof}

\begin{remark}\label{rmk:warnerRegularity2}
	In \cite{Warner1965}, the author used the notion of second order tangent vectors to state the regularity condition \autoref{R2}. In the \autoref{sec:appendix}, we show that our condition \autoref{R2} is directly comparable to (\hyperref[warnerR2]{R2'}), which appears in \cite{Warner1965} (\autoref{prop:secondOrderTangent}). See also \cite[Section 8.1]{ArdoyGuijaro11} for a similar discussion.
\end{remark}

Let us observe the following easy consequence of \autoref{R1}.
\begin{prop}\label{prop:identificationKernelToUnitBundle}
    Suppose $\mathbf{v} \in \hat{\nu}$ is a tangent focal point, and let $0 \ne \mathbf{x} \in \mathcal{K}_{\mathbf{v}}$. Denote $\widehat{\mathbf{x}} \coloneqq d\pi(\mathbf{x}) \in T_{\widehat{\mathbf{v}}} S(\nu)$, where $\pi : \hat{\nu} \rightarrow S(\nu)$ is the map $\mathbf{u} \mapsto \frac{\mathbf{u}}{F(\mathbf{u})}$. Then, $J_{\widehat{\mathbf{x}}}(F(\mathbf{v})) = 0$.
\end{prop}
\begin{proof}
    Let $\tau : (-\epsilon, \epsilon) \rightarrow \hat{\nu}$ be a curve with $\tau(0) = \mathbf{v}$ and $\tau^\prime(0) = \mathbf{x}$. Then, $\hat{\tau}(s) \coloneqq \pi\left( \tau(s) \right) = \frac{\tau(s)}{F(\tau(s))}$ is a curve in $S(\nu)$ with $\hat{\tau}^\prime(0) = d\pi(\mathbf{x}) = \widehat{\mathbf{x}}$. Since $\mathcal{K}_{\mathbf{v}} \cap T_{\mathbf{v}} R_{\mathbf{v}} = 0$ by \autoref{R1}, and $\ker d_{\mathbf{v}}\pi = T_{\mathbf{v}} R_{\mathbf{v}}$, it follows that $\widehat{\mathbf{x}} = d\pi(\mathbf{x}) \ne 0$. Without loss of generality, we can then assume that $\tau : (-\epsilon, \epsilon) \rightarrow \hat{\nu}$ and $\hat{\tau} : (-\epsilon, \epsilon) \rightarrow S(\nu)$ are embedded curves, each intersecting a ray exactly once. Then, $\tau(s) = F(\tau(s))\hat{\tau}(s)$ holds. We now have the two $N$-geodesic variations 
    \[\hat{\Lambda}(s,t) \coloneqq \exp^\nu(t \hat{\tau}(s)), \qquad \Lambda(s,t) \coloneqq \exp^\nu(t \tau(s)) = \exp^\nu\left( tF(\tau(s)) \hat{\tau}(s)\right) = \hat{\Lambda}(s, t F(\tau(s))).\]
    In particular, 
    \[J_{\widehat{\mathbf{x}}}(F(\mathbf{v})) = J_{\widehat{\mathbf{x}}}\left( F(\tau(0)) \right) = \partial_s|_{s=0}\hat{\Lambda}(s, F(\tau(0)))  = \partial_s|_{s=0}\Lambda(s, 1) = J_{\mathbf{x}}(1) = 0.\qedhere\]
\end{proof}

\subsection{Regular Tangent Focal Locus}
In view of \autoref{R3}, one can define a certain regularity for tangent focal points.
\begin{defn}\label{defn:regularFocalLocus}
	A tangent focal point $\mathbf{v} \in \hat{\nu}$ is said to be \emph{regular} if there exists a radially convex open neighborhood $\mathbf{v} \in U \subset \hat{\nu}$ such that for each $\mathbf{u} \in U$, the map $\mathcal{E}$ has at most one critical point on $R_{\mathbf{u}} \cap U$, where the multiplicity necessarily equals that of $\mathbf{v}$. The set of regular tangent focal points, called the \emph{regular tangent focal locus}, will be denoted as $\mathcal{F}^{\textrm{reg}} \subset \mathcal{F}$. A tangent focal point that is not regular is called a \emph{singular} focal point, and the set of all singular tangent focal points will be denoted as $\mathcal{F}^{\textrm{sing}} \subset \mathcal{F}$.
\end{defn}

\begin{example}\label{example:multiplicityOneIsRegular}
	If $\mathbf{v} \in \hat{\nu}$ is a tangent focal point of multiplicity $1$, then by \autoref{R3}, we have a radially convex open set, say $\mathbf{v} \in U \subset \hat{\nu}$ such that for each $\mathbf{u} \in U$, the number of critical points, counted with multiplicity, of $\mathcal{E}$ on $R_{\mathbf{u}} \cap U$ equals exactly $1$. Thus, there is a unique focal point on $R_{\mathbf{u}} \cap U$, and hence $\mathbf{v} \in \mathcal{F}^{\textrm{reg}}$.
\end{example}

\begin{example}\label{example:constantFocalTimeIsRegular}
	Suppose for some $\mathbf{v} \in S(\nu)$, the focal multiplicity at the $k^{\text{th}}$ focal time is locally constant. That is, we have $\dim \mathcal{K}_{\lambda_k(\mathbf{u}) \mathbf{u}} = k_0$ for all $\mathbf{u} \in S(\nu)$ near $\mathbf{v}$. Fix some radially convex neighborhood, say, $U \subset \hat{\nu}$ of $\lambda_k(\mathbf{v})\mathbf{v}$, as in \autoref{R3}. Denote, $\hat{U} = \left\{ \widehat{\mathbf{u}} \coloneqq \frac{\mathbf{u}}{F(\mathbf{u})}  \;\middle|\; \mathbf{u} \in U \right\}$. By the continuity of $\lambda_k$ (\autoref{cor:focalTimeContinuous}), shrinking $U$ and $\hat{U}$ if necessary, for all $\widehat{\mathbf{u}} \in \hat{U}$ we may assume that  $\lambda_k(\widehat{\mathbf{u}})\widehat{\mathbf{u}} \in U$, and also they have focal multiplicity $k_0$. But then clearly $\lambda_k(\mathbf{v})\mathbf{v}$ is a regular focal point. 
\end{example}

We shall need the following technical result.
\begin{lemma}\label{lemma:localCoordinatesAtFocalPoint}
	Suppose $\mathbf{v} \in \hat{\nu}$ is a tangent focal point of multiplicity $k$. Then, there exist coordinate charts $(U, x^1,\dots ,x^n)$ and $(V, y^1,\dots ,y^n)$ around $\mathbf{v}$ and $\gamma_{\mathbf{v}}(1) = \mathcal{E}(\mathbf{v})$ respectively, satisfying 
	\[d_{t\mathbf{v}} \mathcal{E} \left( \left. \frac{\partial}{\partial x^i} \right|_{t\mathbf{v}} \right) = f_i(t) \left. \frac{\partial}{\partial y^i} \right|_{\mathcal{E}(t\mathbf{v})}, \quad \text{ for $t$ with $t\mathbf{v} \in U$,}\]
	where $f_i(t)$ are smooth functions satisfying the following.
	\begin{itemize}
		\item For $1 \le i \le k$, the only zeros of $f_i(t)$ are at $t = 1$, and furthermore, $\dot f_i(1) > 0$.
		\item $f_i(t) \ne 0$ for $k + 1 \le i \le n$.
	\end{itemize}
\end{lemma}
\begin{proof}
	Pick a basis $\left\{ \mathbf{x}_1, \dots, \mathbf{x}_k \right\}$ of $\mathcal{K}_{\mathbf{v}} = \ker d_{\mathbf{v}} \mathcal{E}$, and extend it to a basis $\left\{ \mathbf{x}_1,\dots, \mathbf{x}_n \right\}$ of $T_{\mathbf{v}} \hat{\nu}$. Let us write $\mathbf{v} = \ell \widehat{\mathbf{v}}$ for $\widehat{\mathbf{v}} \in S(\nu)$ and $\ell = F(\mathbf{v})$. As in the proof of \autoref{lemma:jacobiFrameAtFocalPoint}, we get $N$-Jacobi fields $J_1,\dots, J_n$ along the $N$-geodesic $\gamma_{\widehat{\mathbf{v}}}$. In particular, $J_1(\ell) = \dots = J_k(\ell) = 0$. Next, as in \autoref{eq:nonZeroMultipleOfJacobi}, in some neighborhood of $\ell$, we define the vector fields $Y_i$ for $1 \le i \le k$, and set $Y_i = J_i$ for $i > k$. As observed in \autoref{lemma:jacobiFrameAtFocalPoint}, $\left\{ Y_i(t) \right\}$ is basis for $T_{\gamma_{\widehat{\mathbf{v}}}(t)} M$ near $\ell$.
	
	Consider some $\alpha_i : (-\epsilon, \epsilon) \rightarrow \hat{\nu}$ such that $\alpha_i(0) = \mathbf{v}$ and $\dot \alpha_i(0) = \mathbf{x}_i$. Then, we have a vector field $A_i(t) \coloneqq \left. \frac{\partial}{\partial s} \right|_{s = 0} \left( \frac{t}{\ell} \alpha_i(s) \right) = \frac{t}{\ell} \dot\alpha_i(0) = \frac{t}{\ell} \mathbf{x}_i$ along the ray $R_{\widehat{\mathbf{v}}} = \left\{ t \widehat{\mathbf{v}} \;\middle|\; t > 0 \right\}$. It is immediate that 
	\[d_{t \widehat{\mathbf{v}}} \mathcal{E} \left( A_i(t) \right) = \left. \frac{\partial}{\partial s} \right|_{s=0} \mathcal{E}\left( \frac{t}{\ell} \alpha_i(s)\right) = J_i(t).\]
	By construction, $\left\{ A_i(\ell) \right\} = \left\{ \mathbf{x}_i \right\}$ is a basis at $T_{\mathbf{v}} \hat{\nu} = T_{\ell \widehat{\mathbf{v}}} \hat{\nu}$. Also, in a deleted neighborhood of $\ell$, we have $\left\{ d_{t \widehat{\mathbf{v}}}\mathcal{E} \left( A_i(t) \right) \right\} = \left\{ J_i(t) \right\}$ is a basis for $T_{\mathcal{E}(t\widehat{\mathbf{v}})} M$, and hence, $\left\{ A_i(t) \right\}$ is a basis of $T_{t \widehat{\mathbf{v}}} \hat{\nu}$ for a neighborhood of $\ell$.
	
	Now, recall that given any embedded regular curve $\alpha$ in an $n$-dimensional manifold, and a frame of vector fields $E_1,\dots , E_n$ along $\alpha$, one can construct a coordinate system $u^1,\dots u^n$ such that $\frac{\partial}{\partial u^i}|_{\alpha(t)} = E_i (t)$ holds (see \cite[Lemma 2.4]{Warner1965} for a proof). Applying this fact to the frame $\left\{ A_i \right\}$ along the ray, we have a coordinate system $(U, x^1, \dots, x^n)$ around $\mathbf{v}$ so that $\left. \frac{\partial}{\partial x^i} \right|_{t \mathbf{v}} = A_i(t \ell)$. Similarly, applying the same fact for the frame $Y_i$ along geodesic $\gamma_{\mathbf{v}}$, we get a coordinate chart $(V, y^1, \dots  , y^n)$ around $\gamma_{\mathbf{v}}(1)$ so that $\left. \frac{\partial}{\partial y^i} \right|_{\gamma_{\mathbf{v}}(t)} = Y_i(t\ell)$ holds. If we set $\tilde{f}_i(t) = t - \ell$ for $1 \le i \le k$ and $\tilde{f}_i(t) = 1$ for $i > k$, then near $\ell$ we have $J_i = \tilde{f}_i Y_i$ for $1 \le i \le n$. Hence, we get
	\[d_{t\mathbf{v}} \mathcal{E} \left( \left. \frac{\partial}{\partial x^i} \right|_{t\mathbf{v}}  \right) = d_{t\ell\widehat{\mathbf{v}}} \mathcal{E} \left( A_i(t \ell) \right) = J_i(t \ell) = \tilde{f}_i(t \ell) Y_i(t \ell) = f_i(t \ell) \left. \frac{\partial}{\partial y^i} \right|_{\gamma_{\mathbf{v}}(t)}.\]
	Setting $f_i(t) = \tilde{f}_i(t \ell)$ concludes the proof.
\end{proof}

We have the following result, which is a direct generalization of \cite[Theorem 3.1]{Warner1965}.
\begin{theorem}\label{thm:regularTangentFocalPointsSubmanifold}
	$\mathcal{F}^{\textrm{reg}}$ is open and dense in $\mathcal{F}$. Furthermore, $\mathcal{F}^{\textrm{reg}}$ is an embedded codimension $1$ submanifold of $\hat{\nu}$, with $T_{\mathbf{v}} \hat{\nu} = T_{\mathbf{v}} \mathcal{F}^{\textrm{reg}} \oplus T_{\mathbf{v}} R_{\mathbf{v}}$ for all $\mathbf{v} \in \mathcal{F}^{\textrm{reg}}$.
\end{theorem}
\begin{proof}
	Given any open set $U \subset \hat{\nu}$, let us denote $\mathcal{F}^{\textrm{reg}}_U = \mathcal{F}^{\textrm{reg}} \cap U$ and $\mathcal{F}^{\textrm{sing}}_U = \mathcal{F}^{\textrm{sing}} \cap U$. By definition of the regular focal locus, given any $\mathbf{v} \in \mathcal{F}^{\textrm{reg}}$ we have some (radially convex) open set $p \in U \subset \hat{\nu}$ such that every focal point in $U$ is regular. In other words, $\mathcal{F}^{\textrm{reg}}_U = \mathcal{F} \cap U$, proving that $\mathcal{F}^{\textrm{reg}}$ is open in $\mathcal{F}$.
	
	Next, let $\mathbf{v}_0 \in \mathcal{F}^{\textrm{sing}} = \mathcal{F} \setminus \mathcal{F}^{\textrm{reg}}$ be a singular focal point. By \autoref{R3}, we get an arbitrarily small radially convex neighborhood $\mathbf{v}_0 \in U \subset \hat{\nu}$ such that for each $\mathbf{u} \in U$, the number of focal points on the sub-ray $R_{\mathbf{u}} \cap U$ equals the multiplicity of $\mathbf{v}$. Since $\mathbf{v}$ is singular, there is some focal point $\mathbf{v}_1 \in U$ with at least two distinct focal points on $R_{\mathbf{v}_1} \cap U$. But then the focal multiplicity of $\mathbf{v}_1$ is strictly less than that of $\mathbf{v}_0$. If $\mathbf{v}_1$ is regular, we are done. Otherwise, we repeat this argument. Eventually, after some finitely many steps, we either get a regular focal point, or we get a focal point of order $1$. But every order $1$ focal point is automatically regular, as observed in \autoref{example:multiplicityOneIsRegular}. Thus, $U \cap \mathcal{F}^{\textrm{reg}} \ne \emptyset$, proving that $\mathcal{F}^{\textrm{reg}}$ is dense in $\mathcal{F}$.

	In order to show that $\mathcal{F}^{\textrm{reg}}$ is a codimension $1$ submanifold of $\hat{\nu}$, we locally realize $\mathcal{F}^{\textrm{reg}}$ as the zero set of some submersion. Fix $\mathbf{v} \in \mathcal{F}^{\textrm{reg}}$ with focal multiplicity, say, $k$. By \autoref{lemma:localCoordinatesAtFocalPoint}, we get coordinate charts $(U, \mathbf{x}^i)$ and $(V, \mathbf{y}^j)$, respectively around $\mathbf{v}$ and $\mathcal{E}(\mathbf{v}) = \gamma_{\mathbf{v}}(1)$. With these coordinates, we have the eigenvalues of $d\mathcal{E}|_U$ along the ray $R_{\mathbf{v}}$ as the smooth functions $f_1, \dots, f_n$. Denote $\Delta : U \rightarrow \mathbb{R}$ as the $(n-k+1)^{\text{th}}$-elementary symmetric polynomial in $n$ many variables, evaluated on the eigenvalues of $d\mathcal{E}|_{U}$. Even if $d\mathcal{E}|_U$ may not be diagonalizable everywhere, $\Delta$ can be expressed as the sum of all $(n-k+1)\times (n-k+1)$ principal minors of $d\mathcal{E}|_U$, and thus $\Delta$ is a well-defined smooth function. In particular, parametrizing the ray $R_{\mathbf{v}}$ by the curve $\sigma(t) = t \mathbf{v}$, we have
	\begin{equation}\label{eq:Delta}
		\Delta(\mathbf{u}) = \Delta(t \mathbf{v}) = \sum_{1 \le i_1 < \dots < i_{k-1} \le n} f_1(t) \dots \widehat{f_{i_1}(t)} \dots \widehat{f_{i_{k-1}}(t)} \dots f_n(t), \quad \mathbf{u} = t \mathbf{v} \in U \cap R_{\mathbf{v}}.
	\end{equation}
	It follows from \autoref{lemma:localCoordinatesAtFocalPoint} that 
	\[f_1(1) = \dots = f_k(1) = 0, \quad f_{k+1}(1) \ne 0, \dots , f_n(1) \ne 0,\]
	and moreover, we have the radial derivatives
	\[\dot \sigma(1) \left( f_i \right) = d f_i \left( \dot \sigma(1) \right) = \dot f_i(1) > 0, \quad 1 \le i \le k.\]
	Consequently, we observe $\Delta(\mathbf{v}) = 0$, and a simple computation using the Leibniz rule gives us
	\[\dot\sigma(1)\left( \Delta \right) = \left( \sum_{i=1}^k \dot f_i(1) \right) f_{k+1}(1)\dots f_n(1) \ne 0. \]
	Shrinking $U$ as necessary, while keeping it radially convex, we may assume that the radial derivative of  $\Delta$ along each of the rays intersecting $U$ is non-vanishing on $U$. In particular, $\Delta$ is a submersion on $U$. Clearly, the radial directions are transverse to $\Delta^{-1}(0)$, i.e., we have the decomposition
	\[T_{\mathbf{u}} \hat{\nu} = T_{\mathbf{u}}\left( \Delta^{-1}(0) \right) \oplus T_{\mathbf{u}} R_{\mathbf{u}}, \qquad \mathbf{u} \in U.\]
	We now show that $\Delta^{-1}(0) = \mathcal{F}^{\textrm{reg}}_U$. Indeed, for any $\mathbf{u} \in \mathcal{F}^{\textrm{reg}}_U$ we have the multiplicity equals $k$, and thus $d_{\mathbf{u}}\mathcal{E}$ has rank $n - k$. Consequently, at least one of the principal minors of $d\mathcal{E}|_U$ in each of the product terms appearing in the sum of products $\Delta$ vanishes at $\mathbf{u}$. Thus, $\mathcal{F}^{\textrm{reg}}_U \subset \Delta^{-1}(0)$. For the converse, suppose $\Delta(\mathbf{u}) = 0$ for some $\mathbf{u} \in U$. Pick the \emph{unique} focal point $\mathbf{u}_0 \in R_{\mathbf{u}} \cap U$. We have $\Delta(\mathbf{u}_0) = 0 = \Delta(\mathbf{u})$. Also, it follows from the radial convexity of $U$ that the radial line joining $\mathbf{u}$ and $\mathbf{u}_0$ is contained in $U$, and moreover, the radial derivative of $\Delta$ is nonvanishing on this line. Then, by the mean value theorem, we must have $\mathbf{u} = \mathbf{u}_0 \in \mathcal{F}^{\textrm{reg}}_U$. Thus, $\mathcal{F}^{\textrm{reg}}_U = \Delta^{-1}(0)$. The proof then follows.
\end{proof}

\subsection{Components of the Regular Tangent Focal Locus}
Let us now look at the components of the regular tangent focal locus in more detail. We introduce the following definition.
\begin{defn}\label{defn:focalPointType}
	For each tangent focal point $\mathbf{v} \in \mathcal{F}$, we associate a tuple of integers $(k, i)$, where
	\begin{itemize}
		\item $k = \dim \mathcal{K}_{\mathbf{v}}$ is the focal multiplicity, and
		\item $i \ge 1$ is the least integer such that $\mathbf{v} = \lambda_i(\widehat{\mathbf{v}})\widehat{\mathbf{v}}$. 
	\end{itemize}
	We say $\mathbf{v}$ is a tangent focal point of \emph{type $(k, i)$}.
\end{defn}
It is clear that for $\mathbf{v} \in \mathcal{F}$ of type $(k, i)$ we have $F(\mathbf{v}) = \lambda_i(\widehat{\mathbf{v}}) = \dots = \lambda_{i+k - 1}(\widehat{\mathbf{v}})$, and hence, $\mathbf{v} = \lambda_j(\widehat{\mathbf{v}})\widehat{\mathbf{v}}$ for $i \le j < i + k$. Furthermore, $\lambda_{i + k - 1}(\widehat{\mathbf{v}}) \lneq \lambda_{i + k}(\widehat{\mathbf{v}})$, and if $i \ge 2$ then $\lambda_{i-1}(\widehat{\mathbf{v}}) \lneq \lambda_{i}(\widehat{\mathbf{v}})$. We now characterize the connected components of $\mathcal{F}^{\textrm{reg}}$.

\begin{theorem}\label{thm:regularFocalLocusComponent}
	For a connected component $\mathcal{C} \subset \mathcal{F}^{\textrm{reg}}$, the following holds true.
	\begin{enumerate}[label=(\arabic*)]
		\item \label{thm:regularFocalLocusComponent:description} Each $\mathbf{v} \in \mathcal{C}$ has a constant type, say, $(k_0, i_0)$.

		\item \label{thm:regularFocalLocusComponent:boundary} The component $\mathcal{C}$ is an open subset of $\mathcal{F}^{\textrm{reg}}$, and hence, a codimension $1$ submanifold of $\hat{\nu}$. Moreover, the topological boundary $\partial \mathcal{C} \coloneqq \overline{\mathcal{C}} \setminus \mathring{\mathcal{C}} = \overline{\mathcal{C}} \setminus \mathcal{C}$ consists of singular focal points, where $\mathring{\mathcal{C}} = \mathcal{C}$ is the interior of the submanifold, and the closure $\overline{\mathcal{C}}$ is taken in $\hat{\nu}$.
		
		\item \label{thm:regularFocalLocusComponent:boundarySharing} Each $\mathbf{v} \in \partial \mathcal{C}$ is in the boundary of at least one more distinct component. Furthermore, the set 
		\[\mathcal{A}_{\mathbf{v}} \coloneqq \left\{ (k,i) \;\middle|\; \mathbf{v} \in \partial \mathcal{C}^\prime, \; \text{each element in the component $\mathcal{C}^\prime$ has the type $(k, i)$} \right\}\]
		is finite.

		\item \label{thm:regularFocalLocusComponent:embedding} The projection map $\pi : \hat{\nu} \rightarrow S(\nu)$ given by $\pi(\mathbf{v}) = \frac{\mathbf{v}}{F(\mathbf{v})}$ restricts to an embedding on $\mathcal{C}$, and $\widehat{\mathcal{C}} \coloneqq \pi(\mathcal{C})$ is open in $S(\nu)$.
	\end{enumerate}
\end{theorem}
\begin{proof}
	Suppose $\mathbf{v}_0 \in \mathcal{C} \subset \mathcal{F}^{\textrm{reg}}$ has the type $(k_0, i_0)$. Consider the set 
	\[X \coloneqq \left\{ \mathbf{u} \in \mathcal{F}^{\textrm{reg}} \;\middle|\; \text{$\mathbf{u}$ has type $(k_0, i_0)$} \right\}.\]
	Clearly $X \ne \emptyset$ as $\mathbf{v}_0 \in X$. Let us show that $X$ is both open and closed in $\mathcal{F}^{\textrm{reg}}$, whence it is a union of connected components. For any $\mathbf{v} \in X$, get a radially convex open set, say, $\mathbf{v} \in U \subset \hat{\nu}$ as in \autoref{R3}. As $\mathbf{v} \in \mathcal{F}^{\textrm{reg}}$, we may assume that for each $\mathbf{u} \in U$, there is exactly one focal point on $R_{\mathbf{u}} \cap U$, which necessarily has multiplicity $k_0$. Denote $\hat{U} = \left\{ \widehat{\mathbf{u}} \;\middle|\; \mathbf{u} \in U \right\} \subset S(\nu)$. By the continuity of the focal time maps (\autoref{cor:focalTimeContinuous}), shrinking $U$ and $\hat{U}$ if necessary, we see that $\lambda_j(\widehat{\mathbf{u}})\widehat{\mathbf{u}} \in U$ for $i_0 \le j < i_0 + k_0$. But then each focal point on $R_{\mathbf{u}} \cap U$ has the type $(k_0, i_0)$. Thus, $\mathbf{v} \in \mathcal{F}^{\textrm{reg}} \cap U \subset X$, showing that $X$ is open in $\mathcal{F}^{\textrm{reg}}$. Next, suppose $\mathbf{v}_i \in X$ is a sequence converging to some $\mathbf{v} \in \mathcal{F}^{\textrm{reg}}$. As any neighborhood of $\mathbf{v}$ contains some $\mathbf{v}_i$ of multiplicity $k_0$, it follows by the regularity that $\mathbf{v}$ has multiplicity $k_0$ as well. Also, by the continuity of the focal time maps, we have 
	\[\lambda_j(\widehat{\mathbf{v}}) = \lim_i \lambda_j(\widehat{\mathbf{v}}_i) = \lim_i F(\mathbf{v}_i) = F(\mathbf{v}), \quad i_0 \le j < i_0 + k_0.\]
	In particular, $\mathbf{v}$ is of type $(k_0, i_0)$. Hence, $\mathbf{v} \in X$, proving that $X$ is closed in $\mathcal{F}^{\textrm{reg}}$. Thus, $X$ is union of connected subsets of $\mathcal{F}^{\textrm{reg}}$. Since $\mathbf{v}_0 \in \mathcal{C} \cap X$, we have $\mathcal{C} \subset X$, which proves \autoref{thm:regularFocalLocusComponent:description}.

	As $\mathcal{F}^{\textrm{reg}}$ is a codimension $1$ submanifold of $\hat{\nu}$ (\autoref{thm:regularTangentFocalPointsSubmanifold}), any connected component must be open. Thus, $\mathcal{C}$ is open in $\mathcal{F}^{\textrm{reg}}$, and consequently, a codimension $1$ submanifold of $\hat{\nu}$. Now, note that $\mathcal{C} \subset \mathcal{F} \Rightarrow \overline{\mathcal{C}} \subset \overline{\mathcal{F}} = \mathcal{F}$, since $\mathcal{F}$ is closed in $\hat{\nu}$, being the set of critical points of $\mathcal{E}$. Now, the connected component $\mathcal{C}$ is closed in $\mathcal{F}^{\textrm{reg}}$, and hence, a limit point in $\overline{\mathcal{C}} \setminus \mathcal{C}$ cannot be regular. Thus, $\partial \mathcal{C} \subset \mathcal{F}^{\textrm{sing}}$, proving \autoref{thm:regularFocalLocusComponent:boundary}.
	
	Suppose $\mathbf{v} \in \partial \mathcal{C}$ has the type $(k^\prime, i^\prime)$. Choose some neighborhood $\mathbf{v} \in U \subset \hat{\nu}$ as in \autoref{R3}. By the continuity of $\lambda_j$ (\autoref{cor:focalTimeContinuous}), shrinking $U$ if necessary, it follows that for each $i^\prime \le j < i^\prime + k^\prime$ and for each ray that intersects $U$ must intersect in a tangent focal point $\mathbf{u} \in U$ such that $\mathbf{u} = \lambda_j(\widehat{\mathbf{u}}) \widehat{\mathbf{u}}$. But then by \autoref{R3}, every tangent focal point in $U$ are precisely of this form. Consequently, every tangent focal point in $U$ must have a type, say, $(\kappa, \iota)$ satisfying 
	\[\kappa \le k^\prime, \qquad i^\prime \le \iota \le \iota + \kappa - 1 \le i^\prime + k^\prime - 1.\]
	Since there are only finitely many such tuples, and since by \autoref{thm:regularFocalLocusComponent:description} every connected component determines one such tuple, it follows that the set $\mathcal{A}_{\mathbf{v}}$ as in \autoref{thm:regularFocalLocusComponent:boundarySharing} is finite. Let us now suppose $U \cap \mathcal{F}^{\textrm{reg}} = U \cap \mathcal{C}$, i.e., the only regular tangent focal points near $\mathbf{v}$ are from $\mathcal{C}$. Since $\mathbf{v} \in \mathcal{F}^{\textrm{sing}}$, there is some ray that intersects $U$ in at least $2$ distinct tangent focal points, and all of them then have multiplicity strictly less than that of $\mathbf{v}$. In particular, they cannot all have the same type. Hence, not all of them can be regular, as the regular ones must belong to $\mathcal{C}$ whence they must have the same type by \autoref{thm:regularFocalLocusComponent:description}. Choose some singular focal point, say, $\mathbf{v}_1 \in U$ on the ray. But then, after a finitely many steps, we get some $\mathbf{v}_l \in U$ with multiplicity $1$, which is regular (\autoref{example:multiplicityOneIsRegular}). This is a contradiction. Hence, $U$ must intersect at least two distinct components. This proves \autoref{thm:regularFocalLocusComponent:boundarySharing}.

	Next, for the projection map $\pi : \hat{\nu} \rightarrow S(\nu)$, observe that the kernel $d\pi$ is along the rays, i.e., $\ker d_{\mathbf{v}} \pi = T_{\mathbf{v}} R_{\mathbf{v}}$ for all $\mathbf{v} \in \hat{\nu}$. Then by \autoref{thm:regularTangentFocalPointsSubmanifold}, $\ker d \pi$ is transverse to the component $\mathcal{C}$ at each point. Consequently, $\pi|_{\mathcal{C}}$ is a submersion, and hence, $\widehat{\mathcal{C}} = \pi(\mathcal{C})$ is open in $S(\nu)$. Now, $\pi|_{\mathcal{C}}$ is an immersion as well, since $\dim S(\nu) = n - 1 = \dim \mathcal{C}$. Let us show that $\pi$ is injective. Suppose $\mathbf{v} \coloneqq \pi(\mathbf{v}_1) = \pi(\mathbf{v}_2)$ for some $\mathbf{v}_1, \mathbf{v}_2 \in \mathcal{C}$. Since both $\mathbf{v}_j$ has type $(k, i)$, we have $\mathbf{v}_1 = F(\mathbf{v})\mathbf{v} = \mathbf{v}_2$. But then, by the implicit function theorem, $\pi|_{\mathcal{C}}$ is an embedding, proving \autoref{thm:regularFocalLocusComponent:embedding}. This concludes the proof.
\end{proof}

\begin{defn}\label{defn:jthTangentFocalLocus}
	For $j \ge 1$, define the \emph{$j^{\text{th}}$-tangent focal locus} as 
	\[\mathcal{F}_j \coloneqq \left\{ \lambda_j(\mathbf{v})\mathbf{v} \;\middle|\; \mathbf{v} \in S(\nu), \quad \lambda_j(\mathbf{v}) < \infty \right\},\]
	where $\lambda_j : S(\nu) \rightarrow (0, \infty]$ is the $j^{\text{th}}$-focal time map. Denote $\mathcal{F}^{\textrm{reg}}_j \coloneqq \mathcal{F}_j \cap \mathcal{F}^{\textrm{reg}}$ as the \emph{$j^{\text{th}}$ regular tangent focal locus}.
\end{defn}

It is immediate that if $\mathbf{v} \in \mathcal{F}$ has type $(k, i)$ then $\mathbf{v} \in \mathcal{F}_j$ for $i \le j < i + k$. The next result generalizes \cite[Theorem A]{Bishop77}.

\begin{theorem}\label{thm:regularJthTangentFocalLocus}
	For $j \ge 1$, we have $\mathcal{F}^{\textrm{reg}}_j$ is open and closed in $\mathcal{F}^{\textrm{reg}}$, and thus $\mathcal{F}^{\textrm{reg}}_j$ is the disjoint union of connected components of $\mathcal{F}^{\textrm{reg}}$. Moreover, $\mathcal{F}^{\textrm{reg}}_j$ is open and dense in $\mathcal{F}_j$, and $\mathcal{F}_j = \overline{\mathcal{F}^{\textrm{reg}}_j}$.
\end{theorem}
\begin{proof}
	Suppose $\mathbf{v} \in \mathcal{F}^{\textrm{reg}}_j$ has type $(k, i)$, whence $i \le j < i + k$. Then, for the connected component $\mathcal{C} \subset \mathcal{F}^{\textrm{reg}}$ through $\mathbf{v}$, each tangent focal point $\mathbf{u} \in \mathcal{C}$ has the type $(k, i)$ by \autoref{thm:regularFocalLocusComponent} \autoref{thm:regularFocalLocusComponent:description}. In particular, $\mathbf{u} = \lambda_j(\widehat{\mathbf{u}})\widehat{\mathbf{u}}$ for all $\mathbf{u} \in \mathcal{C}$, which implies $\mathcal{C} \subset \mathcal{F}^{\textrm{reg}}_j$. This shows $\mathcal{F}^{\textrm{reg}}_j \subset \mathcal{F}^{\textrm{reg}}$ is open as $\mathcal{C}$ is open in $\mathcal{F}^{\textrm{reg}}$. Next, suppose $\mathbf{v}_i \in \mathcal{F}^{\textrm{reg}}_j$ converges to some $\mathbf{v} \in \mathcal{F}^{\textrm{reg}}$. Then, 
	\[F(\mathbf{v}) = \lim_i F(\mathbf{v}_i) = \lim_i \lambda_j(\widehat{\mathbf{v}}_i) = \lambda_j(\widehat{\mathbf{v}}),\]
	which implies $\mathbf{v} = \lambda_j(\widehat{\mathbf{v}})\widehat{\mathbf{v}}$. Thus, $\mathbf{v} \in \mathcal{F}^{\textrm{reg}}_j$, proving that $\mathcal{F}^{\textrm{reg}}_j$ is closed in $\mathcal{F}^{\textrm{reg}}$. Clearly, $\mathcal{F}^{\textrm{reg}}_j$ is union of connected components of $\mathcal{F}^{\textrm{reg}}$.

	In order to prove the denseness, let $\mathbf{v} \in \mathcal{F}^{\textrm{sing}}_j = \mathcal{F}_j \setminus \mathcal{F}^{\textrm{reg}}_j$. Suppose $\mathbf{v}$ has the type $(k^\prime, i^\prime)$. Choose some neighborhood $\mathbf{v} \in U \subset \hat{\nu}$ as in \autoref{R3}. By the continuity of $\lambda_j$ (\autoref{cor:focalTimeContinuous}), shrinking $U$ as necessary, we may assume that each ray intersects $U$ at a unique tangent focal point $\mathbf{u}$ satisfying $\mathbf{u} \in \mathcal{F}_j$. As $\mathbf{v}$ is singular, some ray must intersect $U$ in at least two distinct tangent focal points, and exactly one of them, say, $\mathbf{v}_1$ is in $\mathcal{F}_j$. Also, the multiplicity of $\mathbf{v}_1$ is strictly less than that of $\mathbf{v}$. If $\mathbf{v}_1 \in \mathcal{F}^{\textrm{reg}}$, then we are done. Otherwise, after finitely many steps, we get some $\mathbf{v}_l \in \mathcal{F}_j$, such that either $\mathbf{v}_l \in \mathcal{F}^{\textrm{reg}}$, or $\mathbf{v}_l$ has multiplicity $1$, whence, $\mathbf{v}_l \in \mathcal{F}^{\textrm{reg}}$ as noted in \autoref{example:multiplicityOneIsRegular}. Thus, $\mathcal{F}^{\textrm{reg}}_j$ is dense in $\mathcal{F}_j$. If we have $\mathbf{v}_i \in \mathcal{F}_j$ converging to some $\mathbf{v} \in \hat{\nu}$, similarly as above, we have $F(\mathbf{v}) = \lim F(\mathbf{v}_i) = \lim \lambda_j(\widehat{\mathbf{v}}_i) = \lambda_j(\widehat{\mathbf{v}})$, showing that $\mathbf{v} \in \mathcal{F}_j$. Thus, $\mathcal{F}_j$ is closed in $\hat{\nu}$, and hence, $\mathcal{F}_j =\overline{\mathcal{F}^{\textrm{reg}}_j}$, concluding the proof.
\end{proof}

\subsection{Smoothness of the Focal Time Maps}
The discussion so far lets us identify an open set of $S(\nu)$ on which $j^{\text{th}}$-focal time map $\lambda_j : S(\nu) \rightarrow (0, \infty]$ is smooth. First, we observe the following.
\begin{prop}\label{prop:focalTimeSmooth}
	Let $\mathcal{C}$ be a connected component of $\mathcal{F}^{\textrm{reg}}$. If the tangent focal points in $\mathcal{C}$ have the type $(k_0, i_0)$, then for each $i_0 \le j < i_0 + k_0$, the focal time map $\lambda_j$ is smooth on $\widehat{\mathcal{C}} = \pi(\mathcal{C}) \subset S(\nu)$. For any $\mathbf{v} \in \mathcal{C}$ and $\mathbf{x} \in T_{\widehat{\mathbf{v}}} S(\nu)$, we have 
	\begin{equation}\label{eq:derivativeJthFocalTime}
		d_{\widehat{\mathbf{v}}}\lambda_j(\mathbf{x}) = \frac{g_{\mathbf{v}}\left( \mathbf{v}, A_{\mathbf{v}}(d\pi(\mathbf{x})) \right)}{\sqrt{\lambda_j(\widehat{\mathbf{v}})}},
	\end{equation}
	where $\pi : \hat{\nu} \rightarrow N$ is the projection map, and $A_{\mathbf{v}}$ is the shape operator of $N$ in the direction $\mathbf{v}$.
\end{prop}
\begin{proof}
	By \autoref{thm:regularFocalLocusComponent} \autoref{thm:regularFocalLocusComponent:embedding}, we have $\hat{\pi} : \mathcal{C} \rightarrow \widehat{\mathcal{C}}$ is a diffeomorphism, and $\widehat{\mathcal{C}}$ is an open submanifold of $S(\nu)$. It follows that 
	\[\lambda_j(\mathbf{x}) = F\left( \lambda_j(\mathbf{x})\mathbf{x} \right) = \left( F \circ \left( \hat{\pi}|_{\mathcal{C}} \right)^{-1} \right)(\mathbf{x}), \qquad \mathbf{x} \in \widehat{\mathcal{C}}, \quad  i_0 \le j < i_0 + k_0.\]
	That is, $\lambda_j = F \circ \left( \hat{\pi}|_{\mathcal{C}} \right)^{-1}$ on $\widehat{\mathcal{C}}$ for $i_0 \le j < i_0 + k_0$. As $F$ is smooth on $\hat{\nu}$, we see that $\lambda_j$ is smooth on $\widehat{\mathcal{C}}$ for $i_0 \le j < i_0 + k_0$. 

	In order to compute the derivative, fix some $\mathbf{v} \in \mathcal{C}$ and some $\mathbf{x} \in T_{\widehat{\mathbf{v}}}S(\nu)$. We have a curve $\sigma : (-\epsilon,\epsilon) \rightarrow \mathcal{C}$ passing through $\mathbf{v}$, such that $\dot{\hat{\sigma}}(0) = \mathbf{x}$, where $\hat{\sigma}(s) = \frac{\sigma(s)}{F(\sigma(s))}$. Denote the curve $c = \pi \circ \sigma$ in $N$, so that $\sigma \in \Gamma c^* \hat{\nu}$. In particular, $\dot c(0) = d\pi(\mathbf{x})$. Now, consider the function $\tau(s) = F(\sigma(s))^2 = g_{\sigma(s)}\left( \sigma(s), \sigma(s) \right)$. Then, from \autoref{defn:covariantDerivative} we have 
	\[\dot \tau(0) = \left.\frac{d}{dt} \right|_{s=0} g_\sigma(\sigma, \sigma) = 2 g_{\mathbf{v}}\left( \mathbf{v}, D^\sigma_c \sigma(0) \right).\]
	Consider the $N$-geodesic variation $\Lambda : (-\epsilon, \epsilon) \times [0,1] \rightarrow M$ given by $\Lambda(s, t) = \exp^\nu(t \sigma(s))$. Denote the $N$-Jacobi field $J(t) = \frac{\partial}{\partial s}|_{s = 0}\Lambda(s, t)$. Set $\gamma_{s_0}(t) = \Lambda(s, t)$ and $\beta_{t_0}(s) = \Lambda(s,t_0)$. Note that $\dot\gamma_s(0) = \sigma(s)$, and $\beta_0(s) = c(s)$. We have the nonvanishing vector field $V(s,t) = \dot \gamma_s(t)$. Now, it follows from \cite[Eq (3)]{Javaloyes2015} that $D^V_{\gamma_s}\dot \beta_t = D^V_{\beta_t}\dot\gamma_s$. Evaluating at $s = 0 = t$ we get 
	\[D^\sigma_c \sigma(0) = D^V_{\beta_t}\dot\gamma_s|_{(0,0)} = D^V_{\gamma_s}\dot\beta_t|_{(0,0)} = \dot J(0).\]
	Then, it follows from \autoref{eq:NJacobiEquation} that,
	\[\dot \tau(0) = 2 g_{\mathbf{v}} \left( \mathbf{v}, \dot J(0) \right) = 2 g_{\mathbf{v}} \left( \mathbf{v}, A_{\mathbf{v}} (J(0)) \right) = 2 g_{\mathbf{v}}(\mathbf{v}, A_{\mathbf{v}}(d\pi(\mathbf{x}))).\]
	Since $\lambda_j > 0$, we have 
	\[d_{\widehat{\mathbf{v}}}\lambda_j(\mathbf{x}) = \left.\frac{d}{ds}\right|_{s=0}\lambda_j(\hat{\sigma}(s)) = \left.\frac{d}{ds}\right|_{s=0} F(\sigma(s)) = \frac{1}{2} \left( \tau(0) \right)^{-\frac{1}{2} } \dot\tau(0) = \frac{g_{\mathbf{v}}\left( \mathbf{v}, A_{\mathbf{v}}(d\pi(\mathbf{x})) \right)}{\sqrt{\lambda_j(\widehat{\mathbf{v}})}}.\]
	This concludes the proof.
\end{proof}

In particular, if for some $\mathbf{v} \in S(\nu)$ we have multiplicity of $\lambda_j(\mathbf{v})\mathbf{v}$ is constant near $\mathbf{v}$, then $\mathbf{v}$ is regular as observed in \autoref{example:constantFocalTimeIsRegular}, and hence, $\lambda_j$ is smooth near $\mathbf{v}$. This was noted in \cite{Itoh2001}. Similarly, $\lambda_j$ is differentiable near $\mathbf{v}$ if $\dim \mathcal{K}_{\lambda_j(\mathbf{v})\mathbf{v}} = 1$, which makes it regular by \autoref{example:multiplicityOneIsRegular}. This fact was used in \cite{ItohTanaka98}. Let us now identify a more general subset of $S(\nu)$ on which $\lambda_j$ is smooth. Denote,
\[\mathcal{R}_j \coloneqq \left\{ \mathbf{v} \in S(\nu) \;\middle|\; \lambda_j(\mathbf{v}) < \infty, \; \lambda_j(\mathbf{v})\mathbf{v} \in \mathcal{F}^{\textrm{reg}} \right\}.\]
\begin{theorem}\label{thm:smoothnessOfFocalTime}
	The projection $\pi : \hat{\nu} \rightarrow S(\nu)$ restricts to an embedding on $\mathcal{F}^{\textrm{reg}}_j$ with $\mathcal{R}_j = \pi(\mathcal{F}^{\textrm{reg}}_j)$. Furthermore, $\mathcal{R}_j$ is open in $S(\nu)$, and $\lambda_j$ is smooth on $\mathcal{R}_j$. If $\lambda_j$ is finite, then $\mathcal{R}_j$ is dense in $S(\nu)$ as well.
\end{theorem}
\begin{proof}
	It is clear from the definition that $\mathcal{R}_j = \pi(\mathcal{F}^{\textrm{reg}}_j)$. By \autoref{thm:regularJthTangentFocalLocus}, we can write $\mathcal{F}^{\textrm{reg}}_j$ as disjoint union of connected components of $\mathcal{F}^{\textrm{reg}}$. For each such component $\mathcal{C} \subset \mathcal{F}^{\textrm{reg}}_j$, it follows from \autoref{thm:regularFocalLocusComponent} that  $\pi|_{\mathcal{C}}$ is an embedding, with $\widehat{\mathcal{C}} \coloneqq \pi(\mathcal{C})$ open in $S(\nu)$. Suppose, for some $\mathbf{v}_1, \mathbf{v}_2 \in \mathcal{F}_j$ we have, $\pi(\mathbf{v}_1) = \mathbf{v} = \pi(\mathbf{v}_2)$. But then, 
	\[\mathbf{v}_1 = \lambda_j(\widehat{\mathbf{v}}_1) \widehat{\mathbf{v}}_1 = \lambda_j (\widehat{\mathbf{v}})\widehat{\mathbf{v}} = \lambda_j(\widehat{\mathbf{v}}_2)\widehat{\mathbf{v}}_2=\mathbf{v}_2.\]
	Thus, $\pi$ is injective on $\mathcal{F}_j$, and in particular, $\pi|_{\mathcal{F}^{\textrm{reg}}_j} : \mathcal{F}^{\textrm{reg}}_j \rightarrow S(\nu)$ is then an embedding, with image $\mathcal{R}_j$. Clearly, $\mathcal{R}_j$ is open in $S(\nu)$ being the union of open sets $\widehat{\mathcal{C}}$, and $\lambda_j$ is then smooth on $\mathcal{R}_j$. If $\lambda_j$ is assumed to be finite, then it follows that $S(\nu) = \pi\left( \mathcal{F}_j \right)$. Since $\mathcal{F}^{\textrm{reg}}_j$ is dense in $\mathcal{F}_j$ by \autoref{thm:regularJthTangentFocalLocus}, it follows that $\mathcal{R}_j$ is dense in $S(\nu)$. This concludes the proof.
\end{proof}

\subsection{Local Form of the Normal Exponential Map}
Let us now fix a connected component, say, $\mathcal{C} \subset \mathcal{F}^{\textrm{reg}}$. By \autoref{thm:regularFocalLocusComponent} \autoref{thm:regularFocalLocusComponent:boundary}, we have $\mathcal{C}$ is open in $\mathcal{F}^{\textrm{reg}}$, and a codimension $1$ submanifold of $\hat{\nu}$. Furthermore, the focal multiplicities of points in $\mathcal{C}$ are constant, say, $k$. At any $\mathbf{v} \in \mathcal{C}$, we have two subspaces of $T_{\mathbf{v}}\hat{\nu}$, namely, $T_{\mathbf{v}} \mathcal{C} = T_{\mathbf{v}} \mathcal{F}^{\textrm{reg}}$ and $\mathcal{K}_{\mathbf{v}} = \ker d_{\mathbf{v}} \mathcal{E}$. Let us denote, 
\[\mathcal{N}_{\mathbf{v}} \coloneqq \mathcal{K}_{\mathbf{v}} \cap T_{\mathbf{v}}\mathcal{F}^{\textrm{reg}}, \quad \mathbf{v} \in \mathcal{C}.\]
We have $\dim \mathcal{N}_{\mathbf{v}}$ is either $k-1$ or $k$, in which case $\mathcal{K}_{\mathbf{v}} \subset T_{\mathbf{v}} \mathcal{F}^{\textrm{reg}}$. Based on $\dim \mathcal{N}_{\mathbf{v}}$, we have a decomposition 
\begin{equation}\label{eq:regularFocalDecomposiiton}
	\mathcal{C} = \mathcal{C}(k) \sqcup \mathcal{C}(k-1).
\end{equation}
Since the restriction $\mathcal{E}|_{\mathcal{C}}$ attains the maximum possible rank $n - (k-1)$ on $\mathcal{C}(k-1)$, we have $\mathcal{C}(k-1)$ is an open submanifold of $\mathcal{C}$. Also, since the union $\cup_{\mathbf{v} \in \mathcal{C}(k-1)} \mathcal{N}_{\mathbf{v}}$ is precisely the kernel of $d\left( \mathcal{E}|_{\mathcal{C}(k-1)} \right)$ and has constant rank, it is a rank $(k-1)$ involutive distribution on $\mathcal{C}(k-1)$. By a similar argument, the union $\cup_{\mathbf{v} \in \Int \mathcal{C}(k)} \mathcal{N}_{\mathbf{v}} = \cup_{\mathbf{v} \in \Int \mathcal{C}(k)} \mathcal{K}_{\mathbf{v}}$ is also an involutive distribution of rank $k$ on the open submanifold $\Int \mathcal{C}(k) \subset \mathcal{C}$. We have the following result, generalizing \cite[Theorem 3.2]{Warner1965}.

\begin{theorem}\label{thm:kernelContainedInTangent}
	For any $\mathbf{v} \in \mathcal{F}^{\textrm{reg}}$ with focal multiplicity $k \ge 2$, we have $\mathcal{K}_{\mathbf{v}} \subset T_{\mathbf{v}} \mathcal{F}^{\textrm{reg}}$, and thus $\mathcal{C}(k-1) = \emptyset$ for the component $\mathcal{C} \subset \mathcal{F}^{\textrm{reg}}$ containing $\mathbf{v}$.
\end{theorem}

\begin{remark}\label{rmk:expMapsLeavesToPoint}
	For a component $\mathcal{C}\subset \mathcal{F}^{\textrm{reg}}$ consisting of tangent focal points of multiplicity $k$, each leaf of the foliation associated with the distribution $\cup_{\mathbf{v} \in \Int \mathcal{C}(k)} \mathcal{K}_{\mathbf{v}}$ on $\Int \mathcal{C}(k)$ is by definition tangential to the $\ker d\mathcal{E}$. Consequently, each leaf gets mapped to a single point by the map $\mathcal{E}$. Note that for $k \ge 2$, we have $\Int \mathcal{C}(k) = \mathcal{C}$, whereas for $k = 1$ both $\mathcal{C}(k)$ and $\mathcal{C}(k-1)$ can have non-empty interiors.
\end{remark}

We are now in a position to state the local form of the normal exponential map $\mathcal{E}$ near a regular focal point, similar to \cite[Theorem 3.3]{Warner1965}. See also \cite[Section 3]{StefUhl12} for a similar discussion.

\begin{theorem}\label{thm:normalExponentialLocalForm}
	Suppose $\mathbf{v} \in \hat{\nu}$ is a regular tangent focal locus of $N$ with multiplicity $k$. Let $\mathcal{C}$ be the connected component of $\mathcal{F}^{\textrm{reg}}$ containing $\mathbf{v}$, and consider the decomposition $\mathcal{C} = \mathcal{C}(k) \sqcup \mathcal{C}(k-1)$ (\autoref{eq:regularFocalDecomposiiton}). Then, there exist coordinates $\left\{ x^1,\dots, x^n \right\}$ and $\left\{ y^1,\dots, y^n \right\}$ near $\mathbf{v} \in \hat{\nu}$ and $\mathcal{E}(\mathbf{v}) \in M$, which can be arranged so that $\mathcal{E}$ has special forms in the following cases.
	\begin{enumerate}[label=(T\arabic*)]
		\item\label{thm:normalExponentialLocalForm:1} If $\mathbf{v}$ has multiplicity $k \ge 2$, then one can arrange so that 
		\begin{equation}\label{eq:normalForm:1}
			y^i \circ \mathcal{E} = 
		\begin{cases}
			x^n \cdot x^i, &i = 1,\dots, k, \\
			x^i, &i = k + 1, \dots, n.
		\end{cases}
		\end{equation}
		Furthermore, $\mathcal{K}_{\mathbf{v}} = \mathrm{Span}\left\langle \frac{\partial}{\partial x^i}\middle|_{\mathbf{v}}, \; i = 1,\dots,k  \right\rangle$. Locally near $\mathbf{v}$ the tangent focal locus is given as $\left\{ x^n = 0 \right\}$, and the image under $\mathcal{E}$ given as $\left\{ y^1 = \dots = y^k = 0 = y^n \right\}$.

		\item\label{thm:normalExponentialLocalForm:2} If $\mathbf{v}$ has multiplicity $k = 1$ and furthermore, $\mathbf{v} \in \Int \mathcal{C}(k)$ then one can arrange so that 
		\begin{equation}\label{eq:normalForm:2}
			y^i \circ \mathcal{E} = 
		\begin{cases}
			x^n \cdot x^1, &i=1,\\
			x^i, &i = 2, \dots , n.
		\end{cases}
		\end{equation}
		Furthermore, $\mathcal{K}_{\mathbf{v}} = \mathrm{Span}\left\langle \frac{\partial}{\partial x^1}\middle|_{\mathbf{v}} \right\rangle$. Locally near $\mathbf{v}$ the tangent focal locus is given as $\left\{ x^n = 0 \right\}$, and the image under $\mathcal{E}$ given as $\left\{ y^1 = 0 = y^n \right\}$.
		
		\item\label{thm:normalExponentialLocalForm:3} If $\mathbf{v}$ has multiplicity $k = 1$ and furthermore, $\mathcal{K}_{\mathbf{v}} \not \subset T_{\mathbf{v}} \mathcal{F}^{\textrm{reg}}$ (i.e., if $\mathbf{v} \in \mathcal{C}(k-1) = \mathcal{C}(0)$), then one can arrange so that
		\begin{equation}\label{eq:normalForm:3}
			y^i \circ \mathcal{E} = 
		\begin{cases}
			x^1 \cdot x^1, &i = 1,\\
			x^i, &i = 2, \dots , n.
		\end{cases}
		\end{equation}
		Furthermore, $\mathcal{K}_{\mathbf{v}} = \mathrm{Span}\left\langle \frac{\partial}{\partial x^1}\middle|_{\mathbf{v}} \right\rangle$. Locally near $\mathbf{v}$ the tangent focal locus is given as $\left\{ x^1 = 0 \right\}$, and the image under $\mathcal{E}$ is given as $\left\{ y^1 = 0 \right\}$.
	\end{enumerate}
\end{theorem}

The proofs of the above two theorems follow in the same vein as \cite[Theorem 3.2 and 3.3]{Warner1965}, which uses the notion of higher order tangent vectors. To keep the article self-contained, we have provided sketches of proofs in the appendix (\autoref{sec:appendix}). 

\begin{remark}\label{rmk:noNormalForm}
	The only regular tangent focal points $\mathbf{v}$ not considered in \autoref{thm:normalExponentialLocalForm} are of the form $\mathcal{C} \setminus \left( \mathcal{C}(0) \sqcup \Int \mathcal{C}(1) \right) = \mathcal{C}(1) \setminus \Int \mathcal{C}(1)$, where $\mathcal{C}$ is a connected component of $\mathcal{F}^{\textrm{reg}}$ consisting of multiplicity $1$ tangent focal points. Clearly, they form a nowhere dense subset of $\mathcal{F}^{\textrm{reg}}$, i.e., the subset of regular tangent points not considered above are not dense in any open set of $\mathcal{F}^{\textrm{reg}}$.
\end{remark}

We can now prove the following.
\begin{theorem}\label{thm:normalExponentialNonInjective}
	Let $N$ be a submanifold of a forward complete Finsler manifold $(M, F)$. Then the normal exponential map $\mathcal{E} = \exp^\nu|_{\hat{\nu}}$ is not injective on any neighborhood of a tangent focal point $\mathbf{v} \in \hat{\nu}$.
\end{theorem}
\begin{proof}
	Since $\mathcal{F}^{\textrm{reg}}$ is dense in $\mathcal{F}$, we only proof the statement for $\mathbf{v} \in \mathcal{F}^{\textrm{reg}}$. Assume $\mathcal{C} \subset \mathcal{F}^{\textrm{reg}}$ is the connected component containing $\mathbf{v}$, and denote the decomposition $\mathcal{C} = \mathcal{C}(k) \sqcup \mathcal{C}(k-1)$. If $\mathbf{v}$ has multiplicity $k \ge 2$, then by \autoref{thm:kernelContainedInTangent}, we have a $k$-dimensional foliation in $\mathcal{C}$, such that $\mathcal{E}$ maps each leaf to a single point. The claim is then immediate. Suppose $k = 1$. If $\mathbf{v} \in \Int \mathcal{C}(k)$, we again have a one dimensional foliation in $\mathcal{C}$, so that each leaf is mapped to a single point by $\mathcal{E}$. If $\mathbf{v} \in \mathcal{C} (k-1)$, then by \autoref{thm:normalExponentialLocalForm} \autoref{thm:normalExponentialLocalForm:3}, we have a coordinate system for which $\mathcal{E}$ looks like \autoref{eq:normalForm:3}. Consequently, $\mathcal{E} $ is not locally injective in a neighborhood of $\mathbf{v} $. Lastly, suppose $\mathbf{v} \in \mathcal{C} \setminus \left( \mathcal{C} (k-1) \sqcup \Int \mathcal{C} (k) \right) $, which is nowhere dense by \autoref{rmk:noNormalForm}. Then in any neighborhood of $\mathbf{v} $, there exists a $\mathbf{v} ' \in \mathcal{C} (k-1) \sqcup \Int \mathcal{C} (k)$ and hence in that neighborhood, the map $\mathcal{E}$ fails to be injective. This concludes the proof. 
\end{proof}

\begin{remark}\label{rmk:alternateProofSepIsDenseInCut}
	In \cite[Theorem 4.8]{BhoPra2023}, it was proved that under hypothesis (\hyperref[eq:hypothesisH]{$\mathsf{H}$}), we have $\mathrm{Cu}(N) = \overline{\mathrm{Se}(N)}$. In the course of the proof, we had reached a dichotomy whose case (a) leads to a contradiction. We would like to point out that in view of \autoref{thm:normalExponentialNonInjective}, case (a) is immediately ruled out. We shall see a stronger result in the tangent bundle following the ideas of \cite{Bishop77} (see \autoref{rmk:cutIsClosureOfSep}).
\end{remark}

\subsection{Decomposition of the Tangent Cut Locus} \label{sec:decomposition}
Let us first observe which tangent focal points can also be cut points. We prove a more general result about the Hausdorff dimension of a large class of tangent focal points. Let $\mathfrak{F}_j$ be the set of all tangent focal points of multiplicity $j$. As noted in \autoref{example:multiplicityOneIsRegular}, we have $\mathfrak{F}_1 \subset \mathcal{F}^{\textrm{reg}}$. Now, by \autoref{thm:regularFocalLocusComponent} \autoref{thm:regularFocalLocusComponent:description}, any connected component of $\mathcal{F}^{\textrm{reg}}$ containing a focal point of multiplicity $1$, must consist of focal points of similar types. Thus, $\mathfrak{F}_1$ is a union of connected components of $\mathcal{F}^{\textrm{reg}}$. Consider the sets 
\begin{equation}\label{eq:tangetialFocalPoints}
	\mathcal{Q} \coloneqq \bigcup_{\mathcal{C} \subset \mathfrak{F}_1} \mathcal{C}(1), \qquad \mathcal{T} \coloneqq \mathcal{Q} \cup \bigcup_{j \ge 2} \mathfrak{F}_j,
\end{equation}
where the first union runs over all the connected components $\mathcal{C} \subset \mathfrak{F}_1$, and we have $\mathcal{C} = \mathcal{C}(0) \sqcup \mathcal{C}(1)$ as in \autoref{eq:regularFocalDecomposiiton}. Note that $\mathcal{T} \subset \mathcal{F}$ consists of all the tangent focal points, except of the type \autoref{thm:normalExponentialLocalForm:3}, which can be called the \emph{fold focal points} in the terminology of \cite{StefUhl12}.

\begin{prop}\label{prop:hausdorffMeasure}
	Each $\mathfrak{F}_j$ for $j \ge 2$, and $\mathcal{Q}$ have $(n-1)$-Hausdorff measure $0$, where $n = \dim M$. Consequently, the set $\exp^\nu\left( \mathcal{T} \right)$ has dimension $\le n-2$.
\end{prop}
\begin{proof}
	Let $\mathbf{v} \in \mathfrak{F}_j$ for some $j \ge 1$. By an application of \autoref{lemma:localCoordinatesAtFocalPoint} we get charts $U$ and $V$ respectively around $\mathbf{v}$ and $\mathcal{E}(\mathbf{v})$.
	As argued in the proof of \autoref{thm:regularTangentFocalPointsSubmanifold}, we get the function $\Delta : U \rightarrow \mathbf{R}$ defined near $\mathbf{v}$, which is expressed as the sum of all $(n- j - 1) \times (n - j -1)$ principal minors of $d\mathcal{E}|_U$. Clearly, for any $\mathbf{u} \in \mathfrak{F}_j \cap U$, we have $\Delta(\mathbf{u}) = 0$ as $\dim \ker \left( d_{\mathbf{u}} \mathcal{E} \right) = j$.
	It follows from a similar argument as in \autoref{thm:regularTangentFocalPointsSubmanifold} that $d_{\mathbf{v}}\mathcal{E}$ is nonsingular, and hence, $\Delta$ is a submersion on $U$, possibly after shrinking. Thus, we see that $\mathfrak{F}_j \cap U$ is contained in the hypersurface $\Delta^{-1}(0)$. We consider the two cases.
	\begin{itemize}
		\item If $j \ge 2$, then we have $\dim \mathcal{K}_{\mathbf{u}} \ge 2$ for $\mathbf{u} \in \mathfrak{F}_j \cap U$. Consequently, $\mathcal{K}_{\mathbf{u}}$ must intersect $T_{\mathbf{u}} \left( \Delta^{-1}(0) \right)$. Hence, $\mathcal{E}|_{\Delta^{-1}(0)}$ has singularities on $\mathfrak{F}_j \cap \Delta^{-1}(0)$.
		\item If $j = 1$, for any $\mathbf{u} \in \mathfrak{F}_1 \cap U$, we have $\mathcal{C} \cap U = \Delta^{-1}(0)$, where $\mathcal{C}$ is the connected component of $\mathbf{u}$. Thus, for $\mathbf{u} \in \mathcal{Q} \cap U$ we clearly have $K_{\mathbf{u}} \subset T_{\mathbf{u}} \left( \Delta^{-1}(0) \right)$. Again, we see that $\mathcal{E}|_{\Delta^{-1}(0)}$ has singularities on $\mathcal{Q} \cap \Delta^{-1}(0)$.
	\end{itemize}
	Then, by an application of Sard's theorem \cite[Theorem 4.1]{Sard42} to the smooth map $\mathcal{E}|_{\Delta^{-1}(0)}$, we see that $\mathcal{E}\left( \mathcal{Q} \cap \Delta^{-1}(0) \right)$, and $\mathcal{E}\left( \mathfrak{F}_j \cap \Delta^{-1}(0) \right)$ for $j \ge 2$, have $(n-1)$-Hausdorff measure $0$ as $\Delta^{-1}(0)$ is an $(n-1)$-dimensional manifold. Now, each $\mathfrak{F}_j$ (and $\mathcal{Q}$ for $j = 1$) is a subset of the second countable space $\hat{\nu}$, and hence, is a Lindel\"of space. Thus, one can cover $\mathfrak{F}_j$ (or $\mathcal{Q}$ for $j = 1$) by countably many open sets $\left\{ U_\alpha \right\}$ such that each $\mathfrak{F}_j \cap U_\alpha$ (or $\mathcal{Q} \cap U_\alpha$, for $j = 1$) is contained in a hypersurface of the form $\Delta^{-1}(0)$ as above. Consequently, the images of $\mathcal{Q} \subset \mathfrak{F}_1$ and $\mathfrak{F}_j$ for $j \ge 2$ under $\mathcal{E}$ have $(n-1)$-Hausdorff measure $0$, and so $\exp^\nu\left( \mathcal{T} \right)$ has dimension $\le n - 2$ \cite[Theorem VII 3]{HurewiczWallmanDimensionTheory}. 
\end{proof}

\begin{remark}\label{rmk:barden}
	The above result was first proved for a point by Warner in \cite[Lemma 1.1, pg. 194]{Warner67}. In \cite[Proposition 1]{BarLe97}, the authors claimed that the whole conjugate locus of a point in the manifold has $(n-1)$-Hausdorff measure $0$, and in particular has codimension at least $2$. This claim seems to be incorrect. Indeed, on the ellipsoid, with the standard metric induced from $\mathbb{R}^3$, one can find a point whose conjugate locus has dimension $1$  \cite{ItohKiyohara04}. Although this does not affect the rest of the article \cite{BarLe97}, since they are only using Proposition 1 to show that \emph{the cut locus of a point has codimension $2$ in $M$}. This remains valid, as can be seen from the next result.
\end{remark}

\begin{prop}\label{prop:dimensionOfConjugateCutPoints}
	A tangent focal point of type \autoref{thm:normalExponentialLocalForm:3} cannot be a tangent cut point. Consequently, the focal cut locus of $N$ has codimension $2$ in $M$.
\end{prop}
\begin{proof}
	Let $\mathbf{v} \in \mathcal{F}$ be a tangent focal locus of multiplicity $1$, and furthermore $\mathcal{K}_{\mathbf{v}} \cap T_{\mathbf{v}} \mathcal{C} = \left\{ 0 \right\}$, where $\mathcal{C}$ is the connected component of $\mathcal{F}^{\textrm{reg}}$ containing $\mathbf{v}$. Suppose, if possible, $\mathbf{v} \in \widetilde{\mathrm{Cu}}(N)$ is a tangent cut point as well. In particular, $\mathbf{v}$ is then a first tangent focal point, and furthermore, $\rho(\widehat{\mathbf{v}}) = \lambda_1(\widehat{\mathbf{v}})$ holds by \autoref{eq:cutTimeFocalTime}. As observed in \autoref{example:multiplicityOneIsRegular}, we have $\mathbf{v}$ is regular and $\mathcal{C}$ consists of multiplicity $1$ tangent focal points, which are also the first tangent focal points. As argued in the proof of \autoref{thm:regularFocalLocusComponent}, we get neighborhoods $\widehat{\mathbf{v}} \in \hat{U} \subset S(\nu)$ and $\mathbf{v} \in U \subset \hat{\nu}$, such that $U$ is radially convex as in \autoref{R3}, and $\hat{U} = \left\{ \widehat{\mathbf{u}} \coloneqq \frac{\mathbf{u}}{F(\mathbf{u})} \;\middle|\; \mathbf{u}\in U \right\}$. Moreover, we can assume that $\mathcal{K}_{\mathbf{u}} \cap T_{\mathbf{u}}\mathcal{C} = \left\{ 0 \right\}$ for $\mathbf{u} \in U \cap \mathcal{C}$. 
	By \autoref{thm:regularTangentFocalPointsSubmanifold} we have the vector bundle decomposition 
	\[T\hat{\nu}|_{U \cap \mathcal{C}} = T(U \cap \mathcal{C}) \oplus \mathcal{R}_U,\]
	where $\mathcal{R}_U \coloneqq \cup_{\mathbf{u} \in U \cap \mathcal{C}} T_{\mathbf{u}} R_{\mathbf{u}}$. Note that the rank $1$ bundle $\mathcal{K}_U \coloneqq \cup_{\mathbf{u} \in U \cap \mathcal{C}} \mathcal{K}_U$ is transverse to $T(U \cap \mathcal{C})$ by hypothesis, and to $\mathcal{R}_U$ by \autoref{R1}. Hence, we have a rank $1$ distribution, say, $\mathcal{W} \subset T(U \cap \mathcal{C})$ given as the projection of $\mathcal{K}_U$ along $\mathcal{R}_U$. Choose a curve $\sigma : (-\epsilon, \epsilon) \rightarrow U \cap \mathcal{C}$ tangential to $\mathcal{W}$, with $\sigma(0) = \mathbf{v}$. Then, we have the curves
	\[\hat{\sigma}(s) \coloneqq \frac{\sigma(s)}{F(\sigma(s))} \in S(\nu), \quad \eta(s) \coloneqq \exp^\nu(\sigma(s)) = \gamma_{\hat{\sigma}(s)}\left( F(\sigma(s)) \right) = \gamma_{\hat{\sigma}(s)} \left( \lambda_1(\hat{\sigma}(s)) \right).\]
	As $\lambda_1$ is smooth on $\hat{U}$ by \autoref{thm:smoothnessOfFocalTime}, it follows that $\eta$ is smooth. Now, by choice of $\sigma$, for each $s$ we have $0 \ne \mathbf{x}(s) \coloneqq \sigma^\prime(s) + \mathbf{r}(s) \in \mathcal{K}_{\sigma(s)}$ for some $\mathbf{r}(s) \in T_{\sigma(s)} R_{\sigma(s)}$. Consider the projection $\pi : \hat{\nu}\rightarrow S(\nu)$ given by $\mathbf{u} \mapsto \frac{\mathbf{u}}{F(\mathbf{u})}$. We have
	\[\hat{\sigma}^\prime(s) = d_{\sigma(s)}\pi\left( \sigma^\prime(s) \right) = d_{\sigma(s)}\pi \left( \mathbf{x}(s) \right),\]
	as $d\pi$ is $0$ along the rays. Then, by \autoref{prop:identificationKernelToUnitBundle}, we get 
	\[J_{\hat{\sigma}^\prime(s)} \left( \lambda_1(\hat{\sigma}(s)) \right) = J_{\hat{\sigma}^\prime(s)}(F(\sigma(s))) = J_{\mathbf{x}(s)}(1) = 0.\]
	Now, consider the $N$-geodesic variation $\Lambda(s, t) = \exp^\nu(t \hat{\sigma}(s))$, so that $\eta(s) = \Lambda(s, \lambda_1(\hat\sigma(s)))$. We compute 
	\begin{align*}
		\dot \eta(s) &= \partial_s \Lambda (s, \lambda_1(\hat{\sigma}(s))) \cdot 1 + \partial_t \Lambda(s, \lambda_1(\hat{\sigma}(s))) \cdot \left( \lambda_1 \circ \hat{\sigma} \right)^\prime(s) \\
		&= \underbrace{J_{\hat{\sigma}^\prime(s)}\left( \lambda_1 \left( \hat{\sigma}(s) \right) \right)}_0 + \dot \gamma_{\hat{\sigma}(s)}\left( \lambda_1(\hat{\sigma}(s)) \right) (\lambda_1 \circ \hat{\sigma})^\prime(s) \\
		&= \dot \gamma_{\hat{\sigma}(s)}\left( \lambda_1(\hat{\sigma}(s)) \right) (\lambda_1 \circ \hat{\sigma})^\prime(s).
	\end{align*}
	On the other hand, $\dot \eta(s) = d \mathcal{E} (\dot\sigma(s)) \ne 0$ as $\dot\sigma(s) \not \in \mathcal{K}_{\sigma(s)}$. Hence, we must have $\left( \lambda_1 \circ \hat{\sigma} \right)^\prime(s) \ne 0$. Without loss of generality, we may assume that $\left( \lambda_1 \circ \hat{\sigma} \right)^\prime(s) > 0$. Now, for some $0 < \delta < \epsilon$, we compute 
	\[L(\eta|_{[-\delta,0]}) = \int_{-\delta}^0 F(\dot\eta(s)) ds = \int_{-\delta}^0 \left( \lambda_1 \circ \hat{\sigma} \right)^\prime(s) \underbrace{F\left( \dot \gamma_{\hat{\sigma}(s)} (\lambda_1(\hat{\sigma}(s)))\right)}_1 ds = \lambda_1(\hat{\sigma}(0)) - \lambda_1(\hat{\sigma}(-\delta)).\]
	Set $\mathbf{w} = \sigma(-\delta)$, so that $L\left( \eta|_{[-\delta, 0]} \right) = \lambda_1(\widehat{\mathbf{v}}) - \lambda_1(\widehat{\mathbf{w}})$. We have the $N$-geodesics $\gamma_{\mathbf{v}}, \gamma_{\mathbf{w}} : [0, 1] \rightarrow M$ joining $N$ to the points $p \coloneqq \gamma_{\mathbf{v}}(1) = \exp^\nu(\mathbf{v})$, and $q \coloneqq \gamma_{\mathbf{w}}(1) = \exp^\nu(\mathbf{w})$ respectively. As $\mathbf{v} \in \widetilde{\mathrm{Cu}}(N)$, we have $d(N, p) = L(\gamma_{\mathbf{v}}) = \rho(\widehat{\mathbf{v}}) = \lambda_1(\widehat{\mathbf{v}})$. On the other hand, for $\mathbf{w}$ we get $d(N,q) \le L(\gamma_{\mathbf{w}}) \le \rho(\widehat{\mathbf{w}}) \le \lambda_1(\widehat{\mathbf{w}})$. It follows that
	\[\lambda_1 \left( \widehat{\mathbf{v}} \right) = d(N,p) \le d(N, q) + d(q, p) \le L(\gamma_{\mathbf{w}}) + L\left( \eta|_{[-\delta, 0]} \right) \le \lambda_1 \left( \widehat{\mathbf{w}} \right) + L\left( \eta|_{[-\delta, 0]} \right) = \lambda_1 \left( \widehat{\mathbf{v}} \right).\]
	 Consequently, all the inequalities above are, in fact, equalities. In particular, $\gamma_{\mathbf{w}}$, followed by $\eta|_{[-\delta, 0]}$ is then an $N$-segment, and thus $q$ cannot be a cut point of $N$. This forces, 
	 \[\lambda_1 \left( \widehat{\mathbf{w}} \right) = L\left( \gamma_{\mathbf{w}} \right) < \rho\left( \widehat{\mathbf{w}} \right),\]
	 which is a contradiction to \autoref{eq:cutTimeFocalTime}. Hence, $\mathbf{v} \not \in \widetilde{\mathrm{Cu}}(N)$. We conclude the proof by applying \autoref{prop:hausdorffMeasure}
\end{proof}

The above result has also been proved in the Riemannian context in \cite[Proposition 3.2]{Heb83}, and in \cite[Lemma 2]{ItohTanaka98}. Now, as mentioned earlier, the cut locus $\mathrm{Cu}(N)$ consists of points that are either a first focal locus along an $N$-geodesic, or points admitting at least two $N$-segments, i.e., points of $\mathrm{Se}(N)$. In the normal bundle $\nu$, we have the set of tangent cut points $\widetilde{\mathrm{Cu}}(N)$, which maps to $\mathrm{Cu}(N)$ under the normal exponential map.

\begin{defn}\label{defn:separatingTangentCutPoint}
    A vector $\mathbf{v} \in \widetilde{\mathrm{Cu}}(N)$ is called a \emph{separating} (or \emph{ordinary}) tangent cut point if there exists some $\mathbf{w} \in \widetilde{\mathrm{Cu}}(N)$ with $\mathbf{w} \ne \mathbf{v}$ such that $\exp^\nu(\mathbf{v}) = \exp^\nu(\mathbf{w})$. Otherwise, $\mathbf{v}$ is called a \emph{singular} tangent cut point. We denote the set of all separating tangent cut points of $N$ by $\widetilde{\mathrm{Se}\kern .01cm}(N)$.
\end{defn}

It is clear that $\widetilde{\mathrm{Se}\kern .01cm}(N)$ is mapped to $\mathrm{Se}(N)$ under the normal exponential map. Then, it follows that 
\begin{equation}\label{eq:singularTangentCutIsFocal}
    \widetilde{\mathrm{Cu}}(N) \setminus \widetilde{\mathrm{Se}\kern .01cm}(N) \subset \mathcal{F}_1(N),
\end{equation}
where $\mathcal{F}_1$ is the first tangent focal locus. Generalizing the main theorem of \cite{Bishop77}, we now prove the following.

\begin{theorem}\label{thm:tangentSepIsDense}
    Let $N$ be a closed submanifold of a forward complete Finsler manifold $(M, F)$, and suppose hypothesis (\hyperref[eq:hypothesisH]{H}) holds. Then, $\widetilde{\mathrm{Cu}}(N) = \overline{\widetilde{\mathrm{Se}\kern .01cm}(N)}$.
\end{theorem}
\begin{proof}
    Since every non-focal tangent cut point is necessarily separating (\autoref{eq:singularTangentCutIsFocal}), we consider some $\mathbf{v} \in \widetilde{\mathrm{Cu}}(N) \cap \mathcal{F} = \widetilde{\mathrm{Cu}}(N) \cap \mathcal{F}_1$. We have the following cases to consider 
    \begin{enumerate}
        \item Suppose $\mathbf{v}$ is a regular focal locus, and of type \autoref{thm:normalExponentialLocalForm:1} or \autoref{thm:normalExponentialLocalForm:2} in \autoref{thm:normalExponentialLocalForm}. Then, by \autoref{rmk:expMapsLeavesToPoint}, we have an involutive distribution on $\mathcal{F}^{\textrm{reg}}$, defined near $\mathbf{v}$, such that each leaf is mapped to a single point by $\exp^\nu$. In particular, the leaf passing through $\mathbf{v}$, say, $\mathfrak{L}$ is mapped to $q \coloneqq \exp^\nu(\mathbf{v})$. Consider a sequence $\mathbf{v}_i \in \mathfrak{L} \setminus \left\{ \mathbf{v} \right\}$ with $\mathbf{v}_i \rightarrow \mathbf{u}$. We have the tangent cut points $\mathbf{w}_i \coloneqq \rho(\widehat{\mathbf{v}}_i)\widehat{\mathbf{v}}_i$, and the tangent focal points $\lambda_1(\widehat{\mathbf{v}}_i) \widehat{\mathbf{v}}_i = \mathbf{v}_i$. Now, $\rho \le \lambda_1$ by \autoref{eq:cutTimeFocalTime}. If $\rho(\widehat{\mathbf{v}}_i) = \lambda_1(\widehat{\mathbf{v}}_i)$ for some $i$, we get $\mathbf{v}_i \in \widetilde{\mathrm{Cu}}(N)$. But then $\mathbf{v} \in \widetilde{\mathrm{Se}\kern .01cm}(N)$, as $\exp^\nu(\mathbf{v}) = \exp^\nu(\mathbf{v}_i)$ and $\mathbf{v} \ne \mathbf{v}_i$. Otherwise, assume that $\rho(\widehat{\mathbf{v}}_i) < \lambda_1(\widehat{\mathbf{v}}_i)$ holds for all $i$. As the cut time map $\rho$ is continuous under hypothesis (\hyperref[eq:hypothesisH]{H}) \cite[Theorem 4.7]{BhoPra2023}, we get $\mathbf{w}_i \rightarrow \mathbf{v}$. As $\mathbf{v}_i$ is a \emph{first} tangent focal locus, we must have $\mathbf{w}_i \in \widetilde{\mathrm{Se}\kern .01cm}(N)$. But then in any neighborhood of $\mathbf{v}$ we have some element of $\widetilde{\mathrm{Se}\kern .01cm}(N)$, proving $\mathbf{v} \in \overline{\widetilde{\mathrm{Se}\kern .01cm}(N)}$.
        
        \item If $\mathbf{v}$ is a regular focal locus of type \autoref{thm:normalExponentialLocalForm:3} as in \autoref{thm:normalExponentialLocalForm}, then by \autoref{prop:dimensionOfConjugateCutPoints}, $\mathbf{v}$ cannot be a tangent cut point.
        
        \item Lastly, assume that $\mathbf{v}$ is a focal point not covered in \autoref{thm:normalExponentialLocalForm}. Since $\mathbf{v} \in \mathcal{F}_1$, by \autoref{thm:regularJthTangentFocalLocus}, we have $\mathbf{v}$ is the limit point of $\mathbf{v}_i \in \mathcal{F}^{\textrm{reg}}_1$. Also, in view of \autoref{rmk:noNormalForm}, we may assume that each $\mathbf{v}_i$ is of type \autoref{thm:normalExponentialLocalForm:1}, \autoref{thm:normalExponentialLocalForm:2} or \autoref{thm:normalExponentialLocalForm:3} as in \autoref{thm:normalExponentialLocalForm}. Denote $\mathbf{w}_i \coloneqq \rho(\widehat{\mathbf{v}}_i) \widehat{\mathbf{v}}_i$, and note that $\mathbf{w}_i \rightarrow \mathbf{v}$ as in case (1). If $\rho(\widehat{\mathbf{v}}_i) = \lambda_1(\widehat{\mathbf{v}}_i)$ holds for some $i$, we have $\mathbf{v}_i \in \widetilde{\mathrm{Cu}}(N)$, and hence by the previous two cases, $\mathbf{w}_i = \mathbf{v}_i$ is a limit point of $\widetilde{\mathrm{Se}\kern .01cm}(N)$. If $\rho(\widehat{\mathbf{v}}_i) < \lambda_1(\widehat{\mathbf{v}}_i)$, then $\mathbf{w}_i \in \widetilde{\mathrm{Cu}}(N) \setminus \mathcal{F}_1 = \widetilde{\mathrm{Se}\kern .01cm}(N)$. By a standard Cantor's diagonalization argument, we then get a sequence in $\widetilde{\mathrm{Se}\kern .01cm}(N)$ converging to $\mathbf{v}$, again showing that $\mathbf{v} \in \overline{\widetilde{\mathrm{Se}\kern .01cm}(N)}$
    \end{enumerate}
    Thus, we have obtained $\widetilde{\mathrm{Cu}}(N) \subset \overline{\widetilde{\mathrm{Se}\kern .01cm}(N)}$.
    
    To show the equality, let us consider some $\mathbf{v} \in \overline{\widetilde{\mathrm{Se}\kern .01cm}(N)} \setminus \widetilde{\mathrm{Se}\kern .01cm}(N)$. Then, we have a sequence $\mathbf{v}_i \in \widetilde{\mathrm{Se}\kern .01cm}(N)$ such that $\mathbf{v}_i \rightarrow \mathbf{v}$. Furthermore, we have $\mathbf{w}_i \ne \mathbf{v}_i$ such that $q_i \coloneqq \exp^\nu(\mathbf{w}_i) = \exp^\nu(\mathbf{v}_i)$. Clearly, $q_i \rightarrow q \coloneqq \exp^\nu(\mathbf{v})$. Also, by \autoref{prop:existenceOfNSegements}, passing to a subsequence, we have $\mathbf{w}_i \rightarrow \mathbf{w}$. Since $q = \exp^\nu(\mathbf{v}) = \lim_i \exp^\nu(\mathbf{v}_i) = \lim_i \exp^\nu(\mathbf{w}_i) = \exp^\nu(\mathbf{w})$, we must have $\mathbf{w} = \mathbf{v}$, as $\mathbf{v} \not \in \widetilde{\mathrm{Se}\kern .01cm}(N)$. But then in any neighborhood of $\mathbf{v}$, the map $\exp^\nu$ fails to be injective. By the inverse function theorem, we must have $d_{\mathbf{v}} \exp^\nu|_{\hat{\nu}}$ is singular, i.e., $\mathbf{v} \in \mathcal{F}$. Also, by \autoref{lemma:convergentSubsequenceNSegment}, we see that $\gamma_{\widehat{\mathbf{v}}}$ being a limit of $N$-segments $\gamma_{\widehat{\mathbf{v}}_i}$, is an $N$-segment itself. But then $q = \gamma_{\mathbf{v}}(1)$ is a \emph{first} focal locus of $N$ along $\gamma_{\mathbf{v}}$ \cite[Lemma 4.4]{BhoPra2023}. In particular, $q$ is then a cut point, and consequently $\mathbf{v} \in \widetilde{\mathrm{Cu}}(N)$. Hence, we have $\widetilde{\mathrm{Cu}}(N) = \overline{\widetilde{\mathrm{Se}\kern .01cm}(N)}$, concluding the proof.
\end{proof}

\begin{remark}\label{rmk:cutIsClosureOfSep}
    Since $\exp^\nu$ is continuous, we immediately get
    \[\mathrm{Cu}(N) = \exp^\nu\left( \widetilde{\mathrm{Cu}}(N) \right) = \exp^\nu\left( \overline{\widetilde{\mathrm{Se}\kern .01cm}(N)} \right) \subset \overline{\exp^\nu\left( \widetilde{\mathrm{Se}\kern .01cm}(N) \right)} = \overline{\mathrm{Se}(N)}.\]
    Consequently, $\mathrm{Se}(N)$ is dense in $\mathrm{Cu}(N)$. In fact, $\mathrm{Cu}(N) = \overline{\mathrm{Se}(N)}$ holds as well \cite[Theorem 4.8]{BhoPra2023}.
\end{remark}

\section{Open Questions}\label{sec:openQuestions}
In this section, we pose some open questions for further research. The first question is related to \autoref{thm:regularFocalLocusComponent}
\begin{question}\label{question:componentBoundary}
    Given a connected component $\mathcal{C} \subset \mathcal{F}^{\textrm{reg}}$, what can be deduced about $\partial \mathcal{C}$?
\end{question}

Let us make some easy observations. If $\mathcal{C} \subset \mathcal{F}^{\textrm{reg}}_j$ for some $j$, then it follows from \autoref{thm:regularFocalLocusComponent} and \autoref{thm:regularJthTangentFocalLocus} that $\partial \mathcal{C} \subset \mathcal{F}^{\textrm{sing}}_j$. Let $\mathbf{v} \in \partial \mathcal{C}$ have the type $(k_0, i_0)$. Suppose, we have a neighborhood $\mathbf{v} \in U \subset \hat{\nu}$ satisfying $U \cap \mathcal{F}^{\textrm{reg}}_j = U \cap \mathcal{C}$, which is the case if $\mathbf{v}$ is isolated in $\mathcal{F}^{\textrm{sing}}$. Then one can easily conclude that $U \cap \mathcal{F}^{\textrm{reg}}_i = U \cap \mathcal{C}$ for all $i_0 \le i < i_0 + k_0$. On the other hand, suppose a sufficiently small neighborhood $U$ intersects a finite list of components, say, $\mathcal{C}_1,\dots ,\mathcal{C}_{r_0}$ from $\mathcal{F}^{\textrm{reg}}_j$, with $r_0 \ge 2$. Shrinking $U$ suitably, and applying \autoref{thm:regularFocalLocusComponent}, one can see that $U \cap \mathcal{F}^{\textrm{sing}} = U \cap \mathcal{F}^{\textrm{sing}}_j$ is mapped homeomorphically to a set which is the topological boundary of finitely many open charts. This begs the question of whether the boundary is rectifiable.\medskip

Next, as mentioned earlier, in \cite{Itoh2001} it is proved that if the focal time map $\lambda_j : S(\nu) \rightarrow (0, \infty]$ of a submanifold in a Riemannian manifold is finite at $\mathbf{v}$, then $\lambda_j$ locally Lipschitz near $\mathbf{v}$. It appears that their argument can be generalized to the Finsler setup, albeit with some effort. Then, assuming $\lambda_j < \infty$, an application of Rademacher's theorem implies that $\lambda_j$ is differentiable almost everywhere. In \autoref{thm:smoothnessOfFocalTime}, we identified an open dense subset $\mathcal{R}_j \subset S(\nu)$ on which $\lambda_j$ is smooth, provided $\lambda_j < \infty$.
\begin{question}\label{question:focalTimeNonSmooth}
    On which points of $S(\nu) \setminus \mathcal{R}_j$ is the map $\lambda_j$ non-differentiable?
\end{question}

Lastly, we repeat a question originally posed by Bishop in \cite{Bishop77}, to the best of our knowledge, which is still unsolved, even for a point in a Riemannian manifold.

\begin{question}\label{question:separtingSetOpen}
    Is $\widetilde{\mathrm{Se}\kern .01cm}(N)$ open in $\widetilde{\mathrm{Cu}}(N)$? Is $\mathrm{Se}(N)$ open in $\mathrm{Cu}(N)$?
\end{question}

One can readily observe the following.
\begin{itemize}
    \item For any $\mathbf{v} \in \widetilde{\mathrm{Cu}}(N) \setminus \mathcal{F}_1 \subset \widetilde{\mathrm{Se}\kern .01cm}(N)$, we have $\rho(\widehat{\mathbf{v}}) < \lambda_1(\widehat{\mathbf{v}})$. Since both $\rho$ and $\lambda_1$ are continuous, it follows that in a neighborhood of $\widehat{\mathbf{v}}$ we have $\rho < \lambda_1$. But then, in a neighborhood of $\mathbf{v}$, every cut point is non-focal, and hence belongs to $\widetilde{\mathrm{Se}\kern .01cm}(N)$. Thus, non-focal cut points lie in the interior of $\widetilde{\mathrm{Se}\kern .01cm}(N)$.
    \item If $\mathbf{v} \in \mathrm{Cu}(N)$ is a regular tangent focal point, then by \autoref{prop:dimensionOfConjugateCutPoints} it must be of type \autoref{thm:normalExponentialLocalForm:1} or \autoref{thm:normalExponentialLocalForm:2}. Now, in a neighborhood of $\mathbf{v}$, $\mathcal{F}^{\textrm{reg}}$ is foliated in such a way that every leaf is mapped to the same point under $\exp^\nu$. In particular, for any $\mathbf{u} \in \widetilde{\mathrm{Cu}}(N)$ near $\mathbf{v}$, if $\mathbf{u} \in \mathcal{F}_1$, which is the case for $\mathbf{v}$ itself, then we must have $\mathbf{u} \in \widetilde{\mathrm{Se}\kern .01cm}(N)$. Otherwise, $\mathbf{u} \in \widetilde{\mathrm{Cu}}(N) \setminus \mathcal{F}_1 \subset \widetilde{\mathrm{Se}\kern .01cm}(N)$. It follows that $\mathbf{v}$ is in the interior of $\widetilde{\mathrm{Se}\kern .01cm}(N)$.
\end{itemize}
Thus, the only points left to consider are the regular focal cut points of multiplicity $1$ as described in \autoref{rmk:noNormalForm}, and the singular focal cut points $\widetilde{\mathrm{Cu}}(N) \cap \mathcal{F}^{\textrm{sing}} = \widetilde{\mathrm{Cu}}(N) \cap \mathcal{F}^{\textrm{sing}}_1$.

\appendix\section{Higher Order Tangent Vector}\label{sec:appendix}
 In this appendix, we introduce the notion of higher order tangent vectors (see \cite[Section 1.26]{WarnerBook} for details), and show that our notion of \autoref{R2} is comparable to that of \cite{Warner1965}. Then, we give the proofs of \autoref{thm:kernelContainedInTangent} and \autoref{thm:normalExponentialLocalForm}.\medskip

Let $M$ be a manifold with $\dim M = n$. For $p \in M$, denote by $\mathcal{I}_p$ the collection of germs at $p$ of smooth functions $M \rightarrow \mathbb{R}$. A germ $\mathfrak{f} \in \mathcal{I}_p$ is assumed to be represented by a function $f : M\rightarrow \mathbb{R}$ locally defined near $p$, and in particular, the evaluation $\mathfrak{f}(p) = f(p)$ is well-defined. $\mathcal{I}_p$ is a local ring, with the maximal ideal $\mathfrak{m}_p \coloneqq \left\{ \mathfrak{f} \;\middle|\; \mathfrak{f}(p) = 0 \right\}$. Denote the $k^{\textrm{th}}$ power ideal as $\mathfrak{m}_p^{k}$, which consists of all possible finite sums of $k$-fold products of elements of $\mathfrak{m}_p$. In particular, elements of $\mathfrak{m}_p^{k}$ are represented by functions with a $0$ of order at least $k$ at $p$. The $k^{\textrm{th}}$ order tangent space at $p$ is then defined as
\[T^k_p M \coloneqq \left( \mathfrak{m}_p / \mathfrak{m}_p^{k+1} \right)^* = \hom\left( \mathfrak{m}_p / \mathfrak{m}_p^{k+1}, \; \mathbb{R} \right).\]
Since $\mathfrak{m}_p^k \supset \mathfrak{m}_p^{k+1}$, we have a natural identification $T_p^k M \subset T_p^{k+1} M$.
	Given a smooth map $f : M \rightarrow N$, one has the $k^{\textrm{th}}$ order derivative map
	\begin{align*}
		d^k_p f : T^k_p M &\rightarrow T^k_{f(p)} N \\
		\theta &\mapsto \left( \mathfrak{g} + \mathfrak{m}_{f(p)}^{k+1} \mapsto \theta \left( f_*(\mathfrak{g}) + \mathfrak{m}^{k+1}_p \right) \right),
	\end{align*}
	where the push-forward $f_*(\mathfrak{g})$ is defined as the germ of $g \circ f$ at the point $p$.

Elements of $T_p^k M$ represents derivations of order $\le k$. Fixing a coordinate system $(x^1,\dots , x^n)$ near $p$, one can produce a basis of $T_p^k M$ as $\left\{ \partial_I \;\middle|\; I = (i_1,\dots , i_n), \, \sum i_j \leq k, \, i_j \ge 0 \right\}$, where $\partial_i(\mathfrak{f}) = \left.\frac{\partial}{\partial x^i}\right|_p f$ for $1 \le i \le n$, and more generally $\partial_I(\mathfrak{f}) = \left. \partial_{1}^{i_1}\dots \partial_n^{i_n} \right|_p(f)$ for any $I = (i_1,\dots ,i_n)$. It follows from the Leibniz rule that any such derivation $\partial_I$ of order $\le k$ vanishes on a germ $\mathfrak{f}$ which has a $0$ of order $k+1$ at $p$. Clearly, we have a (non-canonical) isomorphism $T^k_p M = \oplus_{i = 1}^k \odot^i T_p M$, where $\odot^i$ denotes the $i^{\textrm{th}}$ symmetric tensor product. On the other hand, there is a canonical isomorphism $T^{k}_p M / T_p^{k-1} M = \odot^{k} T_p M$, which represents the derivations of order purely $k$.

We are particularly interested in the identification $T_p M \odot T_p M = T^2_p M / T_p M$,
which is explicitly given as 
\[\mathbf{x}\odot \mathbf{y} \longmapsto \left( \mathfrak{f} \mapsto X\left( Y(f) \right) |_p \right) \mod T_p M\]
for arbitrary extensions $X, Y$ of $\mathbf{x}, \mathbf{y} \in T_p M$, and for $\mathfrak{f} \in \mathfrak{m}_p/\mathfrak{m}^3_p$ represented by some $f$. The Leibniz rule shows that the map is well-defined, whereas the symmetry is a consequence of the identity $[X, Y](f) = XY(f) - YX(f)$ where $\mathfrak{f} \mapsto [X, Y](f)|_p$ is a first order vector. With this identification, the second order derivative of some $f : M \rightarrow N$ induces the following natural map,
\begin{equation}\label{eq:secondOrderDerivative}
    \begin{aligned}
        d^2_p f : T_p M \odot T_p M &\longrightarrow T^2_{f(p)} N / \mathop{\mathrm{Im}} d_p f \\
        \mathbf{x} \odot \mathbf{y} &\longmapsto \left( \mathfrak{g} \mapsto X \left( Y \left( g \circ f \right) \right) |_p \right) \mod \mathop{\mathrm{Im}} d_p f,
    \end{aligned}
\end{equation}
which we still denote by $d^2 f$.

Let us now consider the restricted normal exponential map $\mathcal{E} = \left( \exp^\nu|_{\hat{\nu}} \right) : \hat{\nu} \rightarrow M$. For any $\mathbf{v} \in \hat{\nu}$, we have the geodesic $\gamma_{\mathbf{v}}(t) = \mathcal{E}(t\mathbf{v})$. Using the identification above, we have the maps
\[d_{\mathbf{v}} \mathcal{E} : T_{\mathbf{v}} \hat{\nu} \rightarrow T_{\gamma_{\mathbf{v}}(1)} M, \quad d^2_{\mathbf{v}} \mathcal{E} : T_{\mathbf{v}}\hat{\nu} \odot T_{\mathbf{v}} \hat{\nu} \rightarrow T^2_{\gamma_{\mathbf{v}}(1)}M / \mathop{\mathrm{Im}} d_{\mathbf{v}} \mathcal{E}.\]
Denote the ray $\sigma(t) = t\mathbf{v}$ and the vector $\mathbf{r} = \dot \sigma(1) \in T_{\mathbf{v}} \hat{\nu}$. As $\im d_{\mathbf{v}} \mathcal{E} \subset T_{\gamma_{\mathbf{v}}(1)} M \subset T^2_{\gamma_{\mathbf{v}}(1)}M$, one can naturally identify $T_{\gamma_{\mathbf{v}}(1)}M / \mathop{\mathrm{Im}} d_{\mathbf{v}} \mathcal{E} \subset T^2_{\gamma_{\mathbf{v}}(1)} M /\mathop{\mathrm{Im}}d_{\mathbf{v}} \mathcal{E}$. Now, the condition (R2) of \cite{Warner1965} translates to the following
\begin{enumerate}
    \item[(R2')]\label{warnerR2} The map
    \[\mathcal{K}_{\mathbf{v}} = \ker d_{\mathbf{v}} \mathcal{E} \ni \mathbf{x} \longmapsto d^2_{\mathbf{v}} \mathcal{E}(\mathbf{r} \odot \mathbf{x}) \in T_{\gamma_{\mathbf{v}}(1)}M / \mathop{\mathrm{Im}} d_{\mathbf{v}} \mathcal{E} \subset T^2_{\gamma_{\mathbf{v}}(1)} M /\mathop{\mathrm{Im}}d_{\mathbf{v}} \mathcal{E}\]
    is a linear isomorphism onto the image.
\end{enumerate}

\begin{prop}\label{prop:secondOrderTangent}
    For any $\mathbf{x} \in T_{\mathbf{v}} \hat{\nu}$, consider the $N$-Jacobi field $J_{\mathbf{x}}$ along $\gamma_{\mathbf{v}}$ (\autoref{eq:canonicalIsoJacobi}). Then, we have the following.
    \begin{enumerate}[label=(\arabic*)]
        \item \label{prop:secondOrderTangent:firstOrder} $d_{\mathbf{v}} \mathcal{E}(\mathbf{x}) = J_{\mathbf{x}}(1)$, and in particular $d_{\mathbf{v}} \mathcal{E}( \mathbf{r} ) = \dot \gamma_{\mathbf{v}}(1)$.
        \item \label{prop:secondOrderTangent:secondOrder} $d^2_{\mathbf{v}} \mathcal{E}\left( \mathbf{r} \odot \mathbf{x} \right) = \dot J_{\mathbf{x}}(1) \in T_{\mathcal{E}(\mathbf{v})} M / \im d_{\mathbf{v}} \mathcal{E}$, provided $\mathbf{x} \in \mathcal{K}_{\mathbf{v}}$.
    \end{enumerate}
    Consequently, \autoref{R2} is equivalent to (\hyperref[warnerR2]{R2'}).
\end{prop}
\begin{proof}
    Let $\alpha : (-\epsilon, \epsilon) \rightarrow \hat{\nu}$ be a curve such that $\alpha(0) = \mathbf{v}$ and $\dot \alpha(0) = \mathbf{x}$. Then, we immediately have 
    \[d_{\mathbf{v}} \mathcal{E}(\mathbf{x}) = \left. \frac{\partial}{\partial s} \right|_{s=0} \mathcal{E}\left( \alpha(s) \right) = J_{\mathbf{x}}(1),\]
    which proves \autoref{prop:secondOrderTangent:firstOrder}.
	
	Now, assume $\mathbf{x} \in \mathcal{K}_{\mathbf{v}}$, so that $J_{\mathbf{x}}(1) = d_{\mathbf{v}} \mathcal{E}(\mathbf{x}) = 0$. Consider a germ $\mathfrak{f} \in \mathfrak{m}_q / \mathfrak{m}_q^2$, represented by some $f$ defined near $q \coloneqq \mathcal{E}(\mathbf{v}) = \gamma_{\mathbf{v}}(1)$. Considering the map $\Xi(s, t) = t \alpha(s)$, we have an extension $X(t) \coloneqq \left. \frac{\partial}{\partial s} \right|_{s=0} \Xi(s, t)$ of $\mathbf{x}$ along $\sigma$. Extend this $X$ arbitrarily in a neighborhood of $\mathbf{v}$, and similarly extend $\mathbf{r}$ to some $R$. We have the Jacobi field $J_{\mathbf{x}}(t) = d\mathcal{E}(X(t))$ along $\gamma_{\mathbf{v}}$. Fix some frame $\left\{ e_i(t) \right\}_{i=1}^n$ along $\gamma_{\mathbf{v}}$, i.e, assume $T_{\gamma_{\mathbf{v}}(t)} M = \mathrm{Span}\left\langle e_i(t) \right\rangle$. Then, we can write $J_{\mathbf{x}}(t) = \sum_{i=1}^n f_i(t) e_i(t)$ for some smooth maps $f_i$ along $\gamma_{\mathbf{v}}$. Since $J_{\mathbf{x}}(1) = 0$, we get $f_i(1) = 0$ for all $i$. Thus, we have, $\dot J_{\mathbf{x}}(1) = \sum f_i^\prime(1) e_i(1)$. We now compute,
	\begin{align*}
		R X (f \circ \mathcal{E})|_{\mathbf{v}} &= \left. \frac{\partial}{\partial t} \right|_{t = 1} \left( X(f \circ \mathcal{E}) \right) (t \mathbf{v}) = \left. \frac{\partial}{\partial t} \right|_{t = 1} d(f \circ \mathcal{E}) \left( X_{t \mathbf{v}} \right) = \left. \frac{\partial}{\partial t} \right|_{t = 1} df \left( d \mathcal{E} \left( X(t) \right) \right) \\
		&= \left. \frac{\partial}{\partial t} \right|_{t = 1} df \left( J_{\mathbf{x}}(t) \right) = \left. \frac{\partial}{\partial t} \right|_{t = 1} df \left( \sum f_i(t) e_i(t) \right) = \sum \left. \frac{\partial}{\partial t} \right|_{t = 1} f_i(t) df \left( e_i(t) \right) \\
		&= \sum f_i^\prime(1) df \left( e_i(1) \right) = df \left( \sum f_i^\prime(1) e_i(1) \right) = df \left( \dot J_{\mathbf{x}}(1) \right) = \dot J_{\mathbf{x}}(1) \left( \mathfrak{f} \right)|_q \in T_q M.
	\end{align*}
	From \autoref{eq:secondOrderDerivative},  modulo $\im d_{\mathbf{v}} \mathcal{E}$ we then have
	\[d^2_{\mathbf{v}}\mathcal{E}(\mathbf{r} \odot \mathbf{x}) = \left( \mathfrak{f} \mapsto RX(f \circ \mathcal{E})|_{\mathbf{v}} \right) = \left( \mathfrak{f} \mapsto \dot{J}_{\mathbf{x}}(1)(\mathfrak{f})|_{q} \right).\]
	In other words, $d^2_{\mathbf{v}} \mathcal{E}(\mathbf{r} \odot \mathbf{x}) = \dot J_{\mathbf{x}}(1) \mod \im d_{\mathbf{v}} \mathcal{E}$, proving \autoref{prop:secondOrderTangent:secondOrder}.
	Equivalence of \autoref{R2} and (\hyperref[warnerR2]{R2'}) then follows.
\end{proof}

Let us now prove \autoref{thm:kernelContainedInTangent}.

\begin{proof}[Proof of \autoref{thm:kernelContainedInTangent}]
	Suppose $\mathbf{v} \in \mathcal{F}^{\textrm{reg}}$ be a regular focal point, such that $\mathcal{K}_{\mathbf{v}} \not\subset T_{\mathbf{v}} \mathcal{F}^{\textrm{reg}}$. Let $\mathbf{y} \in \mathcal{K}_{\mathbf{v}} \cap T_{\mathbf{v}}\mathcal{F}^{\textrm{reg}}$. We shall show that $\dot J_{\mathbf{y}}(1) = 0$, whence we get $\mathbf{y} = 0$ by \autoref{R2}.
	
	Consider the vector $\mathbf{r} \in T_{\mathbf{v}} \hat{\nu}$ given by $\mathbf{r} = \dot \sigma(1)$, where $\sigma(t) = t\mathbf{v}$ is the ray. Since $\mathbf{r}$ is transverse to both $T_{\mathbf{v}}\mathcal{F}^{\textrm{reg}}$ and $\mathcal{K}_{\mathbf{v}}$ (by \autoref{thm:regularTangentFocalPointsSubmanifold} and \autoref{R1}, respectively), and since $T_{\mathbf{v}} \hat{\nu} = \mathcal{K}_{\mathbf{v}} + T_{\mathbf{v}} \mathcal{F}^{\textrm{reg}}$ by hypothesis, we have $\mathbf{r} = \mathbf{x} + \mathbf{z}$ for some $\mathbf{x} \in \mathcal{K}_{\mathbf{v}} \setminus T_{\mathbf{v}} \mathcal{F}^{\textrm{reg}}$ and $\mathbf{z} \in T_{\mathbf{v}}\mathcal{F}^{\textrm{reg}} \setminus \mathcal{K}_{\mathbf{v}}$. Then, it follows from \autoref{prop:secondOrderTangent} \autoref{prop:secondOrderTangent:secondOrder} that 
	\[\dot J_{\mathbf{y}}(1) = d^2_{\mathbf{v}}\mathcal{E}(\mathbf{r} \odot \mathbf{y}) = d^2_{\mathbf{v}} \mathcal{E} (\mathbf{x} \odot \mathbf{y}) + d^2_{\mathbf{v}} \mathcal{E}(\mathbf{z} \odot \mathbf{y}).\]
	As $\mathbf{z} \in T_{\mathbf{v}}\mathcal{F}^{\textrm{reg}}$, we can choose an extension $Z$ of $\mathbf{z}$ around $\mathbf{v}$, such that $Z$ is tangential to $\mathcal{F}^{\textrm{reg}}$ on points of $\mathcal{F}^{\textrm{reg}}$. Similarly, we can extend $\mathbf{y}$ to some $Y$ near $\mathbf{v}$ so that $Y$ is in $\mathcal{K}_{\mathbf{u}}$ on points of $\mathcal{F}^{\textrm{reg}}$, which is possible since $\mathcal{K}_{\mathbf{u}}$ has constant rank for $\mathbf{u} \in \mathcal{F}^{\textrm{reg}}$ near $\mathbf{v}$. Then, for an arbitrary germ $\mathfrak{f} \in \mathfrak{m}_p/\mathfrak{m}_p^2$, represented by some function $f$ defined near $\mathcal{E}(\mathbf{v})$, we have
	\[Z \left( Y \left( f \circ \mathcal{E} \right) \right)(\mathbf{v}) = Z \left( df \, d\mathcal{E} (Y)\right)(\mathbf{v}) = 0,\]
	as $d\mathcal{E}(Y) = 0$ on $\mathcal{F}^{\textrm{reg}}$. But then $d^2_{\mathbf{v}} \mathcal{E}(\mathbf{z} \odot \mathbf{y}) = 0$ as $\mathfrak{f}$ is arbitrary. Since $\mathbf{y} \in T_{\mathbf{v}} \mathcal{F}^{\textrm{reg}}$ as well, we get from the symmetry, $d^2_{\mathbf{v}} \mathcal{E}(\mathbf{x} \odot \mathbf{y}) = d^2_{\mathbf{v}} \mathcal{E}(\mathbf{y} \odot \mathbf{x}) = 0$. Hence, $\dot J_{\mathbf{y}}(1) = 0$, which shows that $\mathbf{y} = 0$ by \autoref{R2}. In other words, $\mathcal{K}_{\mathbf{v}} \cap T_{\mathbf{v}} \mathcal{F}^{\textrm{reg}} = 0$, showing that $\dim \mathcal{K}_{\mathbf{v}} = 1$. Thus, for $\dim \mathcal{K}_{\mathbf{v}} \ge 2$ we must have $\mathcal{K}_{\mathbf{v}} \subset T_{\mathbf{v}} \mathcal{F}^{\textrm{reg}}$.
\end{proof}

Lastly, we give a sketch of proof of \autoref{thm:normalExponentialLocalForm}, while deferring to \cite[Theorem 3.3]{Warner1965} for details.

\begin{proof}[Proof of \autoref{thm:normalExponentialLocalForm}]
	Suppose we are in either case \autoref{thm:normalExponentialLocalForm:1} or case \autoref{thm:normalExponentialLocalForm:2}. In both the cases, we have a $k$-dimensional foliation on $\mathcal{C}$ near $\mathbf{v}$, induced by the distribution $\ker d\left( \mathcal{E}|_{\mathcal{C}} \right)$, which is of constant rank $k$ and involutive near $\mathbf{v}$ (\autoref{rmk:expMapsLeavesToPoint}). We can get a coordinate system $\left\{ u^1, \dots , u^n \right\}$ on $\hat{\nu}$ near $\mathbf{v}$ such that the codimension $1$ submanifold $\mathcal{C}$ is locally given as $\left\{ u^n = 0 \right\}$. Furthermore, we can arrange so that the foliation on $\mathcal{C}$ is given by 
	\[\left\{ u^{k+1} = \dots = u^{n-1} = \textrm{constant}, \quad u^n = 0.\right\}.\]
	Recall that each leaf of this foliation is mapped to a constant by $\mathcal{E}$, and $d_{\mathbf{v}}\mathcal{E}$ has full rank in the transverse directions (\autoref{thm:kernelContainedInTangent}). Hence, we can fix a coordinate system $\left\{ v^1,\dots ,v^n \right\}$ near $\mathcal{E}(\mathbf{v})$ so that the image of the slice $\left\{ u^1 = \dots = u^k = 0, \quad u^n = 0 \right\}$ under $\mathcal{E}$ is the slice $\left\{ v^1 = \dots = v^k = 0, \quad v^n = 0 \right\}$, and furthermore, $d_{\mathbf{v}}\mathcal{E} \left( \left. \frac{\partial}{\partial u^i} \right|_{\mathbf{v}} \right) = \left. \frac{\partial}{\partial v^i} \right|_{\mathcal{E}(\mathbf{v})}$ holds for $i = k+1,\dots , n$. In particular, for $i = 1,\dots, k$ the function $v^i \circ \mathcal{E}$ has a zero of order $1$ at $\mathbf{v}$. Moreover, we can arrange so that $\left\{ u^n = 0 \right\} = \left\{ v^n \circ \mathcal{E} = 0 \right\}$, whence it follows that $v^i \circ \mathcal{E}|_{\{ w^n = 0 \} } = 0$ for $i = 1,\dots ,k$. Define the functions 
	\[w^i = 
	\begin{cases}
		u^i, &i = 1,\dots ,k \\
		v^i \circ \mathcal{E}, &i = k + 1, \dots , n,
	\end{cases}\]
	near $\mathbf{v}$. Then, $\left\{ w^1,\dots , w^n \right\}$ is a coordinate system near $\mathbf{v}$, since we can easily compute the Jacobian matrix
	\[\begin{pmatrix} \left. \frac{\partial w^i}{\partial u^j} \right|_{\mathbf{v}} \end{pmatrix}_{n \times n} = \begin{pmatrix} \mathrm{Id}_{k \times k} & 0_{(n-k) \times k} \\ \star_{(n-k) \times k} & \mathrm{Id}_{(n-k)\times (n-k)} \end{pmatrix},\]
	which has full rank. Next, consider the projection $\pi : \mathbb{R}^n \rightarrow \mathbb{R}^n$ onto the \emph{last} $n-k$-coordinates, preceded by $k$ many zeros. Near $\mathcal{E}(\mathbf{v}) \in M$, we then have functions, 
	\[y^i = 
	\begin{cases}
		v^i - v^i \circ \mathcal{E} \circ \left( \phi^{-1} \circ \pi \circ  \psi \right), &i=1,\dots, k\\
		v^i, &i=k+1, \dots , n,
	\end{cases}\]
	where $\phi, \psi$ are respectively the coordinate charts $\left\{ w^1,\dots, w^n \right\}$ and $\left\{ v^1,\dots, v^n \right\}$ on $\hat{\nu}$ and $M$. Note that for $1 \le j \le k$, we have $\left. \frac{\partial}{\partial v^j} \right|_{\mathcal{E}(\mathbf{v})} \left( \pi \circ \psi \right) = 0$. Hence, $\left\{ y^1,\dots ,y^n \right\}$ is a coordinate system near $\mathcal{E}(\mathbf{v})$ as the Jacobian matrix
	\[\begin{pmatrix} \left. \frac{\partial y^i}{\partial v^j} \right|_{\mathcal{E}(\mathbf{v})} \end{pmatrix}_{n \times n} = \begin{pmatrix} \mathrm{Id}_{k \times k} & \star_{(n-k) \times k} \\ 0_{(n-k) \times k} & \mathrm{Id}_{(n-k)\times (n-k)} \end{pmatrix}\]
	has full rank.
	Also, for $i = 1,\dots, k$, we have $y^i \circ \mathcal{E} = 0$ on $\left\{ w^n = 0 \right\}$, which is a $0$ of order $1$. Indeed, for $i = 1,\dots , k$, we have observed $v^i \circ \mathcal{E}|_{\{ w^n = 0 \} } = 0$, and also $v^i \circ \mathcal{E} \circ  \left( \phi^{-1} \circ \pi \circ \psi \circ \mathcal{E} \right)|_{\{ w^n = 0 \} } = v^i \circ \mathcal{E}|_{\{ w^1 = \dots = w^k = 0 = w^n \} } = 0$. Hence, we have functions $x^i$ near $\mathbf{v}$ such that $y^i \circ \mathcal{E} = w^n \cdot x^i$ for $i=1,\dots, k$. Set $x^i = w^i$ for $i = k+1,\dots, n$. Thus, we have
	\[y^i \circ \mathcal{E} = 
	\begin{cases}
		x^n \cdot x^i, &i=1,\dots ,k\\
		x^i, &i=k+1,\dots, n.
	\end{cases}\]
	We check that $\left\{ x^1,\dots , x^n \right\}$ is a coordinate system near $\mathbf{v}$. As $x^i = w^i$ for $i = k+1,\dots ,n$, it follows that the $n \times n$ matrix $\left( \left. \frac{\partial x^i}{\partial w^j} \right|_{\mathbf{v}}  \right)$ has full rank if the $k \times k$ block given by $1 \le i, j \le k$ has full rank. We have 
	\[\left. \frac{\partial^2 \left( y^i \circ \mathcal{E} \right)}{\partial w^n \partial w^j} \right|_{\mathbf{v}} = \left. \frac{\partial x^i}{\partial w^j} \right|_{\mathbf{v}} , \qquad 1 \le i, j \le k,\]
	since $y^i \circ \mathcal{E} = w^n \cdot x^i$. Hence, we need to show that $\begin{pmatrix}
		\left. \frac{\partial^2(y^i \circ \mathcal{E})}{\partial w^n \partial w^j} \right|_{\mathbf{v}}
	\end{pmatrix}$ has full rank. Consider the vectors $\mathbf{w}^i = \left. \frac{\partial}{\partial w^i} \right|_{\mathbf{v}}$. By construction, $\mathcal{K}_{\mathbf{v}} = \ker d_{\mathbf{v}} \mathcal{E} = \mathrm{Span}\langle \mathbf{w}^1,\dots ,\mathbf{w}^k \rangle$. Furthermore, $\mathbf{w}^n$ is transverse to $T_{\mathbf{v}} \mathcal{C}$. Consider the radial vector $\mathbf{r} \in T_{\mathbf{v}}\hat{\nu}$ given by $\mathbf{r} = \dot \sigma(1)$, where $\sigma(t) = t\mathbf{v}$ is the ray. Then, we have $\mathbf{r} = a \mathbf{w}^n + \mathbf{z}$ for some scalar $a \ne 0$ and some $\mathbf{z} \in T_{\mathbf{v}} \mathcal{F}^{\textrm{reg}}$. As argued in the proof of \autoref{thm:kernelContainedInTangent}, we have 
	\[\dot J_{\mathbf{w}^j}(1) = d^2_{\mathbf{v}} \mathcal{E} \left( \mathbf{r} \odot \mathbf{w}^j \right) = a d^2_{\mathbf{v}} \mathcal{E}\left( \mathbf{w}^n \odot \mathbf{w}^j \right) + \underbrace{d^2_{\mathbf{v}} \mathcal{E} \left( \mathbf{z} \odot \mathbf{w}^j \right)}_0 = a d^2_{\mathbf{v}} \mathcal{E}\left( \mathbf{w}^n \odot \mathbf{w}^j \right) .\]
	As $a \ne 0$, by \autoref{R2} and \autoref{prop:secondOrderTangent}, we have $\left\{ d^2_{\mathbf{v}} \mathcal{E}(\mathbf{w}^n \odot \mathbf{w}^j) \mod \mathop{\mathrm{Im}}d_{\mathbf{v}} \mathcal{E}, \quad 1 \le j \le k \right\}$ are linearly independent. Now, by construction, we have
	\[\mathop{\mathrm{Im}} d_{\mathbf{v}} \mathcal{E} = \mathrm{Span}\left\langle \left. \frac{\partial}{\partial v^i} \right|_{\mathcal{E}(\mathbf{v})} , \; i = k+1,\dots ,n \right\rangle = \mathrm{Span}\left\langle \left. \frac{\partial}{\partial y^i} \right|_{\mathcal{E}(\mathbf{v})} , \; i = k+1,\dots ,n \right\rangle.\]
	Thus, evaluating on $y^i$ for $i = 1,\dots ,k$ does not change the rank. Now, by \autoref{eq:secondOrderDerivative}, we have $\left. \frac{\partial^2(y^i \circ \mathcal{E})}{\partial w^n \partial w^j} \right|_{\mathbf{v}} = d^2_{\mathbf{v}} \mathcal{E}(\mathbf{w}^n \odot \mathbf{w}^j)(y^i)$. Consequently, the $k \times k$ matrix 
	\[
	\begin{pmatrix}
		\left. \frac{\partial^2(y^i \circ \mathcal{E})}{\partial w^n \partial w^j} \right|_{\mathbf{v}}
	\end{pmatrix} = 
	\begin{pmatrix}
		d^2_{\mathbf{v}} \mathcal{E}(\mathbf{w}^n \odot \mathbf{w}^j)(y^i)
	\end{pmatrix}\]
	has full rank. This shows that $\left\{ x^1,\dots ,x^n \right\}$ is a coordinate system near $\mathbf{v}$, concluding the proof of \autoref{thm:normalExponentialLocalForm:1} and \autoref{thm:normalExponentialLocalForm:2}.

	Now, suppose the hypothesis of \autoref{thm:normalExponentialLocalForm:3} holds. Thus, $\mathbf{v} \in \mathcal{F}^{\textrm{reg}}$ is such that there is a neighborhood $\mathbf{v} \in U \subset \mathcal{C}$ for which $\mathcal{K}_{\mathbf{u}} \cap T_{\mathbf{u}} \mathcal{C} = 0$ for $\mathbf{u} \in U$. We show that near $\mathbf{v}$, the normal exponential map $\mathcal{E}$ is a \emph{submersion with folds} (see \cite[Defintion 4.2]{GolGui73} for terminology). This amounts to checking the following two properties.
	\begin{itemize}
		\item As in the proof of \autoref{thm:regularTangentFocalPointsSubmanifold}, near $\mathbf{v}$, the codimension $1$ submanifold $\mathcal{C}$ is given as the zero of the function $\Delta$ (\autoref{eq:Delta}). For $k = 1$, this $\Delta$ is precisely the product of eigenfunctions of $d\mathcal{E}$, and thus $\Delta = \det d \mathcal{E}$. As argued in the proof of \autoref{thm:regularTangentFocalPointsSubmanifold}, the derivative of $\Delta$ is non-vanishing. This shows that the first jet $j^1 \mathcal{E}$ is transverse to the codimension $1$ submanifold $S_1 \subset J^1(\hat{\nu}, M)$ consisting of jets of rank $n - 1$.
		\item The singularities of $\mathcal{E}$ near $\mathbf{v}$ are precisely the neighborhood $U$, along which we have $T_{\mathbf{u}}\hat{\nu} = \mathcal{K}_{\mathbf{u}} + T_{\mathbf{u}} \mathcal{C} = \ker d_{\mathbf{u}}\mathcal{E} + T_{\mathbf{u}}\left( \Delta^{-1}(0) \right)$. This shows that every singularity of $\mathcal{E}$ is a fold point.
	\end{itemize}
	Then, the normal form \autoref{eq:normalForm:3} for $\mathcal{E}$ follows immediately from \cite[Theorem 4.5]{GolGui73}, finishing the proof of \autoref{thm:normalExponentialLocalForm:3}.
\end{proof}

\section*{Acknowledgement} The authors would like to thank Prof. J. Itoh for many fruitful discussions, and also to the anonymous referee for several suggestions and improvements. They would also like to thank the International Centre for Theoretical Sciences (ICTS) for the wonderful hosting and a research-friendly environment during the discussion meeting Zero Mean Curvature Surfaces (code: ICTS/zmcs2024/09), where the bulk of this work was done. The first author was supported by the NBHM grant no. 0204/1(5)/2022/R\&D-II/5649, and the second author was supported by the Jilin University. 

\bibliographystyle{alphaurl}
\bibliography{ref}

\end{document}